%% file: main.tex
\input{preamble}

\begin{document}

\begin{abstract}
  In this paper we define a novel product on polyominoes called the \textit{tensor product of polyominoes}. Using a special case of this product which we name a \textit{dilation}, we construct a new class of polyominoes called the \textit{dilated closed paths}. This class is a ``thick'' generalization of Cisto and Navarra's class of closed paths. We prove the Zig-Zag Walk Conjecture for this new class thereby fully characterizing primality for its associated polyomino ideals.
\end{abstract}

\maketitle

\section{Introduction}

In a 2012 paper by Qureshi, a construction was presented to attach a binomial ideal to any given polyomino, which Qureshi called the \textit{polyomino ideal} (\textbf{\cite{Qureshi:2012}}). This new class of ideal was introduced as a generalization of the class of binomial ideals generated by 2-minors of an $m \times n$ matrix of indeterminates that is well-known in commutative algebra (\textbf{\cite{Hochster_Eagon:1971}}). Many papers in the past decade have studied algebraic invariants of polyominoes determined by either the polyomino ideal or by its corresponding coordinate ring. One algebraic invariant that has been focused on is the primality of the polyomino ideal. There is yet no complete description of the primality of the polyomino ideal, however, a complete description of primality has been given for certain classes of polyominoes and more general collections of cells such as simple polyominoes (\textbf{\cite{Qureshi_Shibuta_Shikama:2017}}), rectangle minus convex polyominoes (\textbf{\cite{Hibi_Qureshi:2015}}, \textbf{\cite{Shikama:2018}}), grid polyominoes (\textbf{\cite{Mascia_Rinaldo_Romeo:2020}}), closed paths (\textbf{\cite{Cisto_Navarra:2023}}), weakly closed paths (\textbf{\cite{Cisto_Navarra_Utano:2022}}), and most recently for polyominoes containing a good rectangle piece and for extended bipartite collections of cells (\textbf{\cite{Zheng_Guo_Wu:2026i}}, \textbf{\cite{Zheng_Guo_Wu:2026ii}}). In \textbf{\cite{Mascia_Rinaldo_Romeo:2020}}, the authors defined a specific sequence of inner intervals on a polyomino meeting certain geometric criteria called a \textit{zig-zag walk}. In that same paper, they showed that the non-existence of a zig-zag walk in any polyomino is a necessary condition for its associated polyomino ideal to be prime and conjectured it to be a sufficient condition. This conjecture has been named the ``Zig-Zag Walk Conjecture''. Our main result will show that the Zig-Zag Walk Conjecture holds for a novel class of polyominoes which we introduce and name the class of \textit{dilated closed paths}. This class is a ``thick'' generalization of the class of closed paths that Cisto and Navarra defined in \textbf{\cite{Cisto_Navarra:2023}} and for which they proved the Zig-Zag Walk Conjecture. In order to construct this generalization of closed paths, we introduce a novel product on polyominoes which we call the \textit{tensor product of polyominoes}.

In Section \ref{Section:2}, we give some preliminaries on polyominoes, the polyomino ideal, and we state the Zig-Zag Walk Conjecture as first given in \textbf{\cite{Mascia_Rinaldo_Romeo:2020}}. In Section \ref{Section:3} we introduce a novel product on polyominoes which we name the \textit{tensor product of polyominoes}. We single out a certain tensor product of polyominoes which we call a \textit{dilation}, which is a stretching of a polyomino by integer factors in one or both coordinate directions, and show that it preserves many important polyomino invariants. We end this section by defining our novel class of \textit{dilated closed paths}. Then, in Section \ref{Section:4} we build the necessary steps on the logical staircase to be able to state our main result, Theorem \ref{Thm: ZZW for dilated closed paths}. This result fully characterizes the primality of dilated closed paths by showing that the Zig-Zag Walk Conjecture holds for this class of polyominoes. 

As a note on the notation to follow, we fix $\mathbb{N} = \{0,1,2,...\}$ and $\mathbb{Z}^+ = \{1,2,3,...\}$.

\section{Preliminaries} \label{Section:2}

\subsection{Polyominoes}

Let $a,b\in\mathbb{N}^2$ and $\leq$ be the product order on $\mathbb{N}^2$, that is, if $a = (i,j)$ and $b = (k,\ell)$, then $a \leq b$ if and only if $i\leq k$ and $j\leq \ell$. If $a\leq b$, then $[a,b]=\{c\in\mathbb{N}^2 \mid a\leq c \leq b\}$ is called an \textit{interval} of $\mathbb{N}^2$. Furthermore, if $a<b$, then $[a,b]$ is called a \textit{proper interval}. Let $I=[a,b]$ be a proper interval of $\mathbb{N}^2$, so $a=(i,j)$ and $b=(k,\ell)$ for some $i<k$ and $j<\ell$. We call $a$ and $b$, respectively, the \textit{lower left corner} and \textit{upper right corner} of $I$. Together, $a$ and $b$ are the \textit{diagonal corners} of $I$, with $c=(i,\ell)$ and $d=(k,j)$ being the \textit{anti-diagonal corners} of $I$. If $C$ is an interval of $\mathbb{N}^2$ with lower left corner $a$ such that $C = [a, a+(1,1)]$, then $C$ is called a \textit{cell}. Given a cell $C$ with diagonal corners $a$ and $b$ and anti-diagonal corners $c$ and $d$, we say that $V(C) =\{a,b,c,d\}$ is the set of \textit{vertices} of $C$ and that $E(C)=\{\{a,c\},\{a,d\},\{b,c\},\{b,d\}\}$ is the set of \textit{edges} of $C$.

We now shift our attention to collections of cells in $\mathbb{N}^2$. Let $C$ and $D$ be distinct cells in $\mathbb{N}^2$, then a \textit{walk of length n} from $C$ to $D$ is a sequence of cells $\mathcal{W}: C=C_1, C_2,...,C_n=D$ such that $E(C_i)\cap E(C_{i+1}) \neq \emptyset$ for all $1 \leq i \leq n-1$. If it is the case that $C_i\neq C_j$ for any $i,j\in\{1,..,n\}$, then we call $\mathcal{W}$ a \textit{path} from $C$ to $D$. Let $\mathcal{P}$ be a finite collection of cells in $\mathbb{N}^2$ and let C and D be two cells of $\mathcal{P}$. We say that $C$ and $D$ are \textit{connected} in $\mathcal{P}$ if there exists a walk $\mathcal{W}: C=C_1,C_2,...,C_{n-1},C_n=D$ such that $C_i\in\mathcal{P}$ for all $2 \leq i \leq n-1$. We say that $\mathcal{P}$ is a \textit{polyomino} if every pair of distinct cells in $\mathcal{P}$ is connected. We denote by $|\mathcal{P}|$ the \textit{rank} of $\mathcal{P}$, that is, the number of cells in $\mathcal{P}$.

Note that the cells of polyominoes are always connected edge to edge. If we relax this condition to allow two cells to be only connected vertex to vertex we say that that such a collection of cells, say $\mathcal{Q}$, is \textit{weakly connected}, meaning that for any cells $C,D \in \mathcal{Q}$ there exists a sequence of cells ${\mathcal{C}: C=C_1, C_2,...,C_n=D}$ such that $V(C_i)\cap V(C_{i+1}) \neq \emptyset$ for all $1 \leq i \leq n-1$. Such a collection of cells $\mathcal{Q}$ is often called a \textit{polyking} as it can be viewed as a representation of kingwise movement on a chessboard. A polyomino, on the other hand, can be viewed as a representation of rookwise movement on a chessboard.

By its construction, the class of all polykings contains all polyominoes. We also note that any polyking can be thought as a finite collection of polyominoes. Let $\mathcal{Q}$ be a polyking and let $\mathcal{Q'}$ be a subset of cells of $\mathcal{Q}$. Then $\mathcal{Q'}$ is called a \textit{connected component} of $\mathcal{Q}$ if
$\mathcal{Q'}$ is a polyomino and for any $C\in \mathcal{Q} \setminus \mathcal{Q'}$, we have that $\mathcal{Q'} \cup C$ is not a polyomino, i.e. $\mathcal{Q'}$ is maximal in regard to set inclusion. More specifically, for any polyomino $\mathcal{P}$, if $\mathcal{P'}$ is a subset of cells of $\mathcal{P}$ and $\mathcal{P'}$ is a polyomino, then we call $\mathcal{P'}$ a \textit{subpolyomino} of $\mathcal{P}$, and we notate this containment as $\mathcal{P'} \subseteq \mathcal{P}$.

 For $\mathcal{P}$ a polyomino, we call $\mathcal{P}$ \textit{simple} if for any cells $C$, $D \notin \mathcal{P}$, there exists a walk ${\mathcal{W}: C=C_1, C_2,..., C_n=D}$ such that $C_i \notin \mathcal{P}$ for any $1 \leq i \leq n$. We call $\mathcal{P}$ \textit{non-simple} if $\mathcal{P}$ is not simple. Let $\mathcal{P}$ be a polyomino. We say that a finite collection of cells $\mathcal{H}$ is a \textit{hole} of $\mathcal{P}$ if for any $F \in \mathcal{H}$, we have that $F \notin \mathcal{P}$ and for any other cell $G \in \mathcal{H}$, $F$ and $G$ are connected and $\mathcal{H}$ is maximal with respect to set inclusion.

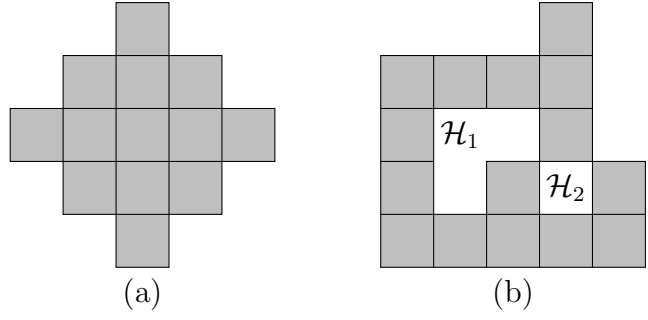
\begin{figure}[h]
    \centering
    \begin{tikzpicture} [scale = 0.7]
        \draw[step=1cm,white,very thin] (0,0) grid (12,6);
        \polyomino[
        empty cell =x,
        grid,
        p={a}{style={lightgray,draw=black}},
        row sep =;
        ]{
        x x a x x x x x x x a x;
        x a a a x x x a a a a x;
        a a a a a x x a x x a x;
        x a a a x x x a x a x a;
        x x a x x x x a a a a a;
        x x x x x x x x x x x x 
        }

        \begin{scope}[color=black]
            \node[anchor=center] () at (8.5,3.5){$\mathcal{H}_1$};
            \node[anchor=center] () at (10.5,2.5){$\mathcal{H}_2$};
            \node[anchor=center] () at (2.5,0.5){(a)};
            \node[anchor=center] () at (9.5,0.5){(b)};
        \end{scope}
        
    \end{tikzpicture}
    \caption{(a) An example of a simple polyomino (b) An example of a non-simple polyomino with two holes $\mathcal{H}_1$ and $\mathcal{H}_2$}
    \label{fig:non-simple polyo}
\end{figure}

Note by its definition, any hole of a polyomino can be thought of as its own simple polyomino, meaning that we cannot have a hole inside of another hole. We also require any hole $\mathcal{H}$ of a polyomino $\mathcal{P}$ to be finite because in general there are an infinite amount of cells of $\mathbb{N}^2$ that are not in $\mathcal{P}$. We define the \textit{exterior} of $\mathcal{P}$ to be the set {$\ext \mathcal{P} = \{C \subset \mathbb{N}^2 \mid C \text{ is a cell, } C \notin \mathcal{P} \cup \mathcal{H}_1 \cup ... \cup \mathcal{H}_n\}$} for holes $\mathcal{H}_1,...,\mathcal{H}_n$ of $\mathcal{P}$. Correspondingly, we define the \textit{interior} of a polyomino $\mathcal{P}$ to be the set $\intpolyo (\mathcal{P}) = \{C \subset \mathbb{N}^2 \mid C \text{ is a cell and } C \notin \ext (\mathcal{P})\}$. In addition to $\ext \mathcal{P}$ and $\intpolyo \mathcal{P}$, we now define some other useful sets associated with polyominoes. Let $\mathcal{P}$ be a polyomino of rank $n$, and let $C_1,..., C_n$ be the cells of $\mathcal{P}$. We define $V(\mathcal{P}) = \bigcup\limits_{i=1}^{n} V(C_i)$ to be the set of vertices of $\mathcal{P}$ and $E(\mathcal{P}) = \bigcup\limits_{i=1}^{n} E(C_i)$ to be the set of edges of $\mathcal{P}$. Note as any hole $\mathcal{H}$ of a polyomino $\mathcal{P}$ can be thought of as a simple polyomino, then we can extend the notation above to describe the sets $V(\mathcal{H})$ and $E(\mathcal{H})$.

Let $A, B \subset \mathbb{N}^2$ be two cells with lower left corners $a = (i,j)$ and $b = (k,l)$ respectively, such that $a \leq b$. We define the \textit{cell interval} $[A,B]$ to be the set of all cells $C \subset \mathbb{N}^2$ such that if $c$ is the lower left corner of $C$ then we have that $a \leq c \leq b$. The \textit{width} of $[A,B]$ is defined to be $k-i+1$ and the \textit{height} of $[A,B]$ is defined to be $l-j+1$, that is to say that they are measured in cells. We say that the cells A and B are in \textit{vertical position} if $i = k$ and in \textit{horizontal position} if $j = l$. By its definition, a cell interval is a polyomino. In this paper we will refer to the class of all cell intervals as the class of \textit{rectangle polyominoes} which we denote by $\mathcal{R}$. We use $\mathcal{R}_{m,n}$ to denote the unique rectangle polyomino with width $m$ and height $n$. Let $\mathcal{P}$ be a polyomino and let $A,B \in \mathcal{P}$ such that $[A,B]$ is a cell interval of $n$ cells in vertical or horizontal position. We call $[A,B]$ a \textit{block of $\mathcal{P}$ of length n} if $C \in \mathcal{P}$ for all cells $C \in [A,B]$. We say that a block $\mathcal{B}$ of $\mathcal{P}$ is \textit{maximal} if there does not exist another block $\mathcal{B}'$ of $\mathcal{P}$ such that $\mathcal{B} \subset \mathcal{B}'$. Let $\mathcal{P}$ a polyomino and let $a, b \in V(\mathcal{P})$ with $a = (i,j)$ and $b = (k,l)$. We say that $[a,b]$ is a \textit{vertical edge interval} of $\mathcal{P}$ if $i = k$ and if the edge {$\{a+(0,i),a+(0,i+1)\}\in E(\mathcal{P})$} for all $i\in\{0,...,l-1\}$. Similarly, we call $[a,b]$ a \textit{horizontal edge interval} of $\mathcal{P}$ if $j = l$ and {$\{a+(i,0),a+(i+1,0)\}\in E(\mathcal{P})$ for all $i\in\{0,...,k-1\}$}. We call a vertical edge interval $[a,b]$ \textit{maximal} if $\{a - (0,1),a\},\{b, b + (0,1)\} \notin E(\mathcal{P})$. Similarly, we call a horizontal edge interval $[a,b]$ \textit{maximal} if $\{a - (1,0),a\},\{b, b + (1,0)\} \notin E(\mathcal{P})$.

\begin{figure}[h]
    \centering
    \begin{tikzpicture} [scale = 0.7]
        \draw[step=1cm,white,very thin] (0,0) grid (5,5);
        \polyomino[
        empty cell =x,
        grid,
        p={a}{style={lightgray,draw=black}},
        row sep =;
        ]{
        x x a x x;
        x a a a x;
        a a x a a;
        x x x x x; 
        x x a x x
        }
        \polyomino[
        empty cell =x,
        grid,
        p={a}{style={gray,draw=black}},
        row sep =;
        ]{
        x x x x x;
        x x x x x;
        x x x x x;
        x a a a x; 
        x x x x x
        }
        
        \draw [red, ultra thick] (1,3)--(4,3);
        \draw [blue, ultra thick] (4,1)--(4,4);
        
        \fill(1,3) circle[radius=2pt] node[above left]{$a$};
        \fill(4,3) circle[radius=2pt] node[above right]{$b$};
        \fill(4,1) circle[radius=2pt] node[below right]{$u$};
        \fill(4,4) circle[radius=2pt] node[above right]{$v$};

        \begin{scope}[color=black]
            \node[anchor=center] () at (2.5,1.5){$\mathcal{B}$};
        \end{scope}
        
    \end{tikzpicture}
    \caption{A polyomino with a maximal block $\mathcal{B}$ of length 3, a horizontal edge interval $[a,b]$ in red, and a maximal vertical edge interval $[u,v]$ in blue}
    \label{fig:blocks and edge intervals}
\end{figure}
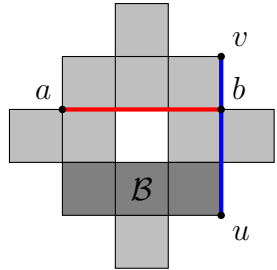

One large class of polyominoes that has been studied is the class of \textit{thin polyominoes}. We say that a polyomino $\mathcal{P}$ is a \textit{thin polyomino} if $\mathcal{R}_{2,2} \nsubseteq \mathcal{P}$, that is, it does not contain the square tetromino as subpolyomino. One subclass of thin polyominoes is the class of grid polyominoes introduced in \textbf{\cite{Mascia_Rinaldo_Romeo:2020}}. Another subclass of thin polyominoes is the class of \textit{closed paths} which were introduced in \textbf{\cite{Cisto_Navarra:2023}} and independently in \textbf{\cite{Mascia_Rinaldo_Romeo:2022}} as \textit{thin cycles}. As we will be investigating a generalization of this class of polyominoes in Section \ref{Section:3}, we fix our terminology now and call these polyominoes closed paths. We give their formal definition as was given in Definition 3.1 of \textbf{\cite{Cisto_Navarra:2023}}.

\begin{definition} \label{Defn: closed path}
    Let $\mathcal{P}$ be a polyomino. We call $\mathcal{P}$ a \textit{closed path} if it is a sequence of cells $A_1,...,A_n,A_{n+1}$, with $n > 5$, such that:
    \begin{itemize}
        \item[(i)] $A_1 = A_{n+1}$,
        \item[(ii)] for all $i\in \{1,...,n\}$, $A_i \cap A_{i+1} = e$ for some common edge $e\in E(A_i) \cap E(A_{i+1})$,
        \item[(iii)] $A_i \neq A_j$ for all $i,j\in\{1,..,n\}$ with $i \neq j$,
        \item[(iv)] for all $i \in \{1,...,n\}$ and for all $j \notin \{i-2, i-1, i, i+1, i+2\}$, $A_i \cap A_j = \emptyset$, given that we define $A_{-2} = A_{n-2}, A_{-1} = A_{n-1}, A_0 = A_n$, and $A_{n+2} = A_2$.
    \end{itemize}
\end{definition}

\begin{figure} [H]
    \centering
    \begin{tikzpicture} [scale = 0.7]
        \draw[step=1cm,white,very thin] (0,0) grid (7,6);
        \polyomino[
        empty cell =x,
        grid,
        p={a}{style={lightgray,draw=black}},
        p={b}{style={lightgray,draw=black}},
        row sep =;
        ]{
        x a a a a x x;
        a a x x a a a;
        a x x x x x a;
        a a x x x a a;
        x a a x x a x;
        x x a a b a x
        }
        
    \end{tikzpicture}
    \caption{An example of a closed path}
    \label{fig:closed path}
\end{figure}
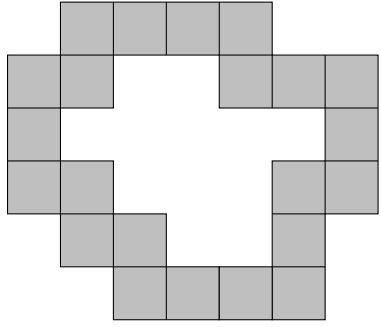

\subsection{The polyomino ideal} \label{Subsection:2.2}

We now give Qureshi's construction from \textbf{\cite{Qureshi:2012}} that attaches a binomial ideal to a polyomino. Let $\mathcal{P}$ be a polyomino and fix $K$ to be a field. We define the ring $S$ to be the polynomial ring $K[x_a \mid a \in V(\mathcal{P})]$. Let $a,b \in V(\mathcal{P})$ with $a < b$. We call the proper interval $[a,b]$ an \textit{inner interval} if the polyomino attached to $[a,b]$ is a subpolyomino of $\mathcal{P}$. For any proper interval $[a,b] \subset \mathbb{N}^2$ with anti-diagonal corners $c,d \in \mathbb{N}^2$, we can associate with it the binomial $x_ax_b-x_cx_d$.

\begin{definition}
    Let $\mathcal{P}$ be a polyomino, let $[a,b]$ be an inner interval of $\mathcal{P}$ and let $c,d \in V(\mathcal{P})$ be the anti-diagonal corners of $[a,b]$. We say that the binomial $x_ax_b-x_cx_d \in S$ is an \textit{inner 2-minor} of $\mathcal{P}$ and we call the ideal $I_{\mathcal{P}} \subset S$ generated by all the inner 2-minors of $\mathcal{P}$ the \textit{polyomino ideal of $\mathcal{P}$}.
\end{definition}

We note that the polyomino ideal of $\mathcal{P}$ naturally induces a quotient ring $K[\mathcal{P}] = S/I_{\mathcal{P}}$ which we will call the \textit{coordinate ring of $\mathcal{P}$}.

\subsection{The toric ideal of a polyomino and prime polyominoes} \label{Subsection:2.3} We now define a certain toric ideal $J_{\mathcal{P}}$ that we can associate with any polyomino $\mathcal{P}$ in the specific way that it contains the polyomino ideal, that is $I_\mathcal{P} \subseteq J_{\mathcal{P}}$. As any toric ideal is prime by its definition (see Theorem 1.6 in \textbf{\cite{Bigatti_Robbiano:2001}}), we can assert the primality of any polyomino ideal $I_\mathcal{P}$ by showing that its containment in $J_{\mathcal{P}}$ is actually an equality, that is $I_\mathcal{P}=J_{\mathcal{P}}$. Though Qureshi first described such a toric ideal in \textbf{\cite{Qureshi:2012}} as a graph-theoretic ideal of an edge ring associated to a polyomino, we follow Mascia, Rinaldo, and Romeo's alternate construction of such a toric ideal as defined in \textbf{\cite{Mascia_Rinaldo_Romeo:2020}}.

Let $\mathcal{P}$ be a polyomino and fix some field $K$. Let $S =K[x_a \mid a \in V(\mathcal{P})]$ and let $I_{\mathcal{P}} \subset S$ be the polyomino ideal of $\mathcal{P}$. We let $a,b \in V(\mathcal{P})$, with $a = (i,j)$ and $b = (k,l)$, and we consider the total order $\prec$ on $V(\mathcal{P})$ given by $a \prec b$ if $i < k$ or if $i = k$ and $j < l$. For any hole $\mathcal{H}$ of $\mathcal{P}$ we define $e\in V(\mathcal{H})$ to be the \textit{lower left corner} of $\mathcal{H}$ if $e \prec f$ for all $f \in V(\mathcal{H})\setminus\{e\}$. Let $\mathcal{H}_1,...,\mathcal{H}_r$ be all the holes of $\mathcal{P}$ and denote by $e_k = (i_k,j_k)$ the lower left corner of the hole $\mathcal{H}_k$ for $k\in\ L = \{1,...,r\}$. We define the set $W_k = \{a\in V(\mathcal{P}) \mid a\leq e_k\}$ where $\leq$ is the product order on $\mathbb{N}^2$ as defined previously. Let $\{H_i\}_{i\in I}$ and $\{V_j\}_{j\in J}$ be the set of all maximal horizontal edge intervals and the set of all maximal vertical edge intervals of $\mathcal{P}$, respectively, for some index sets $I$ and $J$. Let $\{h_i\}_{i\in I}$,$\{v_j\}_{j\in J}$, and $\{w_k\}_{k\in L}$ be three sets of variables associated to the sets $\{H_i\}_{i\in I}$, $\{V_j\}_{j\in J}$, and $\{W_k\}_{k\in L}$, respectively.
We define the map:
\begin{eqnarray*}
    \alpha: V(\mathcal{P}) & \rightarrow & K[\{h_i,v_j,w_k\}\mid i\in I, j\in J,k\in L] \\
            a & \mapsto & \prod_{a\in H_i\cap V_j}h_iv_j \prod_{a\in W_k}w_k
\end{eqnarray*}
We define $T_{\mathcal{P}}$ to be the \textit{toric ring of $\mathcal{P}$}, where 
\[T_{\mathcal{P}} = K[\alpha(a) \mid a\in V(\mathcal{P})] \subset K[\{h_i,v_j,w_k\}\mid i\in I, j\in J,k\in L]
\]

Let $\varphi$ be the ring homomorphism defined by:
\begin{eqnarray*}
    \varphi: S & \rightarrow & T_{\mathcal{P}}\\
             x_a & \mapsto & \alpha(a)
\end{eqnarray*}
Note that $\varphi$ is surjective. We define $J_{\mathcal{P}} = \ker \varphi$ to be the \textit{toric ideal of $\mathcal{P}$}.

\begin{definition}
   Let $\mathcal{P}$ be a polyomino, and $I_{\mathcal{P}}$ be its polyomino ideal. We say that $\mathcal{P}$ is \textit{prime} if $I_{\mathcal{P}}$ is prime and we say that $\mathcal{P}$ is \textit{non-prime} if $I_{\mathcal{P}}$ is non-prime. 
\end{definition}

By the above definitions, we can now give the following remark.

\begin{remark}
   If $I_\mathcal{P} = J_\mathcal{P}$ then $I_\mathcal{P}$ is prime. 
\end{remark}

A pivotal result in the characterization of primality for polyominoes was given by Qureshi, Shibuta, and Shikama in \textbf{\cite{Qureshi_Shibuta_Shikama:2017}}, in which they showed that simple polyominoes are prime. As we will be referencing this result for the remainder of the thesis, we state it now.

\begin{theorem}[\textbf{\cite{Qureshi_Shibuta_Shikama:2017}}, Theorem 2.2] \label{Thm:simple polyos are prime}
    Simple polyominoes are prime.
\end{theorem}

\subsection{Primality of non-simple polyominoes and the Zig-Zag Walk Conjecture} \label{Subsection:2.4}

As all simple polyominoes are prime, what is left to resolve is the characterization of primality for non-simple polyominoes. Though only non-simple polyominoes can have non-prime ideals, not all non-simple polyominoes are non-prime. In this section we introduce the concept of a \textit{zig-zag walk}, which was first introduced in \textbf{\cite{Mascia_Rinaldo_Romeo:2020}}, where in Corollary 3.6, the authors showed that the existence of a zig-zag walk in a polyomino is a sufficient condition for its polyomino ideal to be non-prime, and hence, the non-existence of a zig-zag walk in a polyomino is a necessary condition for its polyomino ideal to be prime.

\begin{definition} \label{Defn:ZZW}
    Let $\mathcal{P}$ be a polyomino and let $\mathcal{W}: I_1,...,I_n$ be a sequence of distinct inner intervals of $\mathcal{P}$ such that for each interval $I_i$, $i \in \{1,...,n\}$, either $v_i$ and $z_i$ are its diagonal corners and $u_i$ and $v_{i+1}$ are its anti-diagonal corners or vice versa. We call $\mathcal{W}$ a \textit{zig-zag walk} of $\mathcal{P}$ if:
        \begin{itemize}
            \item[(i)] $I_1 \cap I_{n} = \{v_1\} = \{v_{n+1}\}$ and $I_i \cap I_{i+1} = \{v_{i+1}\}$ for any $i \in \{1,...,n-1\}$,
            \item[(ii)] $v_i$ and $v_{i+1}$ are on the same edge interval of $\mathcal{P}$ for all $i \in \{1,...,n\}$,
            \item[(iii)] for all $i,j \in \{1,...,n\}$, with $i \neq j$, there exists no inner interval $J$ of $\mathcal{P}$ such that $z_i,z_j \in J$.
        \end{itemize}
\end{definition}

\begin{figure} [h]
    \centering
    \begin{tikzpicture}
        \draw[step=1cm,white,very thin] (0,0) grid (6,7);
        \polyomino[
        empty cell =x,
        grid,
        p={a}{style={lightgray,draw=black}},
        p={b}{style={gray,draw=black}},
        row sep =;
        ]{
        x x a a x x;
        x a b b b x;
        b b x x a b;
        b b x x a b;
        x a b b b x;
        x x b b b x;
        x x x a x x
        }

        \fill(2,3) circle[radius=2pt] node[above right]{$v_1$};
        \fill(2,5) circle[radius=2pt] node[below right]{$v_2$};
        \fill(5,5) circle[radius=2pt] node[below left]{$v_3$};
        \fill(5,3) circle[radius=2pt] node[above left]{$v_4$};

        \fill(0,5) circle[radius=2pt] node[above left]{$z_1$};
        \fill(5,6) circle[radius=2pt] node[above right]{$z_2$};
        \fill(6,3) circle[radius=2pt] node[below right]{$z_3$};
        \fill(2,1) circle[radius=2pt] node[below left]{$z_4$};

        \fill(0,3) circle[radius=2pt] node[below left]{$u_1$};
        \fill(2,6) circle[radius=2pt] node[above left]{$u_2$};
        \fill(6,5) circle[radius=2pt] node[above right]{$u_3$};
        \fill(5,1) circle[radius=2pt] node[below right]{$u_4$};

         \begin{scope}[color=black]
            \node[anchor=center] () at (0.5,3.5){$I_1$};
            \node[anchor=center] () at (3.5,5.5){$I_2$};
            \node[anchor=center] () at (5.5,3.5){$I_3$};
            \node[anchor=center] () at (3.5,2.5){$I_4$};
            \node[anchor=center] () at (0.5,6.5){$\mathcal{P}$};
        \end{scope}
        
    \end{tikzpicture}
    \caption{An example of a zig-zag walk $\mathcal{W} = I_1,I_2,I_3,I_4$ of $\mathcal{P}$}
    \label{fig:zig-zag walk}
\end{figure}
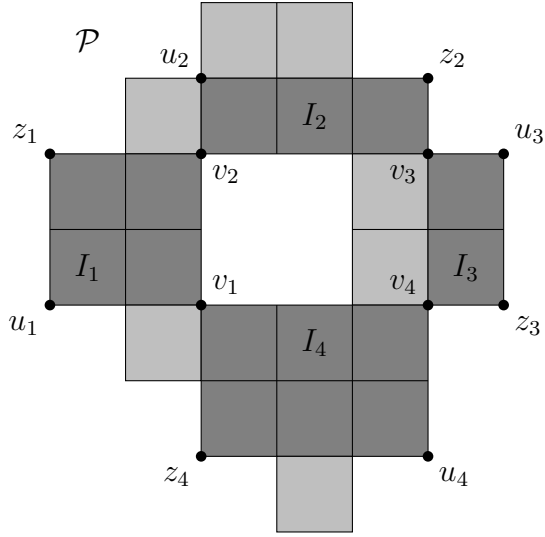

The reason that the authors chose the name ``zig-zag'' is because if $v_i$ is a diagonal corner of $I_i$ then, by the above definition, $v_{i+1}$ must be an anti-diagonal of the next inner interval $I_{i+1}$ (\textbf{\cite{Mascia_Rinaldo_Romeo:2020}}, Remark 3.3). Thus, this alternating pattern of the placement of each $v_i$ and $z_i$ in adjacent inner intervals makes a ``zig-zag'' pattern. We note also that for any zig-zag walk $\mathcal{W}: I_1,...,I_n$, we must have that $n$ is even.

\begin{conjecture}[\textbf{Zig-Zag Walk Conjecture}] \label{Conj:ZZW}
    Let $\mathcal{P}$ be a polyomino. Then $I_{\mathcal{P}}$ is prime if and only if $\mathcal{P}$ contains no zig-zag walks.
\end{conjecture}

The forward direction of this conjecture was proven as Corollary 3.6 in \textbf{\cite{Mascia_Rinaldo_Romeo:2020}}, but its converse, which would give a sufficient condition for the primality of a polyomino, remains an open problem. The conjecture has, however, been proved for grid polyominoes (\textbf{\cite{Mascia_Rinaldo_Romeo:2020}}), closed path polyominoes (\textbf{\cite{Cisto_Navarra:2023}}), and weakly closed path polyominoes (\textbf{\cite{Cisto_Navarra_Utano:2022}}). As well, in \textbf{\cite{Mascia_Rinaldo_Romeo:2020}}, the authors verified the conjecture computationally using \emph{Macaulay2} for all polyominoes with small enough rank. We give their statement below.

\begin{theorem}[\textbf{\cite{Mascia_Rinaldo_Romeo:2020}},Theorem 3.9] \label{Thm: zigzag conj for under 15}
    Let $\mathcal{P}$ be a polyomino with $|\mathcal{P}| \leq 14$. Then $I_{\mathcal{P}}$ is prime if and only if $\mathcal{P}$ contains no zig-zag walks.
\end{theorem}

The full description of the algorithm used to prove the above theorem can be found along with its source code in \textbf{\cite{Mascia_Rinaldo_Romeo_sourcecode}}. Since we will be constructing a generalization of closed paths in the next chapter, we state now for future reference the result of the Zig-Zag Walk Conjecture holding for closed paths, which was given as Theorem 6.2 in \textbf{\cite{Cisto_Navarra:2023}}.

\begin{theorem} [\textbf{\cite{Cisto_Navarra:2023}}, Theorem 6.2] \label{Thm: ZZWC for closed paths}
    Let $\mathcal{P}$ be a closed path, then $I_\mathcal{P}$ is prime if and only if $\mathcal{P}$ contains no zig-zag walks.
\end{theorem}

\section{Dilated closed paths} \label{Section:3}

Our goal in this section is to construct a novel class of polyominoes which we name the \textit{dilated closed paths}. This new class will generalize the class of closed paths (see Definition \ref{Defn: closed path}). We recognize here that the class of closed paths has already been generalized in \textbf{\cite{Cisto_Navarra_Utano:2022}} in which the authors investigated weakly closed paths, which were constructed through a certain localized change to the geometry of a closed path. In our generalization, we will instead describe a certain global change to the geometry of a closed path that we will call a \textit{dilation}. In order to do this, we first define a novel product on polyominoes.

We use a modification of Battaglino's construction in \textbf{\cite{Battaglino:2014}} to attach a binary matrix to any polyomino. Let $\mathcal{P}$ be a polyomino and let $s=(k,\ell) \in \mathbb{N}^2$ such that {$s = \sup V(\mathcal{P})$}, that is, $s$ is the supremum of $V(\mathcal{P})$. We define the matrix $M_{\mathcal{P}}\in \mathcal{M}_{\ell \times k}(\mathbb{Z}_2)$ by:

    \[[M_{\mathcal{P}}]_{i,j} = 
        \begin{cases} 
        1 \text{ if $[(j-1,\ell-i),(j,\ell-i+1)]$ is a cell of $\mathcal{P}$} \\
        0 \text{ otherwise}
        \end{cases} \]

We say that $M_{\mathcal{P}}$ is the \textit{binary matrix associated with $\mathcal{P}$}.

\begin{figure} [h] 
    \centering   
    \begin{tikzpicture}
        \draw[step=1cm,black,very thin, dashed] (0,0) grid (4.9,5.49);
        \draw[->, ultra thick](0,-0.5)--(0,5.5);
        \draw[->, ultra thick](-0.5,0)--(5,0);
    
        \polyomino[
        empty cell =x,
        grid,
        p={a}{style={lightgray,draw=black}},
        row sep =;
        ]{
        x x a x x x x x;
        x a a a x x x x;
        x a x a x x x x;
        x x a a x x x x;
        x x a x x x x x
        }
        \begin{scope}[color=black]
            \node[anchor=center] () at (7.5,2.5){ 
            \huge {$\begin{bmatrix}
            0 & 0 & 1 & 0\\
            0 & 1 & 1 & 1\\
            0 & 1 & 0 & 1\\
            0 & 0 & 1 & 1\\
            0 & 0 & 1 & 0\\
              \end{bmatrix}$}};
        \end{scope}

        \fill(1,0) circle[radius=2pt] node[above left]{\small $r$};
        \fill(4,5) circle[radius=2pt] node[above right]{\small $s$};
        \fill(1,4) circle[radius=0pt] node[above left]{$\mathcal{P}$};

        \fill(0,0) circle[radius=0pt] node[below left]{\small 0};
        \fill(1,0) circle[radius=0pt] node[below]{\small 1};
        \fill(2,0) circle[radius=0pt] node[below]{\small 2};
        \fill(3,0) circle[radius=0pt] node[below]{\small 3};
        \fill(4,0) circle[radius=0pt] node[below]{\small 4};

        \fill(0,1) circle[radius=0pt] node[left]{\small 1};
        \fill(0,2) circle[radius=0pt] node[left]{\small 2};
        \fill(0,3) circle[radius=0pt] node[left]{\small 3};
        \fill(0,4) circle[radius=0pt] node[left]{\small 4};
        \fill(0,5) circle[radius=0pt] node[left]{\small 5};
        
    \end{tikzpicture}
    \caption{A polyomino $\mathcal{P}$ in $\mathbb{N}^2$ alongside its associated binary matrix with $r = \inf V(\mathcal{P})$ and $s = \sup V(\mathcal{P})$}
    \label{Fig:3.1}
\end{figure}
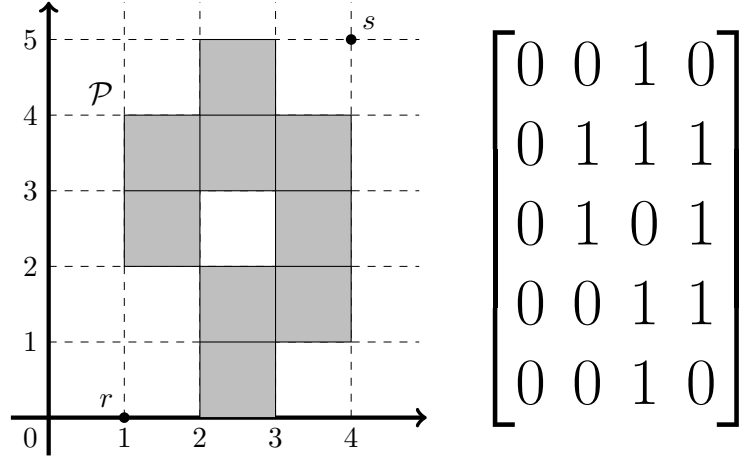

In Figure \ref{Fig:3.1}, $s = (4,5)$ is the supremum of $V(\mathcal{P})$. For any polyomino $\mathcal{P}$, we can correspondingly define the infimum of $V(\mathcal{P})$; in Figure \ref{Fig:3.1}, we have $\inf V(\mathcal{P}) = (1,0)$. Note that for a given polyomino $\mathcal{P}$, its binary matrix $M_{\mathcal{P}}$ cannot have a row (resp. column) of zeroes in its last row (resp. column). However, we allow $M_{\mathcal{P}}$ to contain rows (resp. columns) of zeroes as long as they are lower (resp. lefter) then $\inf V(\mathcal{P})$. That is, if $\inf V(\mathcal{P}) = (e,f)$ and $\sup V(\mathcal{P}) = (k, \ell)$, then for $i > \ell - f$, the $i$th row of $M_\mathcal{P}$ will be a row of zeroes and for $j \leq e $, the $j$th column of $M_\mathcal{P}$ will be a column of zeroes. We say that a polyomino $\mathcal{P}$ is in \textit{minimal position} if $\inf V(\mathcal{P}) = (0,0)$. Note that every polyomino $\mathcal{P}$ in $\mathbb{N}^2$ can be put in minimal position by translating $\mathcal{P}$ such that it lies in $\mathbb{N}^2$ in the unique way to make it contact both coordinate axes, that is, the intersection of $E(\mathcal{P})$ with each coordinate axis is nonempty. This position of $\mathcal{P}$ forces $\inf V(\mathcal{P}) = (0,0)$. Note that if a polyomino $\mathcal{P}$ is in minimal position, then the dimensions of $M_\mathcal{P}$ will be the same as the dimensions of the minimal bounding rectangle of $\mathcal{P}$ and as such it will have no rows or columns of zeroes. Given a polyomino $\mathcal{P}$, if we let $\mathcal{P}'$ denote the polyomino that represents $\mathcal{P}$ in minimal position, then $M_{\mathcal{P}'}$ is the submatrix of $M_\mathcal{P}$ obtained by deleting all rows and columns of zeroes in $M_\mathcal{P}$. Thus, we emphasize here that although we have so far been considering polyominoes in $\mathbb{N}^2$ equivalent up to translation, in this section we make a distinction between two polyominoes if one is a nontrivial translation of the other in $\mathbb{N}^2$ as they will have distinct associated binary matrices. This distinction will give us more flexibility of usage in the product we are to define below. If for two polyominoes $\mathcal{P}$ and $\mathcal{Q}$ it is the case that $M_\mathcal{P} = M_\mathcal{Q}$, then we write $\mathcal{P} = \mathcal{Q}$.

We briefly note here that one could easily extend the above definition to have binary matrices associated with polykings or, even further, with any finite collection of cells in $\mathbb{N}^2$. Therefore, we also state here that if we have two collections of cells $\mathcal{C}$ and $\mathcal{D}$ for which $M_\mathcal{C} = M_\mathcal{D}$, then we say that $\mathcal{C} = \mathcal{D}$.

We now recall the Kronecker product of two matrices. The following definition appears in \textbf{\cite{Henderson_Pukelsheim_Searle:1983}}, and we acknowledge here, as in that paper, that this product was likely first described by Georg Zehfuss. Let $A$ be an $m\times n$ matrix and $B$ be a $p\times q$ matrix. We say the \textit{Kronecker product} of $A$ and $B$, which we denote $A\otimes B$, is the following $mp \times nq$ block matrix:
\[
   A\otimes B = 
       \begin{bmatrix}
       a_{11}B & \dots & a_{1n}B\\
       \vdots  & \ddots & \vdots \\
       a_{n1}B & ... & a_{mn}B\\
       \end{bmatrix}
\]

An explicit formula for the entries of $A \otimes B$ is given by:

\[
    [A \otimes B]_{i,j} = [A]_{\left\lceil \frac{i}{p} \right\rceil, \left\lceil \frac{j}{q} \right\rceil}[B]_{((i-1) \% p)+1,((j-1) \% q)+1}
\]
Here $\%$ represents the remainder function. Modifying the above formula gives another useful formula that does not depend on the remainder function:

\[
 [A \otimes B]_{p(r-1)+u,q(s-1)+v} = [A]_{r,s}[B]_{u,v}
\]

\begin{definition} \label{Defn:tensor product of polyos}
   Let $\mathcal{P}$ and $\mathcal{Q}$ be two polyominoes with $\mathcal{Q}$ in minimal position. Let $M_{\mathcal{P}}$ and $M_{\mathcal{Q}}$ be the binary matrices associated with the $\mathcal{P}$ and $\mathcal{Q}$, respectively. We say that the unique collection of cells with associated binary matrix $M_{\mathcal{P}} \otimes M_{\mathcal{Q}}$ is the \textit{tensor product of polyominoes} $\mathcal{P}$ and $\mathcal{Q}$ and we denote this by $\mathcal{P} \otimes \mathcal{Q}$.
\end{definition}

In Figure \ref{Fig:First tensor product} we give an example of a tensor product of two polyominoes of rank 4, with both polyominoes translated to be in minimal position.

\begin{figure} [H]
    \centering
    \begin{tikzpicture} [scale = 0.8]
        \draw[step=1cm,white,very thin] (0,0) grid (8,10);
        \polyomino[
        empty cell =x,
        grid,
        p={a}{style={lightgray,draw=black}},
        row sep =;
        ]{
        x x x x x a x x;
        x a a x x a a x;
        a a x x x a x x;
        x x x x x x x x;
        x x x x a x a x;
        x x x x a a a a;
        x x x x a x a x;
        x x a x a x x x;
        x x a a a a x x;
        x x a x a x x x
        }

        \begin{scope}[color=black]
            \node[anchor=center] () at (4,8.5){\huge{$\otimes$}};
            \node[anchor=center] () at (8,8.5){\huge{$=$}};
        \end{scope}

        \fill(1,9) circle[radius=0pt] node[above left]{$\mathcal{P}$};
        \fill(5,10) circle[radius=0pt] node[above left]{$\mathcal{Q}$};
        \fill(2,3) circle[radius=0pt] node[above left]{$\mathcal{P} \otimes \mathcal{Q}$};

        \fill(0,7) circle[radius=2pt] node[below left]{$(0,0)$};
        \fill(5,7) circle[radius=2pt] node[below left]{$(0,0)$};
        \fill(2,0) circle[radius=2pt] node[below left]{$(0,0)$};
        
    \end{tikzpicture}
    \caption{The tensor product $\mathcal{P} \otimes \mathcal{Q}$ of two polyominoes $\mathcal{P}$ and $\mathcal{Q}$}
    \label{Fig:First tensor product}
\end{figure}
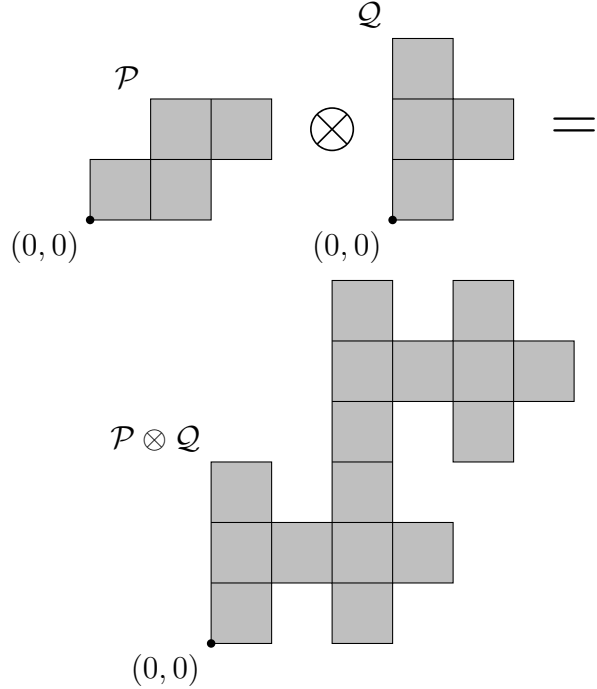

Note that in Figure \ref{Fig:First tensor product}, we have that $M_{\mathcal{P}} = 
\begin{bmatrix} 
0 & 1 & 1 \\
1 & 1 & 0 \\
\end{bmatrix}$ and $M_{\mathcal{Q}} =
\begin{bmatrix}
    1 & 0\\
    1 & 1\\
    1 & 0\\
\end{bmatrix}$ and so
\[M_{\mathcal{P} \otimes \mathcal{Q}} = M_{\mathcal{P}} \otimes M_{\mathcal{Q}} = 
\begin{bmatrix}
    0 & 0 & 1 & 0 & 1 & 0\\
    0 & 0 & 1 & 1 & 1 & 1\\
    0 & 0 & 1 & 0 & 1 & 0\\
    1 & 0 & 1 & 0 & 0 & 0\\
    1 & 1 & 1 & 1 & 0 & 0\\
    1 & 0 & 1 & 0 & 0 & 0\\
\end{bmatrix}
\]

In general, the tensor product of polyominoes is non-commutative. Additionally, the tensor product of two polyominoes does not always result in another polyomino. Though in Figure \ref{Fig:First tensor product} the tensor product $\mathcal{P}\otimes \mathcal{Q}$ results in a polyomino, with those same polyominoes, the tensor product $\mathcal{Q}\otimes \mathcal{P}$ results in a polyking as shown in Figure \ref{Fig:tensor gives polyking}.

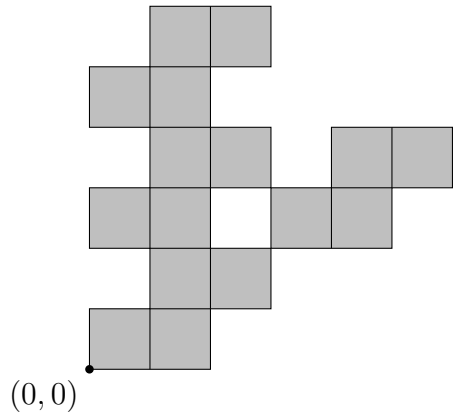
\begin{figure} [h] 
    \centering
    \begin{tikzpicture} [scale = 0.8]
        \draw[step=1cm,white,very thin] (0,0) grid (6,6);
        \polyomino[
        empty cell =x,
        grid,
        p={a}{style={lightgray,draw=black}},
        row sep =;
        ]{
        x a a x x x;
        a a x x x x;
        x a a x a a;
        a a x a a x;
        x a a x x x;
        a a x x x x
        }

        \fill(0,0) circle[radius=2pt] node[below left]{$(0,0)$};
        
    \end{tikzpicture}
    \caption{The tensor product $\mathcal{Q}\otimes \mathcal{P}$ of the polyominoes $\mathcal{P}$ and $\mathcal{Q}$ from Figure \ref{Fig:First tensor product}}
    \label{Fig:tensor gives polyking}
\end{figure}

\begin{lemma} \label{Lemma: 3.1.1}
Let $\mathcal{P}$ and $\mathcal{Q}$ be polyominoes, then $|\mathcal{P} \otimes \mathcal{Q}| = |\mathcal{P}||\mathcal{Q}|$.
\end{lemma}

\begin{proof}
Let $M_{\mathcal{P}}$ and $M_{\mathcal{Q}}$ be the binary matrices associated with $\mathcal{P}$ and $\mathcal{Q}$, respectively.
By its construction, the number of nonzero entries in $M_{\mathcal{P}}$ (resp. $M_{\mathcal{Q}}$) is precisely $|\mathcal{P}|$ (resp. $|\mathcal{Q}|$).
Let $M_{\mathcal{P} \otimes \mathcal{Q}}$ be the binary matrix associated with ${\mathcal{P} \otimes \mathcal{Q}}$. Therefore, by Definition \ref{Defn:tensor product of polyos},  $M_{\mathcal{P} \otimes \mathcal{Q}} = M_{\mathcal{P}}\otimes M_{\mathcal{Q}}$. We see that every block $m_{ij}M_{\mathcal{Q}}$ of $M_{\mathcal{P} \otimes \mathcal{Q}}$ is either of the form $0 \cdot M_{\mathcal{Q}}$ or $1 \cdot M_{\mathcal{Q}}$. Note that none of the blocks of the form $0 \cdot M_{\mathcal{Q}}$ have any nonzero entries. Therefore, the number of nonzero entries in $M_{\mathcal{P} \otimes \mathcal{Q}}$, will be equal to the number of blocks of the form $1 \cdot M_{\mathcal{Q}}$ multiplied by $|\mathcal{Q}|$, however there are precisely $|\mathcal{P}|$ such blocks. So the the number of nonzero entries in $M_{\mathcal{P} \otimes \mathcal{Q}}$ is $|\mathcal{P}||\mathcal{Q}|$, and thus $|\mathcal{P} \otimes \mathcal{Q}| = |\mathcal{P}||\mathcal{Q}|$.
\end{proof}

So intuitively the tensor product of two arbitrary polyominoes $\mathcal{P} \otimes \mathcal{Q}$ can be constructed by replacing every cell in $\mathcal{P}$ with a copy of $\mathcal{Q}$. For this reason we refer to $\mathcal{P}$ as the \textit{underlying polyomino} of the tensor product $\mathcal{P} \otimes \mathcal{Q}$. We now prove that tensor products preserve subsets of cells.

\begin{lemma} \label{Lemma:Tensor products preserve subpolyominoes}
   Let $\mathcal{P}$ and $\mathcal{Q}$ be polyominoes and let $\mathcal{P}' \subseteq \mathcal{P}$ be a subpolyomino of $\mathcal{P}$. Then the collection of cells $\mathcal{P}' \otimes \mathcal{Q}$ is contained in $\mathcal{P}\otimes \mathcal{Q}$.   
\end{lemma}

\begin{proof}
    We let $M_\mathcal{Q} \in \mathcal{M}_{n \times m}(\mathbb{Z}_2)$ for some $m,n \in \mathbb{Z}^+$. Note that as $\mathcal{P}' \subseteq \mathcal{P}$, then $\sup V(\mathcal{P}') \leq \sup V(\mathcal{P})$. We let $\sup V(\mathcal{P}') = (k', \ell')$ and $\sup V(\mathcal{P}) = (k,\ell)$ for some $k,k', \ell, \ell ' \in \mathbb{Z}^+$. So $M_\mathcal{P} \in \mathcal{M}_{\ell,k}(\mathbb{Z}_2)$ and $M_{\mathcal{P}'} \in \mathcal{M}_{\ell ',k'}(\mathbb{Z}_2)$. We can decompose $M_\mathcal{P}$ as $M_\mathcal{P} = A_{\mathcal{P}'} + \bar A_{\mathcal{P}'}$ where
    \[
        [A_{\mathcal{P}'}]_{i,j} = 
        \begin{cases} 
        1 \text{ if $[(j-1,\ell-i),(j,\ell-i+1)]$ is a cell of $\mathcal{P}'$} \\
        0 \text{ otherwise}
        \end{cases}
    \] 
    and
    \[
        [\bar A_{\mathcal{P}'}]_{i,j} = 
        \begin{cases} 
        1 \text{ if $[(j-1,\ell-i),(j,\ell-i+1)]$ is a cell of $\mathcal{P} \setminus \mathcal{P}'$} \\
        0 \text{ otherwise}
        \end{cases}
    \]
    We can now define $M_{\mathcal{P}'}$ to be some submatrix of $A_{\mathcal{P}'}$. This means that if $[M_{\mathcal{P}'}]_{i,j} = 1$, then $[A_{\mathcal{P}'}]_{i,j} = 1$ and so $[M_{\mathcal{P}}]_{i,j} = 1$. Furthermore, the contraposition of this implication gives that if $[M_{\mathcal{P}}]_{i,j} = 0$, then $[M_{\mathcal{P}'}]_{i,j} = 0$. We want to show that the collection of cells $\mathcal{P}' \otimes \mathcal{Q}$ is contained in $\mathcal{P} \otimes \mathcal{Q}$. Let $C \in \mathcal{P}' \otimes \mathcal{Q}$ be an arbitrary cell. We will proceed to show that $C \in \mathcal{P} \otimes \mathcal{Q}$. Let $C = [(a_1,a_2), (a_1,a_2)+(1,1)]$ for some $a_1,a_2 \in \mathbb{N}$. Let $M_{\mathcal{P}' \otimes \mathcal{Q}}$ be the binary matrix associated with $\mathcal{P}' \otimes \mathcal{Q}$. Note that $M_{\mathcal{P}' \otimes \mathcal{Q}} \in \mathcal{M}_{n\ell ' \times mk'}(\mathbb{Z}_2)$, so we have that $a_2 = n\ell' - p$ for some $p\in \{0,...,n\ell '\}$. As $C \in \mathcal{P}' \otimes \mathcal{Q}$, then this means by the definition of the associated binary matrix that $[M_{\mathcal{P}' \otimes \mathcal{Q}}]_{n\ell' - a_2, a_1 + 1} = 1$. We proceed to show that $[M_{\mathcal{P} \otimes \mathcal{Q}}]_{n\ell - a_2, a_1 + 1} = 1$. Towards a contradiction, assume that $[M_{\mathcal{P} \otimes \mathcal{Q}}]_{n\ell - a_2, a_1 + 1} = 0$. Then:
    \begin{align*}
        0 = [M_{\mathcal{P} \otimes \mathcal{Q}}]_{n\ell - a_2, a_1 + 1} 
        &= [M_{\mathcal{P}} \otimes M_\mathcal{Q}]_{n\ell - a_2, a_1 + 1} \\
        &= [M_\mathcal{P}]_{\left\lceil \frac{n\ell - a_2}{n} \right\rceil, \left\lceil \frac {a_1 + 1}{m} \right\rceil} [M_\mathcal{Q}]_{((n\ell-a_2-1)\%n)+1,(a_1\%m)+1}\\
        &= [M_\mathcal{P}]_{\left\lceil \frac{n\ell - n\ell' + p}{n} \right\rceil, \left\lceil \frac {a_1 + 1}{m} \right\rceil} [M_\mathcal{Q}]_{((n\ell-n\ell' + p -1)\%n)+1,(a_1\%m)+1}\\
        &= [M_\mathcal{P}]_{\left\lceil \ell - \ell' + \frac{p}{n} \right\rceil, \left\lceil \frac {a_1 + 1}{m} \right\rceil} [M_\mathcal{Q}]_{((p -1)\%n)+1,(a_1\%m)+1}\\
        &= [M_\mathcal{P}]_{\ell - \ell' + \left\lceil \frac{p}{n} \right\rceil, \left\lceil \frac {a_1 + 1}{m} \right\rceil} [M_\mathcal{Q}]_{((p -1)\%n)+1,(a_1\%m)+1}\\
        &= [M_{\mathcal{P}'}]_{\left\lceil \frac{p}{n} \right\rceil, \left\lceil \frac {a_1 + 1}{m} \right\rceil} [M_\mathcal{Q}]_{((p -1)\%n)+1,(a_1\%m)+1}\\
        &= [M_{\mathcal{P}'} \otimes M_\mathcal{Q}]_{p, a_1 + 1} \\
        &= [M_{\mathcal{P}'} \otimes M_\mathcal{Q}]_{n\ell'-a_2, a_1 + 1} \\
        &= [M_{\mathcal{P}' \otimes \mathcal{Q}}]_{n\ell' - a_2, a_1 + 1}
    \end{align*}
But this is a contradiction as $[M_{\mathcal{P}' \otimes \mathcal{Q}}]_{n\ell' - a_2, a_1 + 1} = 1$. Therefore it must be that $[M_{\mathcal{P} \otimes \mathcal{Q}}]_{n\ell - a_2, a_1 + 1} = 1$, and so thus by definition of the associated binary matrix we have that $C \in \mathcal{P} \otimes \mathcal{Q}$.
\end{proof}

By the above proposition, we have that for any polyominoes $\mathcal{P}$ and $\mathcal{Q}$ and for any cell $C\in\mathcal{P}$, since $C \subseteq \mathcal{P}$, then the collection of cells $C\otimes\mathcal{Q}$ is contained in  $\mathcal{P}\otimes\mathcal{Q}$. In fact, we can decompose $\mathcal{P}\otimes\mathcal{Q}$ into precisely $|\mathcal{P}|$ collections of cells of the form $C\otimes\mathcal{Q}$ for every $C \in \mathcal{P}$, and we state this now as a lemma.

\begin{lemma} \label{Lemma:3.1.3}
     Let $\mathcal{P}$ and $\mathcal{Q}$ be polyominoes. Then $\mathcal{P}\otimes\mathcal{Q} = \bigcup\limits_{C\in \mathcal{P}}C\otimes\mathcal{Q}$.
\end{lemma}

\begin{proof}
    Let $A \in \bigcup\limits_{C\in \mathcal{P}}C\otimes\mathcal{Q}$ be arbitrary, then $A \in D \otimes \mathcal{Q}$ for some $D \in \mathcal{P}$. Since $D \in \mathcal{P}$, then $D \subseteq \mathcal{P}$, and so by Lemma \ref{Lemma:Tensor products preserve subpolyominoes} we have that $D \otimes \mathcal{Q} \subseteq \mathcal{P} \otimes \mathcal{Q}$ and so therefore $A \in \mathcal{P} \otimes \mathcal{Q}$. As $A$ was arbitrary, then we $\bigcup\limits_{C\in \mathcal{P}}C\otimes\mathcal{Q} \subseteq \mathcal{P} \otimes \mathcal{Q}$. By Lemma \ref{Lemma: 3.1.1}, we have that $|\mathcal{P}\otimes\mathcal{Q}| = |\mathcal{P}||\mathcal{Q}|$. As well, we have that $|C\otimes\mathcal{Q}| = |C||\mathcal{Q}| = 1 \cdot |\mathcal{Q}| = |\mathcal{Q}|$ and so $\lvert \bigcup\limits_{C\in \mathcal{P}}C\otimes\mathcal{Q}\text{ } \rvert = |\mathcal{P}||\mathcal{Q}| = |\mathcal{P} \otimes \mathcal{Q}|$. Thus as $\bigcup\limits_{C\in \mathcal{P}}C\otimes\mathcal{Q}$ and $\mathcal{P} \otimes \mathcal{Q}$ have the same rank and $\bigcup\limits_{C\in \mathcal{P}}C\otimes\mathcal{Q} \subseteq \mathcal{P} \otimes \mathcal{Q}$, then it must be that $\mathcal{P}\otimes\mathcal{Q} = \bigcup\limits_{C\in \mathcal{P}}C\otimes\mathcal{Q}$.
\end{proof}

\begin{corollary} \label{Cor: Union of tensor = tensor of union}
    Let $\mathcal{P}_1, \mathcal{P}_2$, and $\mathcal{Q}$ be polyominoes. Then $(\mathcal{P}_1 \otimes \mathcal{Q})\cup(\mathcal{P}_2 \otimes \mathcal{Q}) = (\mathcal{P}_1 \cup \mathcal{P}_2)\otimes \mathcal{Q}$.
\end{corollary}

\begin{proof}
   We proceed by showing containment both ways. Let $C \in (\mathcal{P}_1 \otimes \mathcal{Q})\cup(\mathcal{P}_2 \otimes \mathcal{Q})$ be arbitrary, then $C \in \mathcal{P}_1 \otimes \mathcal{Q}$ or $C \in \mathcal{P}_2 \otimes \mathcal{Q}$. Since $P_1, P_2 \subseteq \mathcal{P}_1 \cup \mathcal{P}_2$, then by Lemma \ref{Lemma:Tensor products preserve subpolyominoes}, we have that $\mathcal{P}_1 \otimes \mathcal{Q},\mathcal{P}_2 \otimes \mathcal{Q} \subseteq (\mathcal{P}_1 \cup \mathcal{P}_2)\otimes \mathcal{Q}$. So if $C \in \mathcal{P}_1 \otimes \mathcal{Q}$, then $C \in (\mathcal{P}_1 \cup \mathcal{P}_2)\otimes \mathcal{Q}$, and if $C \in \mathcal{P}_2 \otimes \mathcal{Q}$, then also $C \in (\mathcal{P}_1 \cup \mathcal{P}_2)\otimes \mathcal{Q}$. So therefore $(\mathcal{P}_1 \otimes \mathcal{Q})\cup(\mathcal{P}_2 \otimes \mathcal{Q}) \subseteq (\mathcal{P}_1 \cup \mathcal{P}_2)\otimes \mathcal{Q}$. Now let $D \in (\mathcal{P}_1 \cup \mathcal{P}_2)\otimes \mathcal{Q}$ be arbitrary. Then by Lemma \ref{Lemma:3.1.3}, $D \in A \otimes \mathcal{Q}$ for some $A \in \mathcal{P}_1 \cup \mathcal{P}_2$. If $A \in \mathcal{P}_1$, then $A \otimes \mathcal{Q} \subseteq \mathcal{P}_1 \otimes \mathcal{Q}$ by Lemma \ref{Lemma:Tensor products preserve subpolyominoes}. So therefore $D \in \mathcal{P}_1 \otimes \mathcal{Q}$. If rather $A \in \mathcal{P}_2$, then $A \otimes \mathcal{Q} \subseteq \mathcal{P}_2 \otimes \mathcal{Q}$ by Lemma \ref{Lemma:Tensor products preserve subpolyominoes}. So therefore $D \in \mathcal{P}_2 \otimes \mathcal{Q}$. Thus $D \in (\mathcal{P}_1 \otimes \mathcal{Q})\cup(\mathcal{P}_2 \otimes \mathcal{Q})$, and so $(\mathcal{P}_1 \cup \mathcal{P}_2)\otimes \mathcal{Q} \subseteq (\mathcal{P}_1 \otimes \mathcal{Q})\cup(\mathcal{P}_2 \otimes \mathcal{Q})$.
\end{proof}

By Corollary \ref{Cor: Union of tensor = tensor of union} above, we have that if two polyominoes overlap in the plane, then their tensor products must overlap, that is, the tensor product of polyominoes preserves non-empty intersection. What we show now with the following lemma is that the tensor product of polyominoes also preserves empty intersection, meaning if two polyominoes do not share any cells, then their tensor products do not share any cells. 

\begin{lemma} \label{Lemma:Dilations preserve disjoint sets}
    Let $\mathcal{P}$ and $\mathcal{Q}$ be polyominoes and let $A,B \in \mathcal{P}$ be arbitrary with $A \neq B$. Then $|(A\otimes\mathcal{Q})\cap(B\otimes\mathcal{Q})| = 0$. Furthermore, if $|\mathcal{P}_1 \cap \mathcal{P}_2| = 0$ for some polyominoes $\mathcal{P}_1$ and $\mathcal{P}_2$, then $|(\mathcal{P}_1 \otimes \mathcal{Q}) \cap (\mathcal{P}_2 \otimes \mathcal{Q})| = 0$.
\end{lemma}

\begin{proof}
    By its definition, $M_{\mathcal{P} \otimes \mathcal{Q}}$ is a block matrix consisting of $|\mathcal{P}|$ blocks, meaning every cell of $\mathcal{P}$ corresponds to a block of $M_{\mathcal{P} \otimes \mathcal{Q}}$. We let $\sup V(\mathcal{P}) = (k,\ell)$ and $\sup V(\mathcal{Q}) = (m,n)$ for some $k, \ell, m, n \in \mathbb{N}$, so then $M_\mathcal{P} \in \mathcal{M}_{\ell \times k}(\mathbb{Z}_2)$ and $M_\mathcal{Q} \in \mathcal{M}_{n \times m}(\mathbb{Z}_2)$. As such, we have that $M_{\mathcal{P} \otimes \mathcal{Q}} \in \mathcal{M}_{n \ell \times mk}(\mathbb{Z}_2)$. Towards a contradiction, suppose there exists some cell $C \in \mathcal{P} \otimes \mathcal{Q}$ such that $C \in (A\otimes\mathcal{Q})\cap(B\otimes\mathcal{Q})$. Since $C \in \mathcal{P} \otimes \mathcal{Q}$, then $C$ corresponds to some non-zero entry in $M_{\mathcal{P} \otimes \mathcal{Q}}$, say $[M_{\mathcal{P} \otimes \mathcal{Q}}]_{r,s}$. Therefore, $C = [(s-1,n\ell - r), (s, n\ell-r+1)]$. Let $A = [(a_1,a_2),(a_1,a_2)+(1,1)]$ and $B = [(b_1,b_2),(b_1,b_2)+(1,1)]$ for some $a_1,a_2,b_1,b_2 \in \mathbb{N}$. As $C \in (A\otimes\mathcal{Q})$, then by definition of an associated binary matrix and the Kronecker product we have that:
    \[
    ma_1 \leq s-1 < m(a_1 + 1) \text{ and } na_2 \leq n\ell - r < n(a_2+1)
    \] and so,
    \[
    ma_1 + 1 \leq s \leq ma_1 + m \text{ and } n \ell - na_2 - n +1 \leq r \leq n\ell - na_2 \text{.}
    \]
    Similarly, as $C \in (B\otimes\mathcal{Q})$, then also
    \[
    mb_1 + 1 \leq s \leq mb_1 + m \text{ and } n \ell - nb_2 - n +1 \leq r \leq n\ell - nb_2 \text{.}
    \]
    This gives that:
    \begin{align*}
        \left\lceil \frac{ma_1 +1}{m} \right\rceil \leq &\left\lceil \frac{s}{m} \right\rceil \leq \left\lceil \frac{ma_1 + m}{m} \right\rceil \\
        \implies a_1 + \left\lceil \frac{1}{m} \right\rceil \leq &\left\lceil \frac{s}{m} \right\rceil \leq a_1 + 1 \\
        \implies \textcolor{white}{......} a_1 + 1 \leq &\left\lceil \frac{s}{m} \right\rceil \leq a_1 + 1\\
        \implies \textcolor{white}{...................} &\left\lceil \frac{s}{m} \right\rceil = a_1+1 \\ 
    \end{align*} and
    \begin{align*}
        \left\lceil \frac{mb_1 +1}{m} \right\rceil \leq &\left\lceil \frac{s}{m} \right\rceil \leq \left\lceil \frac{mb_1 + m}{m} \right\rceil \\
        \implies b_1 + \left\lceil \frac{1}{m} \right\rceil \leq &\left\lceil \frac{s}{m} \right\rceil \leq b_1 + 1 \\
        \implies \textcolor{white}{......}b_1 + 1 \leq &\left\lceil \frac{s}{m} \right\rceil \leq b_1 + 1\\
        \implies \textcolor{white}{...................} &\left\lceil \frac{s}{m} \right\rceil = b_1+1 \text{.} \\ 
    \end{align*}
So $a_1 + 1 = b_1 + 1$, and therefore $a_1 = b_1$. Also, we have that:
\begin{align*}
        \left\lceil \frac{n\ell - na_2 -n +1}{n} \right\rceil \leq &\left\lceil \frac{r}{n} \right\rceil \leq \left\lceil \frac{n\ell - na_2}{n} \right\rceil \\
        \implies \ell - a_2 - 1 + \left\lceil \frac{1}{n} \right\rceil \leq &\left\lceil \frac{r}{n} \right\rceil \leq \ell - a_2 \\
        \implies \textcolor{white}{.....} \ell - a_2 - 1 + 1 \leq &\left\lceil \frac{r}{n} \right\rceil \leq \ell - a_2\\
        \implies \textcolor{white}{.................} \ell - a_2 \leq &\left\lceil \frac{r}{n} \right\rceil \leq \ell - a_2 \\
        \implies \textcolor{white}{..............................} &\left\lceil \frac{r}{n} \right\rceil = \ell - a_2 \\     
    \end{align*} and
     \begin{align*}
        \left\lceil \frac{n\ell - nb_2 -n +1}{n} \right\rceil \leq &\left\lceil \frac{r}{n} \right\rceil \leq \left\lceil \frac{n\ell - nb_2}{n} \right\rceil \\
        \implies \ell - b_2 - 1 + \left\lceil \frac{1}{n} \right\rceil \leq &\left\lceil \frac{r}{n} \right\rceil \leq \ell - b_2 \\
        \implies \textcolor{white}{.....} \ell - b_2 - 1 + 1 \leq &\left\lceil \frac{r}{n} \right\rceil \leq \ell - b_2\\
        \implies \textcolor{white}{.................} \ell - b_2 \leq &\left\lceil \frac{r}{n} \right\rceil \leq \ell - b_2 \\
        \implies \textcolor{white}{..............................} &\left\lceil \frac{r}{n} \right\rceil = \ell - b_2 \text{.} \\     
    \end{align*} 
    So $\ell - a_2 = \ell - b_2$, and therefore $a_2 = b_2$. So then $A = B$, but this is a contradiction. Therefore it must be that $|(A\otimes\mathcal{Q})\cap(B\otimes\mathcal{Q})| = 0$. Now assume that $|\mathcal{P}_1 \cap \mathcal{P}_2| = 0$ for some polyominoes $\mathcal{P}_1$ and $\mathcal{P}_2$. Then, for all $C \in \mathcal{P}_1$ and for all $D \in \mathcal{P}_2$, $C \neq D$. So then from the above, $|(C \otimes \mathcal{Q}) \cap (D \otimes \mathcal{Q})| = 0$ for all $C \in \mathcal{P}_1$ and for all $D \in \mathcal{P}_2$. Therefore,
\[
    \left|\left(\bigcup_{C\in \mathcal{P}_1}(C \otimes \mathcal{Q})\right) \bigcap \left(\bigcup_{D\in \mathcal{P}_2}(D \otimes \mathcal{Q})\right)\right| = 0
\]
and so by Lemma \ref{Lemma:3.1.3}, $|(\mathcal{P}_1 \otimes \mathcal{Q}) \cap (\mathcal{P}_2 \otimes \mathcal{Q})| = 0$.
\end{proof}

\begin{definition}
    We call the unique polyomino consisting of a single cell the \textit{monomino} and we denote it by $\mathcal{R}_{1,1}$.
\end{definition}

\begin{lemma}
    Let $\mathcal{P}$ be an arbitrary polyomino, then $\mathcal{P} \otimes \mathcal{R}_{1,1} = \mathcal{P}$.
\end{lemma}

\begin{proof}
    If $\mathcal{R}_{1,1}$ is in minimal position, then $M_{\mathcal{R}_{1,1}} = [1]$. Therefore, 
    \[
    M_{\mathcal{P} \otimes {\mathcal{R}_{1,1}}} = M_{\mathcal{P}} \otimes M_{\mathcal{R}_{1,1}} = M_{\mathcal{P}} \otimes [1] = [[M_\mathcal{P}]_{ij}[1]] = [[M_\mathcal{P}]_{ij}] = M_\mathcal{P}
    \]
    Thus $\mathcal{P} \otimes \mathcal{R}_{1,1} = \mathcal{P}$.
\end{proof}

An interesting question to ask now is how the tensor product $\mathcal{P} \otimes \mathcal{R}_{m,n}$ generalizes for arbitrary rectangle polyominoes of width $m$ and height $n$. We note that if $\mathcal{R}_{m,n}$ is in minimal position, then $M_{\mathcal{R}_{m,n}} \in \mathcal{M}_{n \times m}(\mathbb{Z}_2)$ with $[M_{\mathcal{R}_{m,n}}]_{i,j} = 1$ for all $i \in \{1,..,n\}$ and $j\in \{1,...,m\}$. Figure \ref{Fig:3.4} depicts a polyomino $\mathcal{P}$ in the coordinate plane $\mathbb{N}^2$ and Figure \ref{Fig:3.5} depicts $\mathcal{P} \otimes \mathcal{R}_{3,2}$.

\begin{figure} [H]
    \centering
    \begin{tikzpicture} [scale = 0.8]
        \draw[step=1cm,black,very thin, dashed] (0,0) grid (3.9,2.9);
        \draw[->, ultra thick](0,-0.5)--(0,3);
        \draw[->, ultra thick](-0.5,0)--(4,0);
        \polyomino[
        empty cell =x,
        grid,
        p={a}{style={lightgray,draw=black}},
        row sep =;
        ]{
        a x a;
        a a a
        }

        \fill(0,0) circle[radius=2pt] node[above left]{\footnotesize$(0,0)$};
        \fill(1,1) circle[radius=2pt] node[above right]{\footnotesize$(1,1)$};
                
        \begin{scope}[color=black]
            \node[anchor=center] () at (0.5,0.5){$A$};
        \end{scope}
    \end{tikzpicture}
    \caption{A polyomino $\mathcal{P}$ in the coordinate plane $\mathbb{N}^2$}
    \label{Fig:3.4}
\end{figure}
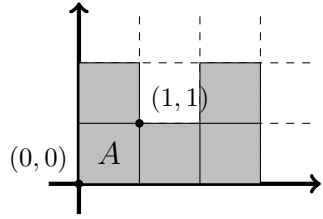

\begin{figure} [H] 
    \centering
    \begin{tikzpicture} [scale = 0.8]
        \draw[step=1cm,black,very thin, dashed] (0,0) grid (9.9,4.9);
        \draw[->, ultra thick](0,-0.5)--(0,5);
        \draw[->, ultra thick](-0.5,0)--(10,0);
        \polyomino[
        empty cell =x,
        grid,
        p={a}{style={lightgray,draw=black}},
        row sep =;
        ]{
        a a a x x x a a a;
        a a a x x x a a a;
        a a a a a a a a a;
        a a a a a a a a a
        }

        \fill(0,0) circle[radius=2pt] node[above left]{\footnotesize $(0,0)$};
        \fill(3,2) circle[radius=2pt] node[above right]{\footnotesize $(3,2)$};
        
    \end{tikzpicture}
    \caption{The polyomino $\mathcal{P} \otimes \mathcal{R}_{3,2}$}
    \label{Fig:3.5}
\end{figure}
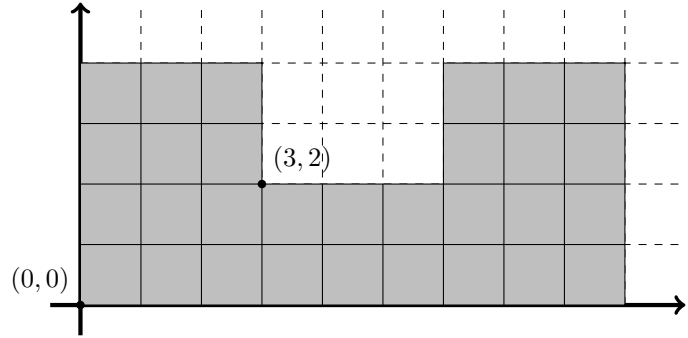

By the above figure, it seems that the tensor product of a polyomino with a rectangle polyomino results in a stretching or dilation of the original polyomino, for intuitively we are replacing every cell in $\mathcal{P}$ with some cell interval of width $m$ and height $n$. We will soon show in Proposition \ref{Proposition: Dilations of polyos are polyos} that the tensor product $\mathcal{P} \otimes \mathcal{R}_{m,n}$ is always a polyomino. First, we introduce a map on $\mathbb{N}^2$ and, thus, implicitly a map on the vertices of a polyomino, that will allow us to label vertices in $\mathcal{P} \otimes \mathcal{R}_{m,n}$ in the natural way in which they correspond to vertices in $\mathcal{P}$.

\begin{definition}
    Let $m,n\in \mathbb{Z}^+$. We say that the map
    \begin{eqnarray*}
       \Delta_{m,n}: \mathbb{N}^2 & \rightarrow & \mathbb{N}^2\\
       (x,y) & \mapsto & (mx, ny)
    \end{eqnarray*} is the $(m,n)$\textit{-dilation map} on $\mathbb{N}^2$.
\end{definition}

If we have some collection of cells $\mathcal{Q} = \mathcal{P} \otimes \mathcal{R}_{m,n}$ for some polyomino $\mathcal{P}$, then we refer to $\mathcal{Q}$ as a \textit{dilation} of $\mathcal{P}$. As an example, looking back at Figure \ref{Fig:3.4} and Figure \ref{Fig:3.5}, we see that applying the above definition in labeling $\mathcal{P} \otimes \mathcal{R}_{m,n}$ gives that $\Delta_{3,2}((1,1)) = (3,2)$. Certainly, there is a natural correspondence between the vertex $(1,1)$ in $\mathcal{P}$ and the vertex $(3,2)$ in $\mathcal{P} \otimes \mathcal{R}_{m,n}$ and this is the reason we claim this map is the natural choice for this labeling. For ease of communication, we define the following notation.

\begin{definition}
    Let $[a,b]$ be an interval in $\mathbb{N}^2$ and $m,n\in \mathbb{Z}^+$. We use $\Delta_{m,n}([a,b])$ to denote the interval $[\Delta_{m,n}(a),\Delta_{m,n}(b)]$ in $\mathbb{N}^2$.
\end{definition}

\begin{remark}
    Note that if $\mathcal{P}$ is a polyomino and $H = [a,b]$ is a horizontal edge interval of $\mathcal{P}$ consisting of $k$ edges, then $\Delta_{m,n}(H)$ is a horizontal edge interval of $\mathcal{P} \otimes \mathcal{R}_{m,n}$ consisting of $mk$ edges. Similarly, if $V = [a,b]$ is a vertical edge interval of $\mathcal{P}$ consisting of $\ell$ edges, then $\Delta_{m,n}(V)$ is a vertical edge interval of $\mathcal{P} \otimes \mathcal{R}_{m,n}$ consisting of $n \ell$ edges. Thus, dilation preserves horizontal and vertical edge intervals. As well, if there is a vertex $v \in \mathcal{P} \otimes \mathcal{R}_{m,n}$ such that $v$ lies on $\Delta_{m,n}(H)$ and $v$ lies on $\Delta_{m,n}(V)$ for some horizontal edge interval $H$ and some vertical edge interval $V$ in $\mathcal{P}$, then $v = \Delta_{m,n}(w)$ for some $w \in \mathcal{P}$.
\end{remark} 

We now show that the dilation of any polyomino is a polyomino.

\begin{proposition} \label{Proposition: Dilations of polyos are polyos}
    Let $\mathcal{P}$ be a polyomino and let $\mathcal{R}_{m,n}$ be a rectangle polyomino for some $m,n\in \mathbb{Z}^+$. Then $\mathcal{Q} = \mathcal{P}\otimes\mathcal{R}_{m,n}$ is a polyomino.
\end{proposition}

\begin{proof}
    We want to show that $\mathcal{Q}$ is a polyomino. We proceed to show that any two cells in $\mathcal{Q}$ are connected. Let $C,D \in \mathcal{Q}$ be arbitrary cells. Then $C \in A \otimes \mathcal{R}_{m,n}$ and $D \in B \otimes \mathcal{R}_{m,n}$ for some $A,B \in \mathcal{P}$ by Lemma \ref{Lemma:3.1.3}. As $\mathcal{P}$ is a polyomino then we know that $A$ and $B$ are connected in $\mathcal{P}$, that is, there exists some walk $\mathcal{W}: A=A_1, A_2,...,A_{r-1},A_r=B$ with $A_i \in \mathcal{P}$ for all $i \in \{2,...,r-1\}$. As consecutive cells in $\mathcal{W}$ are connected edge to edge, then there exist edges $e_1,e_2,...,e_{r-1}\in E(\mathcal{P})$ such that $e_i = A_i \cap A_{i+1}$ for all $i\in\{1,...,r-1\}$. Note that for this set of edges, if $e_i = [u_i,v_i]$ is a horizontal edge of $\mathcal{P}$ for some $u_i,v_i \in V(\mathcal{P})$, then $\Delta_{m,n}(e_i) = [\Delta_{m,n}(u_i), \Delta_{m,n}(v_i)]$ is a horizontal edge interval of $\mathcal{Q}$ that consists of $m$ edges. We define $H_{e_i} = \bigcup \limits_{j=1}^{m}{\varepsilon_{i,j}}$, with $\varepsilon_{i,j}\in E(\mathcal{Q})$, to be this horizontal edge interval $\Delta_{m,n}(e_i)$ in $\mathcal{Q}$. If rather, $e_i = [u_i,v_i]$ is a vertical edge of $\mathcal{P}$, then $\Delta_{m,n}(e_i)$ is a vertical edge interval of $\mathcal{Q}$ that consists of $n$ edges. We define $V_{e_i} = \bigcup \limits_{j=1}^{n}{\varepsilon_{i,j}}$ to be this vertical edge interval with all the $\varepsilon_{i,j}\in E(\mathcal{Q})$. As we had that each $e_i = A_i \cap A_{i+1}$, then we have that $\varepsilon_{i,j} = K_{i,j} \cap K_{i,j}'$ for some $K_{i,j} \in A_i \otimes \mathcal{R}_{m,n}$ and $K_{i,j}' \in A_{i+1} \otimes \mathcal{R}_{m,n}$. Note in particular, that since $K_{1,1} \in A_1 \otimes \mathcal{R}_{m,n} = A \otimes \mathcal{R}_{m,n}$, and $A \otimes \mathcal{R}_{m,n}$ is a cell interval and therefore a polyomino in $\mathbb{N}^2$, then as $C \in A \otimes \mathcal{R}_{m,n}$, we have the $C$ and $K_{1,1}$ are connected in $A \otimes \mathcal{R}_{m,n}$, so there exists a walk $\mathcal{W}_1 :C,C_{1,1},C_{1,2},...,C_{1,{t_1}},K_{1,1}$ for some cells $C_{1,1},...,C_{1,{t_1}}\in A_1 \otimes \mathcal{R}_{m,n}$ and some $t_1 \in \mathbb{N}$. By the same argument, there exists a walk $\mathcal{W}_2 :K_{1,1}',C_{2,1},C_{2,2},...,C_{2,{t_2}},K_{2,1}$ in $A_2 \otimes \mathcal{R}_{m,n}$. We therefore define the following collection of walks $\{\mathcal{W}_\ell\}_{\ell \in \{2,...,r-1\}}$, with $\mathcal{W}_{\ell}: K_{{(\ell -1)},1}',C_{\ell ,1},C_{\ell,2},...,C_{\ell,{t\ell}},K_{\ell,1}$ for some cells $C_{\ell,1},...,C_{\ell,{t_\ell}} \in A_\ell \otimes \mathcal{R}_{m,n}$ and some $t_\ell \in \mathbb{N}$. We also note that as both $K_{(r-1),1}'$ and $D$ are cells in $B \otimes \mathcal{R}_{m,n}$, then we have that there exists some walk $\mathcal{W}_{r}: K_{{(r -1)},1}',C_{r,1},C_{r,2},...,C_{r,{t_r}},D$ for some cells $C_{r,1},...,C_{r,{t_r}} \in B \otimes \mathcal{R}_{m,n}$ and some $t_r \in \mathbb{N}$. Thus we have that the concatenated walk $\mathcal{W}_{\Delta}: \mathcal{W}_1, \mathcal{W}_2,..., \mathcal{W}_{r-1}, \mathcal{W}_r$ is a walk from $C$ to $D$ in $\mathcal{Q}$, and so $C$ and $D$ are connected (see Figure \ref{Fig:Dilations preserve connectedness}). Since $C$ and $D$ were arbitrary cells in $\mathcal{Q}$, then we have that $\mathcal{Q}$ is a polyomino.
\end{proof}

 \begin{figure} [H]
    \centering
    \begin{tikzpicture}
        \draw[step=1cm,white,very thin] (0,0) grid (12,12);
        \draw [white, thin, fill = lightgray] (4,4) rectangle (5.6,8);
        \draw [white, thin, fill = lightgray] (6.4,4) rectangle (8,8);
        \draw [black, thin, fill = lightgray] (0,0) rectangle (4,8);
        \draw [black, thin, fill = lightgray] (8,4) rectangle (12,12);
        \polyomino[
        empty cell =x,
        grid,
        p={a}{style={lightgray,draw=black}},
        p={b}{style={gray,draw=black}},
        p={c}{style={cyan,draw=black}},
        p={d}{style={Thistle,draw=black}},
        row sep =;
        ]{
        x x x x x x x x x x x x;
        x x x x x x x x x b d x;
        x x x x x x x x x x x x;
        x x x x x x x x b b x x;
        x x x x x x x x b x x x;
        x x x x x x x x x x x x;
        x x x x x x x x b x x x;
        b b x b b x x b b x x x;
        b x x x x x x x x x x x;
        b x b x x x x x x x x x;
        x x c x x x x x x x x x;
        x x x x x x x x x x x x
        }

    \draw[black, thin] (0,4) -- (5.6,4);
    \draw[black, thin] (4,8) -- (5.6,8);
    \draw[black, thin] (6.4,4) -- (8,4);
    \draw[black, thin] (6.4,8) -- (12,8);

    \draw[red, ultra thick] (0,4) -- (4,4);
    \draw[blue, ultra thick] (4,4) -- (4,8);
    \draw[blue, ultra thick] (8,4) -- (8,8);
    \draw[red, ultra thick] (8,8) -- (12,8);

     \begin{scope}[color=black]
            \node[below] () at (0,0){$A \otimes \mathcal{R}_{m,n}$};
            \node[above] () at (0,8) {$A_2 \otimes \mathcal{R}_{m,n}$};
            \node[below] () at (12,4){$A_{r-1} \otimes \mathcal{R}_{m,n}$};
            \node[above] () at (12,12){$B \otimes \mathcal{R}_{m,n}$};

            \node[anchor=center] () at (6,6) {\huge $\cdot \cdot \cdot$};
            
            \node[anchor=center] () at (2.5,1.5) {$C$};
            \node[anchor=center] () at (2.5,2.5) {$C_{1,1}$};
            \node[anchor=center] () at (0.5,2.5) {$C_{1,{t_1}}$};
            \node[anchor=center] () at (0.5,3.5) {$K_{1,1}$};
            \node[anchor=center] () at (0.5,4.5) {$K_{1,1}'$};
            \node[anchor=center] () at (1.5,4.5) {$C_{2,1}$};
            \node[anchor=center] () at (3.5,4.5) {$K_{2,1}$};
            \node[anchor=center] () at (4.5,4.5) {$K_{2,1}'$};
            
            \node[anchor=center] () at (7.5,4.5) {\tiny $K_{(r-2),1}$};
            \node[anchor=center] () at (8.5,4.5) {\tiny $K_{(r-2),1}'$};
            \node[anchor=center] () at (8.5,5.5) {\tiny $C_{(r-1),1}$};
            \node[anchor=center] () at (8.5,7.5) {\tiny $K_{(r-1),1}$};
            \node[anchor=center] () at (8.5,8.5) {\tiny $K_{(r-1),1}'$};
            \node[anchor=center] () at (9.5,8.5) {$C_{r,1}$};
            \node[anchor=center] () at (9.5,10.5) {$C_{r,{t_r}}$};
            \node[anchor=center] () at (10.5,10.5) {$D$};

            \node[above] () at (1.5,2.5) {$\mathcal{W}_1$};
            \node[above] () at (2.5,4.5) {$\mathcal{W}_2$};
            \node[right] () at (8.5,6.5) {$\mathcal{W}_{r-1}$};
            \node[right] () at (9.5,9.5) {$\mathcal{W}_{r}$};
    \end{scope}

\begin{scope}[color=black]
            \node[left] () at (0,4) {\textcolor{red}{$H_{e_1}$}};
            \node[above] () at (4,8) {\textcolor{blue}{$V_{e_2}$}};
            \node[below] () at (8,4) {\textcolor{blue}{$V_{e_{r-2}}$}};
            \node[right] () at (12,8) {\textcolor{red}{$H_{e_{r-1}}$}};
\end{scope}
        
    \draw[dots]  (2.3,4.5)--(2.8,4.5);
    \draw[dots]  (5.1,4.5)--(5.6,4.5);
    \draw[dots]  (6.47,4.5)--(6.9,4.5);
    \draw[dots]  (1.3,2.5)--(1.8,2.5);
    \draw[dots]  (8.5,6.3)--(8.5,6.8);
    \draw[dots]  (9.5,9.3)--(9.5,9.8);
    
    \end{tikzpicture}
    \caption{$\mathcal{W}_\Delta: \mathcal{W}_1,\mathcal{W}_2,..., \mathcal{W}_{r-1}, \mathcal{W}_r$ is a walk that connects $C$ and $D$ in $\mathcal{Q}$.}
    \label{Fig:Dilations preserve connectedness}
\end{figure}
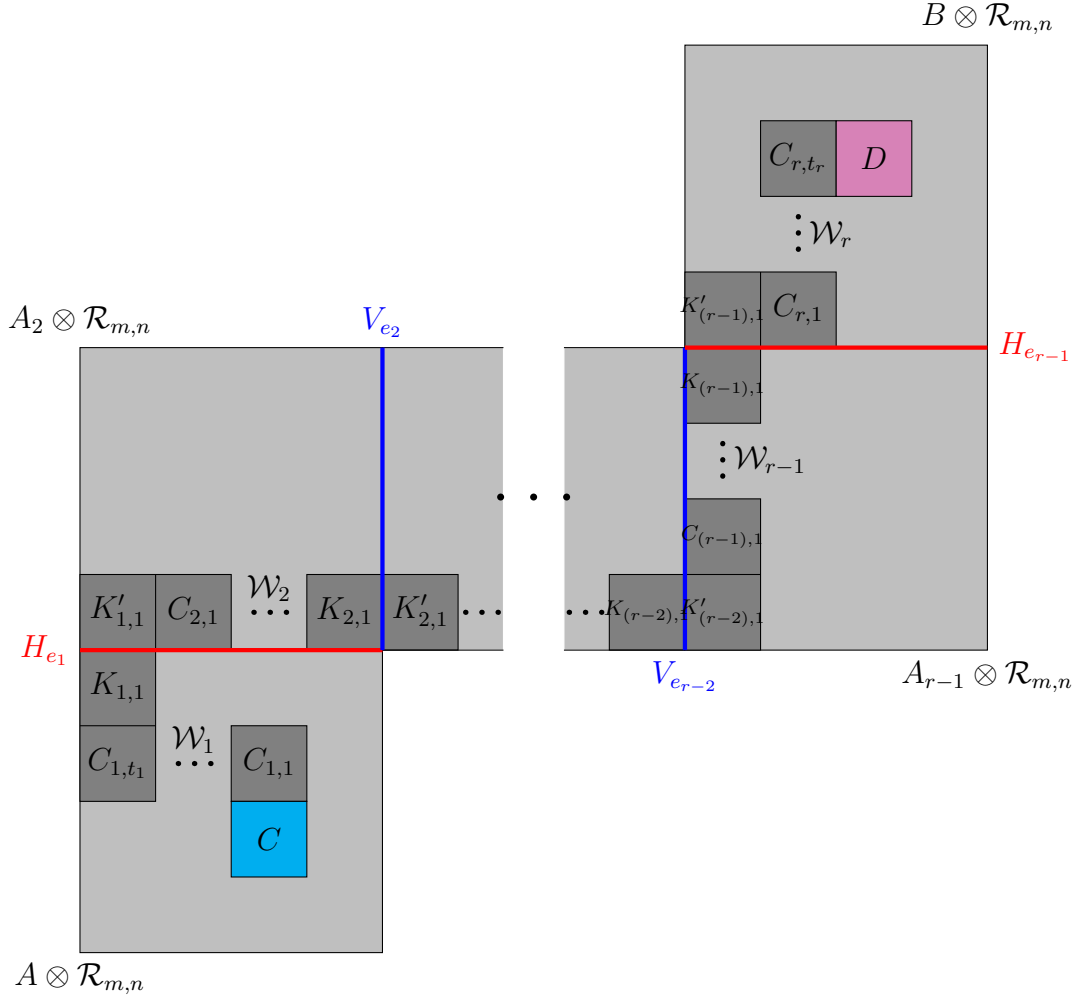

\begin{corollary} \label{Cor: dilations preserve connectedness}
    Let $\mathcal{P}$ be a polyomino and let  $A, B \in \mathcal{P}$ be cells that are connected in $\mathcal{P}$. If $A' \in A \otimes \mathcal{R}_{m,n}$ and $B' \in B \otimes \mathcal{R}_{m,n}$, then $A'$ and $B'$ are connected in $\mathcal{P} \otimes \mathcal{R}_{m,n}$.
\end{corollary}

\begin{proof}
    If $A$ and $B$ are connected, then there is a walk between them, say $\mathcal{W}$. As $\mathcal{W}$ is a connected collection of cells, then it is a polyomino. By Proposition \ref{Proposition: Dilations of polyos are polyos}, we have that $\mathcal{W} \otimes \mathcal{R}_{m,n}$ is a polyomino. So as $A' \in A \otimes \mathcal{R}_{m,n} \subseteq \mathcal{W} \otimes \mathcal{R}_{m,n}$ and $B' \in B \otimes \mathcal{R}_{m,n} \subseteq \mathcal{W} \otimes \mathcal{R}_{m,n}$, then $A'$ and $B'$ are connected.
\end{proof}

\begin{corollary} 
     \label{Cor: dilations preserve unconnectedness}
    Let $\mathcal{P}$ be a polyomino and let $F$ and $G$ be two cells not in $\mathcal{P}$ such that there is a walk of cells not in $\mathcal{P}$ between them. Let $m,n \in \mathbb{Z}^+$. If $F' \in F \otimes \mathcal{R}_{m,n}$ and $G' \in G \otimes \mathcal{R}_{m,n}$, then there exists a walk of cells not in $\mathcal{P} \otimes \mathcal{R}_{m,n}$ between them.
\end{corollary}

\begin{proof} 
    Let $\mathcal{W}$ denote the walk of cells not in $\mathcal{P}$ between $F$ and $G$. Then we can think of $\mathcal{W}$ as a polyomino with $|\mathcal{W} \cap \mathcal{P}| = 0$. By Proposition \ref{Proposition: Dilations of polyos are polyos}, we have that $\mathcal{W} \otimes \mathcal{R}_{m,n}$ is a polyomino, and so by Lemma \ref{Lemma:Dilations preserve disjoint sets}, $|(\mathcal{W} \otimes \mathcal{R}_{m,n}) \cap (\mathcal{P} \otimes \mathcal{R}_{m,n})| = 0$. So as $F' \in F \otimes \mathcal{R}_{m,n} \subseteq \mathcal{W} \otimes \mathcal{R}_{m,n}$ and $G' \in G \otimes \mathcal{R}_{m,n} \subseteq \mathcal{W} \otimes \mathcal{R}_{m,n}$, then $F'$ and $G'$ are connected by a walk of cells in $\mathcal{W} \otimes \mathcal{R}_{m,n}$, with every cell in the walk not in $\mathcal{P} \otimes \mathcal{R}_{m,n}$.
\end{proof}

\begin{corollary} \label{Cor:Simple implies dilation simple}
    Let $\mathcal{P}$ be a polyomino and let $\mathcal{Q} = \mathcal{P} \otimes \mathcal{R}_{m,n}$ for some $m,n \in \mathbb{Z}^+$. If $\mathcal{P}$ is simple, then $\mathcal{Q} = \mathcal{P} \otimes \mathcal{R}_{m,n}$ is simple.
\end{corollary}

\begin{proof}
    Assume $\mathcal{P}$ is simple. Let $F',G' \notin \mathcal{Q}$ be arbitrary. By Lemma \ref{Lemma:3.1.3}, we have that $F' \in F \otimes \mathcal{R}_{m,n}$ and $G' \in G \otimes \mathcal{R}_{m,n}$, for some $F,G \notin \mathcal{P}$, for if it were the case that $F,G \in \mathcal{P}$, then $F \otimes \mathcal{R}_{m,n} \subseteq \mathcal{Q}$ and $F \otimes \mathcal{R}_{m,n} \subseteq \mathcal{Q}$ by Lemma \ref{Lemma:Tensor products preserve subpolyominoes}, so therefore $F',G' \in \mathcal{Q}$, but this is a contradiction. As $F,G \notin \mathcal{P}$ and $\mathcal{P}$ is simple, then there exists a walk of cells not in $\mathcal{P}$ between $F$ and $G$. By Corollary \ref{Cor: dilations preserve unconnectedness}, we have that there must exist a walk of cells not in $\mathcal{Q}$ between $F'$ and $G'$. As $F',G' \notin \mathcal{Q}$ were arbitrary, then it must be that $\mathcal{Q}$ is simple.
\end{proof}

Now that we have proven that any dilation $\mathcal{P} \otimes \mathcal{R}_{m,n}$ of a polyomino $\mathcal{P}$ is itself a polyomino, we proceed to state and prove some useful lemmas that give insight into its structure and that will help us prove our main results in Section \ref{Section:4}.

\begin{lemma} \label{Lemma:Dilated vertices are in dilation}
    Let $\mathcal{P}$ be a polyomino and let $\mathcal{R}_{m,n}$ be a rectangle polyomino for some $m,n\in \mathbb{Z}^+$. Then $\Delta_{m,n}(V(\mathcal{P})) \subseteq V(\mathcal{P}\otimes\mathcal{R}_{m,n})$.
\end{lemma}

\begin{proof}
    Let $v \in V(\mathcal{P})$ be arbitrary with $v = (a,b)$. We let $\sup V(\mathcal{P}) = (k, \ell)$ for some $k,\ell \in \mathbb{N}$, therefore $M_\mathcal{P} \in \mathcal{M}_{\ell \times k}(\mathbb{Z}_2)$. If $v \in V(\mathcal{P})$, then at least one of the four cells in $\mathbb{N}^2$ that has $v$ as a corner is in $\mathcal{P}$, that is $[v,v+(1,1)],[v-(1,0),v+(0,1)],[v-(1,1),v]$, or $[v-(0,1),v+(1,0)]$ is a cell of $\mathcal{P}$. Equivalently, this means that:
    \[
    [M_\mathcal{P}]_{\ell - b, a+1}+[M_\mathcal{P}]_{\ell - b, a}+[M_\mathcal{P}]_{\ell - b+1, a}+[M_\mathcal{P}]_{\ell - b+1, a+1} \geq 1,
    \]
    If it is not the case that $\ell - b, \ell - b + 1 \in \{1,...,\ell\}$ and $a,a+1 \in \{1,...,k\}$, then we will still have that at least one of the summands is nonzero, so we can simply ignore the summands $[M_\mathcal{P}]_{i,j}$ whose indices $i,j$ are outside the dimensions of $M_\mathcal{P}$. We wish to show that $\Delta_{m,n}(V(\mathcal{P})) \subseteq V(\mathcal{P}\otimes\mathcal{R}_{m,n})$. As $v \in V(\mathcal{P})$ was arbitrary, this is equivalent to showing that $\Delta_{m,n}(v) \in V(\mathcal{P}\otimes\mathcal{R}_{m,n})$. We have $\Delta_{m,n}(v) = (ma,nb)$. Since $M_\mathcal{P} \in \mathcal{M}_{\ell \times k}(\mathbb{Z}_2)$ and $M_{\mathcal{R}_{m,n}} \in \mathcal{M}_{n \times m}(\mathbb{Z}_2)$, then $M_{\mathcal{P} \otimes \mathcal{R}_{m,n}} \in \mathcal{M}_{n\ell \times mk}(\mathbb{Z}_2)$. To show that $\Delta_{m,n}(v) = (ma,nb)\in V(\mathcal{P}\otimes\mathcal{R}_{m,n})$, we proceed to show at least one of the four cells with $\Delta_{m,n}(v)$ as a corner is in $\mathcal{P}\otimes\mathcal{R}_{m,n}$, or equivalently, that
    \begin{multline*}
    [M_{\mathcal{P}\otimes \mathcal{R}_{m,n}}]_{n\ell - nb, ma+1}+[M_{\mathcal{P}\otimes \mathcal{R}_{m,n}}]_{n\ell - nb, ma}
    +[M_{\mathcal{P}\otimes \mathcal{R}_{m,n}}]_{n\ell - nb+1, ma}\\+[M_{\mathcal{P}\otimes \mathcal{R}_{m,n}}]_{n\ell - nb+1, ma+1} \geq 1,
    \end{multline*}
    where again, we ignore the summands in the inequation $[M_{\mathcal{P}\otimes \mathcal{R}_{m,n}}]_{i,j}$ whose indices $i,j$ are outside the dimensions of $M_{\mathcal{P}\otimes \mathcal{R}_{m,n}}$. By the definition of the Kronecker product we have that
    \[
    [M_{\mathcal{P} \otimes \mathcal{R}_{m,n}}]_{n(r-1)+w,m(s-1)+z} = [M_\mathcal{P} \otimes M_{\mathcal{R}_{m,n}}]_{n(r-1)+w,m(s-1)+z} = [M_\mathcal{P}]_{rs}[ M_{\mathcal{R}_{m,n}}]_{wz}
    \]
    Therefore,
    \begin{align*}
    [M_{\mathcal{P}\otimes \mathcal{R}_{m,n}}]_{n\ell - nb, ma+1} &= [M_\mathcal{P} \otimes M_{\mathcal{R}_{m,n}}]_{n\ell - nb, ma+1} \\
    &= [M_\mathcal{P} \otimes M_{\mathcal{R}_{m,n}}]_{n(\ell - b-1)+n, ma+1} \\
    &= [M_\mathcal{P}]_{\ell -b, a+1}\cdot[M_{\mathcal{R}_{m,n}}]_{n,1}\\
    &= [M_\mathcal{P}]_{\ell -b, a+1} \cdot 1 \\
    &= [M_\mathcal{P}]_{\ell -b, a+1}
    \end{align*}
    \begin{align*}
    [M_{\mathcal{P}\otimes \mathcal{R}_{m,n}}]_{n\ell - nb, ma} &= [M_\mathcal{P} \otimes M_{\mathcal{R}_{m,n}}]_{n\ell - nb, ma} \\
    &= [M_\mathcal{P} \otimes M_{\mathcal{R}_{m,n}}]_{n(\ell - b-1)+n, m(a-1)+m} \\
    &= [M_\mathcal{P}]_{\ell -b, a}\cdot[M_{\mathcal{R}_{m,n}}]_{n,m}\\
    &= [M_\mathcal{P}]_{\ell -b, a} \cdot 1 \\
    &= [M_\mathcal{P}]_{\ell -b, a}
    \end{align*}
    \begin{align*}
    [M_{\mathcal{P}\otimes \mathcal{R}_{m,n}}]_{n\ell - nb+1, ma} &= [M_\mathcal{P} \otimes M_{\mathcal{R}_{m,n}}]_{n\ell - nb+1, ma} \\
    &= [M_\mathcal{P} \otimes M_{\mathcal{R}_{m,n}}]_{n(\ell - b)+1, m(a-1)+m} \\
    &= [M_\mathcal{P}]_{\ell -b+1, a}\cdot[M_{\mathcal{R}_{m,n}}]_{1,m}\\
    &= [M_\mathcal{P}]_{\ell -b+1, a} \cdot 1 \\
    &= [M_\mathcal{P}]_{\ell -b+1, a}
    \end{align*} and
    \begin{align*}
    [M_{\mathcal{P}\otimes \mathcal{R}_{m,n}}]_{n\ell - nb+1, ma+1} &= [M_\mathcal{P} \otimes M_{\mathcal{R}_{m,n}}]_{n\ell - nb+1, ma+1} \\
    &= [M_\mathcal{P} \otimes M_{\mathcal{R}_{m,n}}]_{n(\ell - b)+1, ma+1} \\
    &= [M_\mathcal{P}]_{\ell -b+1, a+1}\cdot[M_{\mathcal{R}_{m,n}}]_{1,1}\\
    &= [M_\mathcal{P}]_{\ell -b+1, a+1} \cdot 1 \\
    &= [M_\mathcal{P}]_{\ell -b+1, a+1}
    \end{align*}
Thus,
 \begin{multline*}
    [M_{\mathcal{P}\otimes \mathcal{R}_{m,n}}]_{n\ell - nb, ma+1}+[M_{\mathcal{P}\otimes \mathcal{R}_{m,n}}]_{n\ell - nb, ma}
    +[M_{\mathcal{P}\otimes \mathcal{R}_{m,n}}]_{n\ell - nb+1, ma}\\+[M_{\mathcal{P}\otimes \mathcal{R}_{m,n}}]_{n\ell - nb+1, ma+1} \\= [M_\mathcal{P}]_{\ell - b, a+1}+[M_\mathcal{P}]_{\ell - b, a}+[M_\mathcal{P}]_{\ell - b+1, a}+[M_\mathcal{P}]_{\ell - b+1, a+1} \geq 1
    \end{multline*}
    and therefore $\Delta_{m,n}(v) \in V(\mathcal{P}\otimes\mathcal{R}_{m,n})$.
\end{proof}

\begin{lemma} \label{Lemma: Dilations preserve inner intervals}
    Let $\mathcal{P}$ be a polyomino with $u,v \in V(\mathcal{P})$ and let $\mathcal{I} = [u,v]$ be an inner interval of $\mathcal{P}$. Then $\mathcal{I} \otimes \mathcal{R}_{m,n} = [\Delta_{m,n}(u),\Delta_{m,n}(v)]$ is an inner interval of $\mathcal{P} \otimes \mathcal{R}_{m,n}$ for any $m,n \in \mathbb{Z}^+$.
\end{lemma}

\begin{proof}
    We let $\sup V(\mathcal{P}) = (k, \ell)$ for some $k, \ell \in \mathbb{N}$, so therefore $M_\mathcal{P} \in \mathcal{M}_{\ell \times k}(\mathbb{Z}_2)$. Let $u = (a,b)$ and $v = (c,d)$. Since $\mathcal{I} = [u,v]$ is an inner interval of $\mathcal{P}$, then $[M_\mathcal{P}]_{i,j} = 1$ for all $i \in \{1,...,\ell\}$, $j \in \{1,...,k\}$ with $\ell - d+1 \leq i \leq l-b$ and $a+1 \leq j \leq c$. We have by Lemma \ref{Lemma:Dilated vertices are in dilation} that $\Delta_{m,n}(u), \Delta_{m,n}(v) \in V(\mathcal{P} \otimes \mathcal{R}_{m,n})$. We note that $\Delta_{m,n}(u) = (ma, nb)$ and $\Delta_{m,n}(v) = (mc, nd)$ and also that $M_{\mathcal{P} \otimes \mathcal{R}_{m,n}} \in \mathcal{M}_{n\ell \times mk}(\mathbb{Z}_2)$. Towards a contradiction, assume that $\mathcal{I} \otimes \mathcal{R}_{m,n} = [\Delta_{m,n}(u),\Delta_{m,n}(v)]$ is not an inner interval of $\mathcal{P} \otimes \mathcal{R}_{m,n}$. Then there is some cell of $\mathcal{I} \otimes \mathcal{R}_{m,n}$ that is not in $\mathcal{P} \otimes \mathcal{R}_{m,n}$, or equivalently $[M_{\mathcal{P} \otimes \mathcal{R}_{m,n}}]_{rs} = 0$ for some $r \in \{1,...,n\ell\}$, $s \in \{1,...,mk\}$ with $n\ell - nd + 1 \leq r \leq n\ell - nb$ and $ma+1 \leq s \leq mc$. Thus, $r = n(\ell - \beta) + p$ for some $b \leq \beta \leq d$ and some $0 \leq p \leq n-1$, where
    \begin{align*}
        &p = 0 \text{ if } \beta = b \text{ ,} \\
        0 \leq &p \leq n-1 \text{ if } b < \beta < d \text{ , and} \\
        1 \leq &p \leq n-1 \text{ if } \beta = d \text{ .}
    \end{align*}
    Similarly, $s = m\alpha + q$ for some $a \leq \alpha \leq c$ and some $0 \leq q \leq m-1$, where 
    \begin{align*}
        1 \leq &q \leq m-1 \text{ if } \alpha = a \text{ ,} \\
        0 \leq &q \leq m-1 \text{ if } a < \alpha < c \text{ , and} \\
        &q = 0 \text{ if } \alpha = c \text{ .}
    \end{align*}
    We recall that by its definition, $[M_{\mathcal{R}_{m,n}}]_{i,j} = 1$ for all $i \in \{1,...,n\}$ and $j \in \{1,...,m\}$. So then
    \begin{align*}
    0 = [M_{\mathcal{P} \otimes \mathcal{R}_{m,n}}]_{rs} 
    &= [M_\mathcal{P} \otimes M_{\mathcal{R}_{m,n}}]_{rs} \\
    &= [M_\mathcal{P}]_{\left\lceil \frac{r}{n} \right\rceil, \left\lceil \frac{s}{m} \right\rceil} \cdot [M_{\mathcal{R}_{m,n}}]_{((r-1)\%n)+1,((s-1)\%m)+1} \\
    &= [M_\mathcal{P}]_{\left\lceil \frac{n(\ell - \beta) + p}{n} \right\rceil, \left\lceil \frac{m\alpha + q}{m} \right\rceil} \cdot 1 \\
    &= [M_\mathcal{P}]_{\ell - \beta + \left\lceil \frac{p}{n} \right\rceil, \alpha + \left\lceil \frac{q}{m} \right\rceil}
    \end{align*}
    However, $\ell - d + 1 \leq \ell - \beta + \lceil \frac{p}{n} \rceil \leq \ell - b$ and $a + 1 \leq \alpha + \lceil \frac{q}{m} \rceil \leq c$, so 
    \[ 
    [M_{\mathcal{P} \otimes \mathcal{R}_{m,n}}]_{rs} = [M_\mathcal{P}]_{\ell - \beta + \lceil \frac{p}{n} \rceil, \alpha + \lceil \frac{q}{m} \rceil} = 1
    \]
    which is a contradiction. Thus it must be that $\mathcal{I} \otimes \mathcal{R}_{m,n} = [\Delta_{m,n}(u),\Delta_{m,n}(v)]$ is an inner interval of $\mathcal{P} \otimes \mathcal{R}_{m,n}$.
\end{proof}

\begin{lemma} \label{Lemma: Dilations preserve non-inners}
    Let $\mathcal{P}$ be a polyomino with $u,v \in V(\mathcal{P})$ and let $\mathcal{I} = [u,v]$ be an interval of $\mathbb{N}^2$. If $\mathcal{I} \otimes \mathcal{R}_{m,n} = [\Delta_{m,n}(u),\Delta_{m,n}(v)]$ is an inner interval of $\mathcal{P} \otimes \mathcal{R}_{m,n}$ for some $m,n \in \mathbb{Z}^+$, then $\mathcal{I}$ is an inner interval of $\mathcal{P}$.
\end{lemma}

\begin{proof}
    We proceed with a proof by contraposition. Assume $\mathcal{I} = [u,v]$ is not an inner interval of $\mathcal{P}$. We let $u = (a,b)$ and $v = (c,d)$. So there exists some $f \in \mathbb{N}^2$ with $u \leq f < v$, such that the cell $F = [f, f+(1,1)] \notin \mathcal{P}$. Then $|F \cap \mathcal{P}| = 0$, and so by Lemma \ref{Lemma:Dilations preserve disjoint sets}, we have that $|(F \otimes \mathcal{R}_{m,n}) \cap (\mathcal{P} \otimes \mathcal{R}_{m,n})| = 0$. Let $f = (f_1,f_2)$. We note that $\Delta_{m,n}(u) = (ma,nb)$ and $\Delta_{m,n}(v)=(mc,nd)$ and also that:
    \begin{align*}
        F \otimes \mathcal{R}_{m,n} &= \Delta_{m,n}([f,f+(1,1)]) \\
        &= [\Delta_{m,n}(f), \Delta_{m,n}(f+(1,1))] \\
        &= [\Delta_{m,n}(f_1,f_2), \Delta_{m,n}(f_1+1,f_2+1)]\\
        &= [(mf_1,nf_2),(m(f_1+1),n(f_2+1))]
    \end{align*}
    Since $u \leq f < v$, then $a \leq f_1 < c$ and $b \leq f_2 < d$. So then $ma \leq mf_1 < mc$ and also $nb \leq nf_2 < nd$, and therefore $\Delta_{m,n}(u) \leq \Delta_{m,n}(f) < \Delta_{m,n}(v)$. Since the cell $[\Delta_{m,n}(f),\Delta_{m,n}(f)+(1,1)]\in F \otimes \mathcal{R}_{m,n}$, then $[\Delta_{m,n}(f),\Delta_{m,n}(f)+(1,1)]\notin \mathcal{P} \otimes \mathcal{R}_{m,n}$. Thus $\mathcal{I} \otimes \mathcal{R}_{m,n} = [\Delta_{m,n}(u),\Delta_{m,n}(v)]$ is not an inner interval of $\mathcal{P} \otimes \mathcal{R}_{m,n}$.
\end{proof}

The above lemmas give us that dilations preserve vertices (Lemma \ref{Lemma:Dilated vertices are in dilation}), inner intervals (Lemma \ref{Lemma: Dilations preserve inner intervals}), and intervals that are not inner intervals. We now give a corollary to Lemma \ref{Lemma:Tensor products preserve subpolyominoes} which states that dilations also preserve subpolyominoes.

\begin{corollary} \label{Cor: Dilations preserve subpolyominoes}
   Let $\mathcal{P}$ be a polyomino and let $\mathcal{Q} \subseteq \mathcal{P}$ be a subpolyomino of $\mathcal{P}$. Then for any $m,n\in \mathbb{Z}^+$, $\mathcal{Q}\otimes\mathcal{R}_{m,n} \subseteq \mathcal{P}\otimes \mathcal{R}_{m,n}$.   
\end{corollary}

\begin{proof}
    Result follows directly from Lemma \ref{Lemma:Tensor products preserve subpolyominoes}.
\end{proof}

In a similar fashion to our earlier remark, this means for any polyomino $\mathcal{P}$ and for any cell $C\in\mathcal{P}$, we have that $C\otimes\mathcal{R}_{m,n}$ is a cell interval of rank $mn$ contained in $\mathcal{P}\otimes\mathcal{R}_{m,n}$. We will refer to $C\otimes\mathcal{R}_{m,n}$ as an $(m,n)$-\textit{dilated cell} of $\mathcal{P}\otimes\mathcal{R}_{m,n}$. As well any block $\mathcal{B}\subseteq\mathcal{P}$ corresponds to the cell interval $\mathcal{B}\otimes\mathcal{R}_{m,n}$ in $\mathcal{P}\otimes\mathcal{R}_{m,n}$ and we will refer to this cell interval as an $(m,n)$-\textit{dilated block} of $\mathcal{P}\otimes\mathcal{R}_{m,n}$. In the specific context of dilations, we can now also state the following corollary of Lemma \ref{Lemma:3.1.3}.

\begin{corollary} \label{Corollary: Dilation is union of dilated cells}
     Let $\mathcal{P}$ be a polyomino and let $\mathcal{R}_{m,n}$ be a rectangle polyomino for some $m,n\in \mathbb{Z}^+$. Then $\mathcal{P}\otimes\mathcal{R}_{m,n} = \bigcup\limits_{C\in \mathcal{P}}C\otimes\mathcal{R}_{m,n}$.
\end{corollary}

\begin{proof}
    Result follows directly from Lemma \ref{Lemma:3.1.3}.
\end{proof}

We are now ready to define the novel class of polyominoes that we will study for the remainder of this chapter.

\begin{definition} \label{Defn:dilated closed path}
    Let $\mathcal{P}$ be a closed path polyomino as per Definition \ref{Defn: closed path} and let $\mathcal{R}_{m,n}$ be a rectangle polyomino for some $m,n \in \mathbb{Z}^+$. If $\mathcal{Q}$ is a polyomino such that $\mathcal{Q} = P \otimes \mathcal{R}_{m,n}$, then we say that $\mathcal{Q}$ is a \textit{dilated closed path}.
\end{definition}

\begin{figure} [H]
    \centering
    \begin{tikzpicture} [scale = 0.8]
        \draw[step=1cm,white,very thin] (0,0) grid (17,6);
        \polyomino[
        empty cell =x,
        grid,
        p={a}{style={lightgray,draw=black}},
        p={b}{style={lightgray,draw=black}},
        row sep =;
        ]{
        x a a a x x x x x a a a a a a x x;
        x a x a a x x x x a a x x a a a a;
        a a x x a x x a a a a x x x x a a;
        a x x a a x x a a x x x x a a a a;
        a b a a x x x a b a a a a a a x x;
        x x x x x x x x x x x x x x x x x
        }
        \begin{scope}[color=black]
            \node[anchor=center] () at (2.5,0.5){(a)};
            \node[anchor=center] () at (12,0.5){(b)};
            \node[anchor=center] () at (0.5,5.5){$\mathcal{P}$};
            \node[anchor=center] () at (7.5,5.5){$\mathcal{P} \otimes \mathcal{R}_{2,1}$};
            
        \end{scope}
        
    \end{tikzpicture}
    \caption{(a) A closed path $\mathcal{P}$ (b) The dilated closed path $\mathcal{P} \otimes \mathcal{R}_{2,1}$}
    \label{fig:dilated closed path}
\end{figure}
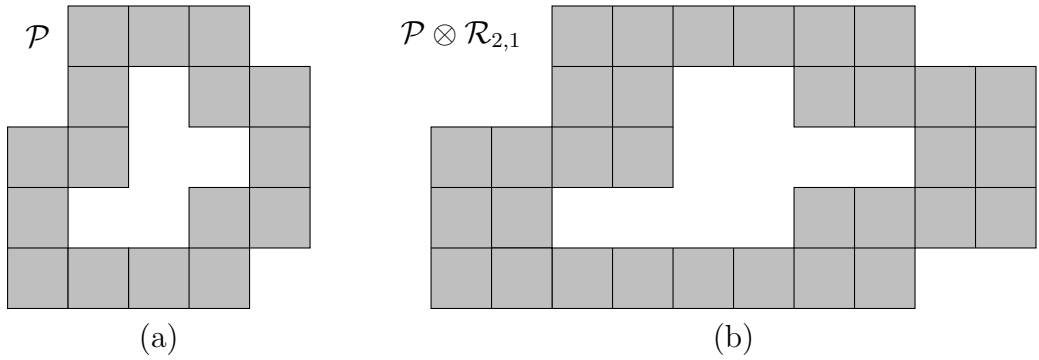

\begin{proposition} \label{Proposition: Dilated holes are holes}
    Let $\mathcal{P}$ be a polyomino and let $m,n \in \mathbb{Z}^+$. If $\mathcal{H}$ is a hole of $\mathcal{P}$, then $\mathcal{H} \otimes \mathcal{R}_{m,n}$ is a hole of $\mathcal{P} \otimes \mathcal{R}_{m,n}$.
\end{proposition}

\begin{proof}
    Let $\mathcal{H}$ be a hole of $\mathcal{P}$. Since we can view $\mathcal{H}$ as a simple polyomino in $\mathbb{N}^2$, then the tensor product $\mathcal{H} \otimes \mathcal{R}_{m,n}$ is well-defined. We proceed to show that $\mathcal{H} \otimes \mathcal{R}_{m,n}$ is a hole of $\mathcal{P} \otimes \mathcal{R}_{m,n}$. As $\mathcal{H}$ is a hole of $\mathcal{P}$, then it only contains finitely many cells, thus $|\mathcal{H}| = k$ for some $k \in \mathbb{Z}^+$. Therefore $|\mathcal{H} \otimes \mathcal{R}_{m,n}| = |\mathcal{H}||\mathcal{R}_{m,n}| = kmn$ and so $\mathcal{H} \otimes \mathcal{R}_{m,n}$ is a finite collection of cells. Let $F \in \mathcal{H} \otimes \mathcal{R}_{m,n}$ be arbitrary, we now show that $F \notin \mathcal{P} \otimes \mathcal{R}_{m,n}$. Since $F \in \mathcal{H} \otimes \mathcal{R}_{m,n}$, then by Corollary \ref{Corollary: Dilation is union of dilated cells}, $F \in E \otimes \mathcal{R}_{m,n}$ for some $E \in \mathcal{H}$. Since $E \in \mathcal{H}$, then $E \notin \mathcal{P}$. So for any $A \in \mathcal{P}$, we have that $E \neq A$, and therefore by Lemma \ref{Lemma:Dilations preserve disjoint sets}, it must be that $|(E \otimes \mathcal{R}_{m,n}) \cap (A \otimes \mathcal{R}_{m,n})| = 0$ for all $A \in \mathcal{P}$. Then
    \[
    \left|(E \otimes \mathcal{R}_{m,n}) \bigcap \left(\bigcup \limits_{A\in \mathcal{P}}A \otimes \mathcal{R}_{m,n}\right)\right| = 0
    \] and so by Corollary \ref{Corollary: Dilation is union of dilated cells}, we have that
    \[
    |(E \otimes \mathcal{R}_{m,n}) \cap (\mathcal{P} \otimes \mathcal{R}_{m,n})| = 0.
    \] So $F \notin \mathcal{P} \otimes \mathcal{R}_{m,n}$. As $F \in \mathcal{H} \otimes \mathcal{R}_{m,n}$ was arbitrary, this means all cells contained in $\mathcal{H} \otimes \mathcal{R}_{m,n}$ are not contained in $\mathcal{P} \otimes \mathcal{R}_{m,n}$. Now let $G_1, G_2 \in \mathcal{H} \otimes \mathcal{R}_{m,n}$ be arbitrary. As $\mathcal{H}$ is itself a polyomino, then by Proposition \ref{Proposition: Dilations of polyos are polyos}, $\mathcal{H} \otimes \mathcal{R}_{m,n}$ is also a polyomino and therefore $G_1$ and $G_2$ are connected in $\mathcal{H} \otimes \mathcal{R}_{m,n}$. As $\mathcal{H}$ is a hole  of $\mathcal{P}$ then it is maximal with respect to set inclusion, so its rank $|\mathcal{H}| = k$ is maximal. Therefore the maximal rank of $\mathcal{H} \otimes \mathcal{R}_{m,n}$ is $kmn$. If it were the case that $|\mathcal{H} \otimes \mathcal{R}_{m,n}| < kmn$ then $|\mathcal{H}||{R}_{m,n}| < kmn$, which would give $kmn < kmn$, but this is a contradiction. Thus it must be that $|\mathcal{H} \otimes \mathcal{R}_{m,n}| =kmn$ and so $\mathcal{H} \otimes \mathcal{R}_{m,n}$ is maximal with respect to set inclusion. Thus we have that $\mathcal{H} \otimes \mathcal{R}_{m,n}$ is a finite, connected collection of cells not contained in $\mathcal{P} \otimes \mathcal{R}_{m,n}$ that is maximal with respect to set inclusion and so by definition, $\mathcal{H} \otimes \mathcal{R}_{m,n}$ is a hole of $\mathcal{P} \otimes \mathcal{R}_{m,n}$.
\end{proof}

\begin{corollary} 
    Let $\mathcal{P}$ be a polyomino and let $m,n \in \mathbb{Z}^+$. If $\mathcal{P}$ is non-simple then $\mathcal{P} \otimes \mathcal{R}_{m,n}$ is non-simple.
\end{corollary}

\begin{proof}
    If $\mathcal{P}$ is non-simple then it contains at least one hole, say $\mathcal{H}$. By Proposition \ref{Proposition: Dilated holes are holes}, $\mathcal{H} \otimes \mathcal{R}_{m,n}$ is a hole of $\mathcal{P} \otimes \mathcal{R}_{m,n}$ and so $\mathcal{P} \otimes \mathcal{R}_{m,n}$ is non-simple.
\end{proof}

In Proposition 3.4 of \textbf{\cite{Cisto_Navarra:2023}}, Cisto and Navarra characterize closed paths by proving that closed paths are non-simple polyominoes that contain a unique hole. By the above proposition we have that dilated closed paths are non-simple polyominoes. We proceed to show that dilated closed paths also contain a unique hole. In order to show this we will use the following two lemmas that appear in \textbf{\cite{Cisto_Navarra:2023}} which characterize the geometry of closed paths.

\begin{lemma} [\textbf{\cite{Cisto_Navarra:2023}}, Remark 3.2] \label{Lemma:Closed path geometry}
    Let $\mathcal{P}$ be a closed path and let $C,D,E,F \notin \mathcal{P}$. Then up to rotations and reflections, there are only two possible configurations of a sequence of three cells $A_{i-1}$, $A_i$, and $A_{i+1}$ in $\mathcal{P}$ as shown in Figure \ref{Fig: three cells in closed path} and each of these configurations necessarily appear in any closed path. Furthermore, for any cell $G \notin \mathcal{P}$ and for any cell $A \in \mathcal{P}$ such that $A$ is not positioned as $A_i$ is in Figure \ref{Fig: three cells in closed path} (I), there exists a path of cells $G = G_1,...,G_k$ not in $\mathcal{P}$ such that $G_k \cap A$ is a common edge of $G_k$ and $A$.

    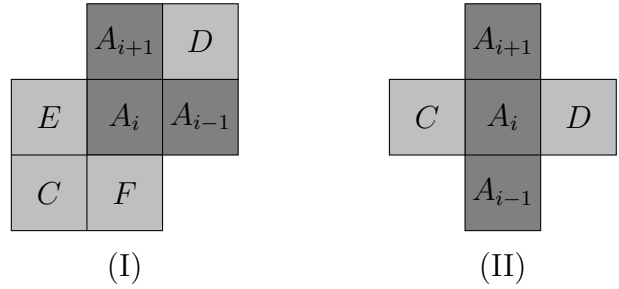
\begin{figure}[H] 
    \centering
    \begin{tikzpicture}
        \draw[step=1cm,white,very thin] (0,0) grid (8,4);
        \polyomino[
        empty cell =x,
        grid,
        p={a}{style={gray,draw=black}},
        p={f}{style={lightgray,draw=black}},
        row sep =;
        ]{
        x a f x x x a x;
        f a a x x f a f;
        f f x x x x a x;
        x x x x x x x x
        }

         \begin{scope}[color=black]
            \node[anchor=center] () at (1.5,0.5){(I)};
            \node[anchor=center] () at (6.5,0.5){(II)};
            
            \node[anchor=center] () at (0.5,1.5){$C$};
            \node[anchor=center] () at (2.5,3.5){$D$};
            \node[anchor=center] () at (0.5,2.5){$E$};
            \node[anchor=center] () at (1.5,1.5){$F$};
            \node[anchor=center] () at (2.5,2.5){$A_{i-1}$};
            \node[anchor=center] () at (1.5,2.5){$A_{i}$};
            \node[anchor=center] () at (1.5,3.5){$A_{i+1}$};

            \node[anchor=center] () at (6.5,1.5){$A_{i-1}$};
            \node[anchor=center] () at (6.5,2.5){$A_i$};
            \node[anchor=center] () at (6.5,3.5){$A_{i+1}$};
            \node[anchor=center] () at (5.5,2.5){$C$};
            \node[anchor=center] () at (7.5,2.5){$D$};
        \end{scope}
    
    \end{tikzpicture}
    \caption{The only possible configurations of a sequence of three cells in a closed path.}
    \label{Fig: three cells in closed path}
    \end{figure}
\end{lemma}

\begin{lemma} [\textbf{\cite{Cisto_Navarra:2023}}, Lemma 3.3] \label{Lemma:closed path has a 3 block}
    If $\mathcal{P}$ is a closed path then $\mathcal{P}$ contains a block of length at least 3.
\end{lemma}

We state now a useful lemma that shows that subtracting a polyomino from another
and then dilating the result is the same as dilating both polyominoes first and then
performing the subtraction.

\begin{lemma} \label{Lemma:dilation subtraction = subtraction dilation}
    Let $\mathcal{P}$ be a polyomino, let $\mathcal{Q}$ be a subpolyomino of $\mathcal{P}$ and let $m,n \in \mathbb{Z}^+$. Then the collection of cells $(\mathcal{P} \otimes \mathcal{R}_{m,n}) \setminus (Q \otimes \mathcal{R}_{m,n}) = (\mathcal{P} \setminus \mathcal{Q}) \otimes \mathcal{R}_{m,n}$.
\end{lemma}

\begin{proof}
    Let $C' \in (\mathcal{P} \otimes \mathcal{R}_{m,n}) \setminus (Q \otimes \mathcal{R}_{m,n})$ be arbitrary, then $C' \in \mathcal{P} \otimes \mathcal{R}_{m,n}$ and $C' \notin Q \otimes \mathcal{R}_{m,n}$. Since $C' \in \mathcal{P} \otimes \mathcal{R}_{m,n}$, then $C' \in C \otimes \mathcal{R}_{m,n}$ for some $C \in \mathcal{P}$ by Lemma \ref{Lemma:3.1.3}. Suppose $C \in \mathcal{Q}$, then by Lemma \ref{Lemma:Tensor products preserve subpolyominoes}, $C' \in C \otimes \mathcal{R}_{m,n} \subseteq Q \otimes \mathcal{R}_{m,n}$, but this is a contradiction, therefore $C \notin \mathcal{Q}$. Then $C \in \mathcal{P} \setminus \mathcal{Q}$, so $C' \in C \otimes \mathcal{R}_{m,n} \subseteq (\mathcal{P} \setminus \mathcal{Q}) \otimes \mathcal{R}_{m,n}$. As C' was arbitrary, we have that $(\mathcal{P} \otimes \mathcal{R}_{m,n}) \setminus (Q \otimes \mathcal{R}_{m,n}) \subseteq (\mathcal{P} \setminus \mathcal{Q}) \otimes \mathcal{R}_{m,n}$. Now, let $D' \in (\mathcal{P} \setminus \mathcal{Q}) \otimes \mathcal{R}_{m,n}$ be arbitrary. Then by Lemma \ref{Lemma:3.1.3}, $D' \in D \otimes \mathcal{R}_{m,n}$ for some $D \in \mathcal{P} \setminus \mathcal{Q}$, meaning that $D \in \mathcal{P}$ and $D \notin \mathcal{Q}$. As $D \in \mathcal{P}$, then by Lemma \ref{Lemma:Tensor products preserve subpolyominoes}, we have that $D' \in D \otimes \mathcal{R}_{m,n} \subseteq \mathcal{P} \otimes \mathcal{R}_{m,n}$. As $D \notin \mathcal{Q}$, then $|(D \otimes \mathcal{R}_{m,n}) \cap (\mathcal{Q} \otimes \mathcal{R}_{m,n})| = 0$ by Lemma \ref{Lemma:Dilations preserve disjoint sets}. Thus, $D' \notin \mathcal{Q} \otimes \mathcal{R}_{m,n}$. So $D' \in (\mathcal{P} \otimes \mathcal{R}_{m,n}) \setminus (Q \otimes \mathcal{R}_{m,n})$. Since $D'$ was arbitrary, then we have that ($\mathcal{P} \setminus \mathcal{Q}) \otimes \mathcal{R}_{m,n} \subseteq (\mathcal{P} \otimes \mathcal{R}_{m,n}) \setminus (Q \otimes \mathcal{R}_{m,n})$. Thus, $(\mathcal{P} \otimes \mathcal{R}_{m,n}) \setminus (Q \otimes \mathcal{R}_{m,n}) = (\mathcal{P} \setminus \mathcal{Q}) \otimes \mathcal{R}_{m,n}$.
\end{proof}

\begin{proposition} \label{Prop:dilated closed path has unique hole}
    If $\mathcal{Q} = \mathcal{P} \otimes \mathcal{R}_{m,n}$ is a dilated closed path, then $\mathcal{Q}$ is non-simple and $\mathcal{Q}$ contains exactly one hole. 
\end{proposition}

\begin{proof}
    By Proposition 3.4 in \textbf{\cite{Cisto_Navarra:2023}}, we have that since $\mathcal{P}$ is a closed path, then $\mathcal{P}$ contains exactly one hole, say $\mathcal{H}$. By Proposition \ref{Proposition: Dilated holes are holes}, we have that $\mathcal{H} \otimes \mathcal{R}_{m,n}$ is a hole of $\mathcal{Q}$, and so $\mathcal{Q}$ is non-simple. We claim that $\mathcal{H} \otimes \mathcal{R}_{m,n}$ is the unique hole of $\mathcal{Q}$. We proceed by a similar argument to that in the proof of Proposition 3.4 in \textbf{\cite{Cisto_Navarra:2023}}. Assume that $\mathcal{H}_1$ and $\mathcal{H}_2$ are two distinct holes of $\mathcal{Q}$. Then there exist three cells $E' \in \mathcal{H}_1$, $F' \in \mathcal{H}_2$, and $G' \in \ext \mathcal{Q}$ such that there exists no walk of cells not in $\mathcal{Q}$ that connects any pair of the three cells. Let $k, \ell \in \mathbb{Z}^+$ and let $\mathcal{R}_{km,\ell n}$ be the rectangle polyomino of minimal rank such that $\mathcal{Q} \subseteq \mathcal{R}_{km,\ell n}$ and $E', F', G' \in \mathcal{R}$. We note that $\mathcal{R}_{km,\ell n} = \mathcal{R}_{k,\ell} \otimes \mathcal{R}_{m,n}$. Since $E', F', G' \notin \mathcal{Q}$, then we have that $E', F', G' \in (\mathcal{R}_{km,\ell n} \setminus \mathcal{Q})$. So $E', F', G' \in (\mathcal{R}_{k,\ell} \otimes \mathcal{R}_{m,n}) \setminus (\mathcal{P} \otimes \mathcal{R}_{m,n})$. Therefore, by Lemma \ref{Lemma:dilation subtraction = subtraction dilation}, we have that $E', F', G' \in (\mathcal{R}_{k,\ell} \setminus \mathcal{P}) \otimes \mathcal{R}_{m,n}$. So by Corollary \ref{Corollary: Dilation is union of dilated cells}, we have that $E' \in E \otimes \mathcal{R}_{m,n}$, $F' \in F \otimes \mathcal{R}_{m,n}$, and $G' \in G \otimes \mathcal{R}_{m,n}$ for some $E, F, G\in (\mathcal{R}_{k,\ell} \setminus \mathcal{P})$. So as such, $E,F,G \notin \mathcal{P}$. By Lemma \ref{Lemma:closed path has a 3 block}, we know that $\mathcal{P}$ contains a block of length 3, say $\mathcal{B}$. We relabel the cells of $\mathcal{P}$ such that $\mathcal{B} = [A_1, A_3]$. Then $A_2$ is in the same position as $A_i$ in Figure \ref{Fig: three cells in closed path} (II). As $E, F, G\notin \mathcal{P}$, then by Lemma \ref{Lemma:Closed path geometry}, there exist three walks $\mathcal{W}_E: E = E_1,...,E_r$, $\mathcal{W}_F: F = F_1,...,F_s$, and $\mathcal{W}_G: G = G_1,...,G_t$ of cells not in $\mathcal{P}$ such that $E_r \cap A_2 = \{e\}$, $F_s \cap A_2 = \{f\}$, and $G_t \cap A_2 = \{g\}$ for some edges $e,f,g \in E(A_2)$. We know that $A_1 \cap A_2 = \{e_1\}$ and $A_2 \cap A_3 = \{e_2\}$ for some distinct edges $e_1,e_2 \in E(A_2)$. Since $E_r \notin \mathcal{P}$, then $E_r \neq A_1 \neq A_3$, and as such $e \neq e_1 \neq e_2$. For similar reasons, we have that $f \neq e_1 \neq e_2$ and $g \neq e_1 \neq e_2$. Since, by the definition of a cell, $|E(A_2)| = 4$, then it must be that either $e = f$, $f = g$, or $g = e$. Suppose $e = f$. Then $E_r = F_s$, and so there exists a walk $\mathcal{W}_{E \rightarrow F}: E = E_1,...,E_{r-1}, E_r = F_s, F_{s-1},...,F_1 = F$ of cells not in $\mathcal{P}$ that connects $E$ and $F$. Therefore, by Corollary \ref{Cor: dilations preserve unconnectedness}, $E'$ and $F'$ are connected by a walk of cells not in $\mathcal{P}$, but this is a contradiction. If instead we have that $f = g$ or $g = e$, we would obtain an analogous contradiction. Thus, it must be that $\mathcal{H} \otimes \mathcal{R}_{m,n}$ is the unique hole of $\mathcal{Q}$.
\end{proof}

\section{Primality of dilated closed paths} \label{Section:4}
As we have shown that dilated closed paths are non-simple polyominoes, there is reason to ask questions about their primality. Primarily, the big question is: \textit{does the Zig-Zag Walk Conjecture hold for the class of dilated closed paths?} Our main result at the end of this section will answer this question in the affirmative (see Theorem \ref{Thm: ZZW for dilated closed paths}). As all dilated closed paths contain exactly one hole by Proposition \ref{Prop:dilated closed path has unique hole}, then this result will bring us closer to classifying all polyominoes with exactly one hole whose polyomino ideal is prime. This was first stated as a point of research interest by Qureshi, Shibuta, and Shikama in \textbf{\cite{Qureshi_Shibuta_Shikama:2017}}, but it has yet to be resolved. As our initial step towards our main result, we show that zig-zag walks are preserved by dilations with the below proposition. We first give a short remark.

\begin{remark} \label{Remark:About rectangles with two common verts}
    If two rectangle polyominoes share at least two common vertices, then they must share at least one common edge.
\end{remark}

\begin{proposition} \label{Prop: Construct a zig-zag}
    Let $\mathcal{P}$ be a polyomino and let $\mathcal{Q} = \mathcal{P} \otimes \mathcal{R}_{m,n}$ for some $m,n \in \mathbb{Z}^+$. If $\mathcal{Q}$ contains no zig-zag walks, then $\mathcal{P}$ contains no zig-zag walks.
\end{proposition}

\begin{proof}
    We proceed with a proof by contraposition. Assume that $\mathcal{P}$ contains a zig-zag walk, say $\mathcal{W}: I_1,...,I_\ell$. We claim that $W_\Delta:I_1 \otimes \mathcal{R}_{m,n},...,I_\ell \otimes \mathcal{R}_{m,n}$ is a zig-zag walk in $\mathcal{Q}$. Since $I_1,...,I_\ell$ are required to be inner intervals, then by Lemma \ref{Lemma: Dilations preserve inner intervals}, we have that $I_1 \otimes \mathcal{R}_{m,n},...,I_\ell \otimes \mathcal{R}_{m,n}$ are inner intervals. Without loss of generality, we apply any necessary rotations to $\mathcal{P}$ such that $v_1$ is an anti-diagonal corner of $I_1$. Then by definition of a zig-zag walk (see Definition \ref{Defn:ZZW}), we have that:
    \[
     I_k = \begin{cases}
                [u_k, v_{k+1}] \text{ or } [v_{k+1}, u_k] \text{, if }k = 1 \bmod 2 \\
                [v_k, z_{k}] \text { or } [z_k, v_{k}] \text{, if }k = 0 \bmod 2 \\      
            \end{cases}
    \] 
  and so by Lemma \ref{Lemma: Dilations preserve inner intervals}, we must have that:
  \[
     I_k \otimes \mathcal{R}_{m,n} = 
            \begin{cases}
                [\Delta_{m,n}(u_k), \Delta_{m,n}(v_{k+1})] \text{ or } [\Delta_{m,n}(v_{k+1}), \Delta_{m,n}(u_k)] \text{, if }k = 1 \bmod 2 \\
                [\Delta_{m,n}(v_k), \Delta_{m,n}(z_{k})] \text { or } [\Delta_{m,n}(z_k), \Delta_{m,n}(v_{k})] \text{, if }k = 0 \bmod 2 \\      
            \end{cases}
    \]
We proceed to show that conditions (i), (ii), and (iii) of Definition \ref{Defn:ZZW} hold for $\mathcal{W}_\Delta$. By condition (i) for $\mathcal{W}$ as a zig-zag walk, we have that $I_1 \cap I_\ell = \{v_1\}$ and $I_i \cap I_{i+1} = \{v_{i+1}\}$ for $i \in \{1,...,\ell-1\}$. Let $i \in \{1,...,\ell-1\}$ and let $v_{i+1} \in I_i \cap I_{i+1}$, then $v_{i+1}\in I_i$ and $v_{i+1}\in I_{i+1}$. So by Lemma \ref{Lemma:Dilated vertices are in dilation}, we have that $\Delta_{m,n}(v_{i+1}) \in I_i \otimes \mathcal{R}_{m,n}$ and $\Delta_{m,n}(v_{i+1}) \in I_{i+1} \otimes \mathcal{R}_{m,n}$, respectively. So then $\Delta_{m,n}(v_{i+1}) \in (I_i \otimes \mathcal{R}_{m,n}) \cap (I_{i+1} \otimes \mathcal{R}_{m,n})$. We note that by Lemma \ref{Lemma: Dilations preserve inner intervals}, we have that $I_i \otimes \mathcal{R}_{m,n}$ and $I_{i+1} \otimes \mathcal{R}_{m,n}$ are both inner intervals of $\mathcal{Q}$ and so by definition of an inner interval they are both rectangle subpolyominoes of $\mathcal{Q}$. Suppose there exists another vertex $w\in (I_i \otimes \mathcal{R}_{m,n}) \cap (I_{i+1} \otimes \mathcal{R}_{m,n})$ with $w \neq \Delta_{m,n}(v_{i+1})$. By Remark \ref{Remark:About rectangles with two common verts}, we have that as $I_i \otimes \mathcal{R}_{m,n}$ and $I_{i+1} \otimes \mathcal{R}_{m,n}$ share two common vertices, then they must share a common edge, say $e$. Therefore, for any cell $C \in I_i \otimes \mathcal{R}_{m,n}$ and $D \in I_{i+1} \otimes \mathcal{R}_{m,n}$, there exists a walk that crosses the edge $e$ which connects $C$ and $D$. So by definition, we have that $(I_i \otimes \mathcal{R}_{m,n}) \cup (I_{i+1} \otimes \mathcal{R}_{m,n})$ is a polyomino. By Corollary \ref{Cor: Union of tensor = tensor of union}, we know that $(I_i \otimes \mathcal{R}_{m,n}) \cup (I_{i+1} \otimes \mathcal{R}_{m,n}) = (I_i \cup I_ {i+1})\otimes \mathcal{R}_{m,n}$, so then $(I_i \cup I_ {i+1})\otimes \mathcal{R}_{m,n}$ is a polyomino. Therefore, $I_i \cup I_{i+1}$ is a polyomino by definition of the tensor product, but this is a contradiction because $I_i \cap I_{i+1} = \{v_{i+1}\}$, so their union cannot be a polyomino as they do not have a common edge and as such $I_i \cup I_{i+1}$ is only weakly connected. Thus, it must be that $(I_i \otimes \mathcal{R}_{m,n}) \cap (I_{i+1} \otimes \mathcal{R}_{m,n}) = \{\Delta_{m,n}(v_{i+1})\}$. By a similar argument, we have that since $I_1 \cap I_\ell = \{v_1\}$, then $(I_1 \otimes \mathcal{R}_{m,n}) \cap (I_{\ell} \otimes \mathcal{R}_{m,n}) = \{\Delta_{m,n}(v_1)\}$. So condition (i) holds for $\mathcal{W}_\Delta$.
    
 By condition (ii) for $\mathcal{W}$ as a zig-zag walk, we have that $v_1$ and $v_\ell$ lie on the same edge interval and that $v_i$ and $v_{i+1}$ lie on the same edge interval in $\mathcal{P}$ for $i\in \{1,...,\ell-1\}$. Suppose that $v_i$ and $v_{i+1}$ lie on the same horizontal edge interval in $\mathcal{P}$. Then $v_i = (a,h)$ and $v_{i+1} = (b,h)$ for some $a,b,h \in \mathbb{N}$. Therefore $\Delta_{m,n}(v_i) = \Delta_{m,n}((a, h)) = (ma, nc)$ and $\Delta_{m,n}(v_{i+1}) = \Delta_{m,n}((b, h)) = (mb, nc)$, and so $\Delta_{m,n}(v_i)$ and $\Delta_{m,n}(v_{i+1})$ lie on the same horizontal edge interval in $\mathcal{Q}$. By a symmetric argument, if $v_i$ and $v_{i+1}$ lie on the same vertical edge interval in $\mathcal{P}$, then $\Delta_{m,n}(v_i)$ and $\Delta_{m,n}(v_{i+1})$ must lie on the same vertical edge interval in $\mathcal{Q}$. So condition (ii) holds for $\mathcal{W}_\Delta$.
 
 By condition (iii) for $\mathcal{W}$ as a zig-zag walk, we have that for any $i,j \in \{1,...,\ell\}$, with $i \neq j$, there exists no inner interval $I$ of $\mathcal{P}$ such that $z_i,z_j \in I$. Suppose that $\Delta_{m,n}(z_i),\Delta_{m,n}(z_j) \in J$ for some inner interval $J$ in $\mathcal{Q}$ with $i \neq j$. Let $J' \subseteq J$ be the inner interval in $\mathcal{Q}$ of minimal rank that contains $\Delta_{m,n}(z_i)$ and $\Delta_{m,n}(z_j)$. Then $\Delta_{m,n}(z_i)$ and $\Delta_{m,n}(z_j)$ are both corners of $J'$. Suppose $\Delta_{m,n}(z_i)$ and $\Delta_{m,n}(z_j)$ lie on the same edge interval in $\mathcal{Q}$, then $z_i$ and $z_j$ would lie on the same edge interval in $\mathcal{P}$, but this contradicts that $\mathcal{W}$ is a zig-zag walk in $\mathcal{P}$, since for each pair of consecutive inner intervals $I_i$ and $I_{i+1}$ in $\mathcal{W}$ with $i\in \{1,...,\ell\}$, we have that one of $z_i$ and $z_{i+1}$ is a diagonal corner of its inner interval while the other is not, and also we have that the pairs of vertices $z_i$ and $v_{i+1}$ and $v_i$ and $v_{i+1}$ lie on the same edge intervals, but $z_i$ is always either diagonal or anti-diagonal to $v_i$. Therefore, $\Delta_{m,n}(z_i)$ and $\Delta_{m,n}(z_j)$ must either both be diagonal corners or both anti-diagonal corners of $J'$. Suppose $\Delta_{m,n}(z_i)$ and $\Delta_{m,n}(z_j)$ are the diagonal corners of $J'$. Then $J' = [\Delta_{m,n}(z_i),\Delta_{m,n}(z_j)]$ or $J' = [\Delta_{m,n}(z_j),\Delta_{m,n}(z_i)]$, and so by Lemma \ref{Lemma: Dilations preserve non-inners}, we have that $[z_i,z_j]$ or $[z_j,z_i]$, respectively, is an inner interval of $\mathcal{P}$, but this contradicts $\mathcal{W}$ being a zig-zag walk in $\mathcal{P}$. Now suppose $\Delta_{m,n}(z_i)$ and $\Delta_{m,n}(z_j)$ are the anti-diagonal corners of $J'$. Let $a$ and $b$ be the diagonal corners of $J'$ such that $a$ lies on the same vertical edge interval as $\Delta_{m,n}(z_i)$ and such that $b$ lies on the same vertical edge interval as $\Delta_{m,n}(z_j)$. Then it must be that $a$ and $b$ lie on the same horizontal edge intervals as $\Delta_{m,n}(z_j)$ and $\Delta_{m,n}(z_i)$, respectively. This forces $a = \Delta_{m,n}(w_a)$ and $b = \Delta_{m,n}(w_b)$ for some $w_a,w_b \in V(\mathcal{P})$. So as $J' = [a,b]$ or $J' = [b,a]$, then by Lemma \ref{Lemma: Dilations preserve non-inners} we have that either $[w_a,w_b]$ or $[w_b,w_a]$ is an inner interval of $\mathcal{P}$ with anti-diagonal corners $z_i$ and $z_j$, but this contradicts $\mathcal{W}$ being a zig-zag walk in $\mathcal{P}$. Thus, there is no inner interval $J$ in $\mathcal{Q}$ that contains $z_i$ and $z_j$ with $i \neq j$. So condition (iii) holds for $\mathcal{W}_\Delta$. Therefore, $W_\Delta:I_1 \otimes \mathcal{R}_{m,n},...,I_\ell \otimes \mathcal{R}_{m,n}$ is a zig-zag walk in $\mathcal{Q}$.
\end{proof}

We seek to generalize the results by Cisto and Navarra in \textbf{\cite{Cisto_Navarra:2023}} on the primality of closed paths to the class of dilated closed paths.  In following their argument, we first state a convenient definition and lemma introduced by Shikama in \textbf{\cite{Shikama:2018}}. Recall that by $S$ we mean the polynomial ring which contains the polyomino ideal $I_{\mathcal{P}}$ for any polyomino $\mathcal{P}$, that is $S = K[x_v \mid v \in V(\mathcal{P})]$. For a binomial $f$ in a binomial ideal $J \subset S$, we can write $f = f^+-f^-$ to distinguish the two terms $f^+$ and $f^-$ of $f$. We denote by $V_f^+$ the subset of vertices $v \in V(\mathcal{P})$ such that $x_v$ divides $f^+$ and by $V_f^-$ the subset of vertices $v \in V(\mathcal{P})$ such that $x_v$ divides $f^-$. As an example, if $f =x_{(0,0)}x_{(1,1)}-x_{(0,1)}x_{(1,0)} \in J \subset S$, then $V_f^+ = \{(0,0),(1,1)\}$ and $V_f^- = \{(0,1),(1,0)\}$. 

\begin{definition}
    Let $f \in J$ for some binomial ideal $J \subset S$. We say that $f$ is \textit{redundant} if it can be expressed as a linear combination of binomials in $J$ of lower degree. Otherwise, we say that $f$ is \textit{irredundant}.
\end{definition}

\begin{lemma} (\textbf{\cite{Cisto_Navarra:2023}}, Lemma 2.2). \label{Lemma: 2.2}
    Let $\mathcal{P}$ be a polyomino, let $T$ be an integral domain, and let $\varphi : S \to T$ be a ring homomorphism. Let $J = \ker \varphi$ and let $f = f^+ - f^-$ be a binomial in $J$ of degree at least $3$. Suppose that $I_{\mathcal{P}} \subseteq J$ and $\varphi(x_a) \neq 0$ for all $a \in V(\mathcal{P})$. If there exist three vertices $p,q \in V_f^+$ and $r \in V_f^-$ such that $p,q$ are diagonal (resp. anti-diagonal) corners of an inner interval $\mathcal{I}$ of $\mathcal{P}$ and $r$ is one of the anti-diagonal (resp. diagonal) corners of $\mathcal{I}$, then $f$ is redundant in $J$. Similarly, if there exist three vertices $p \in V_f^+$ and $r,s \in V_f^-$ such that $p$ is one of the diagonal (resp. anti-diagonal) corners of an inner interval $\mathcal{I}$ of $\mathcal{P}$ and $r,s$ are the anti-diagonal (resp. diagonal) corners of $\mathcal{I}$, then $f$ is redundant in $J$.
\end{lemma}

We note here that we have stated the above lemma in its more general form as it appears in \textbf{\cite{Cisto_Navarra:2023}}. If we are to prove the Zig-Zag Walk Conjecture for dilated closed paths, we need to know what substructures of a dilated closed path might serve as sufficient conditions for its polyomino ideal to be prime. We proceed to show that there is a natural way to generalize the substructures that serve as these primality conditions for closed paths to this new setting of dilated closed paths. We now define these substructures for closed paths as first given by Cisto and Navarra in \textbf{\cite{Cisto_Navarra:2023}}.

\begin{definition}(\textbf{\cite{Cisto_Navarra:2023}}, Definition 3.5).
    Let $\mathcal{P}$ be a polyomino. A path of five cells $A_1,A_2,A_3,A_4,A_5$ of $\mathcal{P}$ is called an \textit{L-configuration} if the two blocks of cells $A_1,A_2,A_3$ and $A_3,A_4,A_5$ go in orthogonal directions, that is one must be a horizontal block while the other is a vertical block. 
\end{definition}

\begin{definition}(\textbf{\cite{Cisto_Navarra:2023}}, Definition 3.9).
    Let $\mathcal{P}$ be a polyomino and let $\mathcal{B} = \{\mathcal{B}_i\}_{i=1,..,s}$ be a set of maximal horizontal (or vertical) blocks of $\mathcal{P}$ such that each $\mathcal{B}_i$ has length at least two and for all $i\in\{1,...,s-1\}$, we have that $V(\mathcal{B}_i)\cap V(\mathcal{B}_{i+1}) = \{a_i,b_i\}$, with $a_i \neq b_i$. We call $\mathcal{B}$ a \textit{ladder of s steps} if for all $i\in\{1,...,s-2\}$, there does not exist an edge interval which contains both the edges $[a_i,b_i]$ and $[a_{i+1},b_{i+1}]$.
\end{definition}

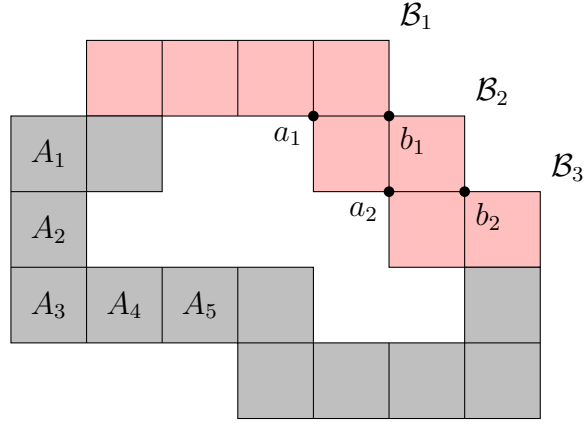
\begin{figure} [h]
    \centering
    \begin{tikzpicture}
        \draw[step=1cm,white,very thin] (0,0) grid (7,5);
        \polyomino[
        empty cell =x,
        grid,
        p={a}{style={lightgray,draw=black}},
        p={b}{style={pink,draw=black}},
        row sep =;
        ]{
        x b b b b x x;
        a a x x b b x;
        a x x x x b b;
        a a a a x x a;
        x x x a a a a
        }
        \begin{scope}[color=black]
            \node[anchor=center] () at (0.5,3.5) {$A_1$};
            \node[anchor=center] () at (0.5,2.5) {$A_2$};
            \node[anchor=center] () at (0.5,1.5) {$A_3$};
            \node[anchor=center] () at (1.5,1.5) {$A_4$};
            \node[anchor=center] () at (2.5,1.5) {$A_5$};

            \fill(5,5) circle[radius=0pt] node[above right]{$\mathcal{B}_1$};
            \fill(6,4) circle[radius=0pt] node[above right]{$\mathcal{B}_2$};
            \fill(7,3) circle[radius=0pt] node[above right]{$\mathcal{B}_3$};

            \fill(4,4) circle[radius=2pt] node[below left]{$a_1$};
            \fill(5,4) circle[radius=2pt] node[below right]{$b_1$};
            \fill(5,3) circle[radius=2pt] node[below left]{$a_2$};
            \fill(6,3) circle[radius=2pt] node[below right]{$b_2$};

        \end{scope}
        
    \end{tikzpicture}
    \caption{A closed path with an L-configuration and a ladder of 3 steps}
    \label{fig:L-configuration example}
\end{figure}

In Theorem 4.2 and Theorem 5.2 of \textbf{\cite{Cisto_Navarra:2023}}, Cisto and Navarra showed that if a closed path contains an L-configuration or a ladder of at least 3 steps, then its polyomino ideal is prime. We now give the natural generalization of an L-configuration and a ladder of $s$ steps in the context of dilated closed paths.

\begin{definition}
    Let $\mathcal{Q} = \mathcal{P} \otimes \mathcal{R}_{m,n}$ be a dilated closed path and let $A_1,A_2,A_3,A_4,A_5$ be a sequence of cells which define an L-configuration in $\mathcal{P}$. We say that the sequence of cell intervals $A_1\otimes\mathcal{R}_{m,n}$, $A_2\otimes\mathcal{R}_{m,n}$, $A_3\otimes\mathcal{R}_{m,n}$, $A_4\otimes\mathcal{R}_{m,n}$, $A_5\otimes\mathcal{R}_{m,n}$ in $\mathcal{Q}$ is an $(m,n)$-\textit{dilated L-configuration}.
\end{definition}

\begin{definition}
    Let $\mathcal{P} = \mathcal{Q} \otimes \mathcal{R}_{m,n}$ be a closed path and let $\mathcal{B} = \{\mathcal{B}_i\}_{i=1,..,s}$ be a set of maximal horizontal (or vertical) blocks of $\mathcal{P}$ that form a ladder of $s$ steps. We say that the set of cell intervals $\mathcal{B}\otimes\mathcal{R}_{m,n} = \{\mathcal{B}_i \otimes \mathcal{R}_{m,n}\}_{i=1,...s}$ in $\mathcal{Q}$ is an $(m,n)$-\textit{dilated ladder of s steps}.
\end{definition}

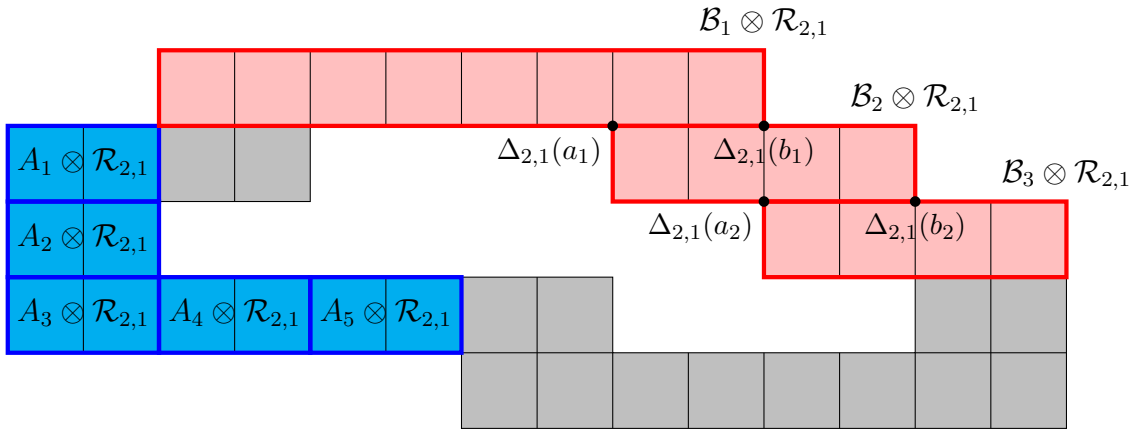
\begin{figure} [H]
    \centering
    \begin{tikzpicture}
        \draw[step=1cm,white,very thin] (0,0) grid (10,4);
        \polyomino[
        empty cell =x,
        grid,
        p={a}{style={lightgray,draw=black}},
        p={b}{style={pink,draw=black}},
        p={c}{style={cyan,draw=black}},
        row sep =;
        ]{
        x x b b b b b b b b x x x x;
        c c a a x x x x b b b b x x;
        c c x x x x x x x x b b b b;
        c c c c c c a a x x x x a a;
        x x x x x x a a a a a a a a
        }
        
         \begin{scope}[color=black]
            \node[anchor=center] () at (1,3.5) {$A_1\otimes\mathcal{R}_{2,1}$};
            \node[anchor=center] () at (1,2.5) {$A_2\otimes\mathcal{R}_{2,1}$};
            \node[anchor=center] () at (1,1.5) {$A_3\otimes\mathcal{R}_{2,1}$};
            \node[anchor=center] () at (3,1.5) {$A_4\otimes\mathcal{R}_{2,1}$};
            \node[anchor=center] () at (5,1.5) {$A_5\otimes\mathcal{R}_{2,1}$};
        \end{scope}

        \draw[draw=blue, ultra thick] (0,1) rectangle (2,2);
        \draw[draw=blue, ultra thick] (0,2) rectangle (2,3);
        \draw[draw=blue, ultra thick] (0,3) rectangle (2,4);
        \draw[draw=blue, ultra thick] (2,1) rectangle (4,2);
        \draw[draw=blue, ultra thick] (4,1) rectangle (6,2);

        \fill(9,5) circle[radius=0pt] node[above right]{$\mathcal{B}_1\otimes \mathcal{R}_{2,1}$};
        \fill(11,4) circle[radius=0pt] node[above right]{$\mathcal{B}_2\otimes \mathcal{R}_{2,1}$};
        \fill(13,3) circle[radius=0pt] node[above right]{$\mathcal{B}_3\otimes \mathcal{R}_{2,1}$};

        \draw[draw=red, ultra thick] (2,4) rectangle (10,5);
        \draw[draw=red, ultra thick] (8,3) rectangle (12,4);
        \draw[draw=red, ultra thick] (10,2) rectangle (14,3);

        \fill(8,4) circle[radius=2pt] node[below left]{\small $\Delta_{2,1}(a_1)$};
        \fill(10,4) circle[radius=2pt] node[below]{\small $\Delta_{2,1}(b_1)$};
        \fill(10,3) circle[radius=2pt] node[below left]{\small $\Delta_{2,1}(a_2)$};
        \fill(12,3) circle[radius=2pt] node[below]{\small $\Delta_{2,1}(b_2)$};
        
    \end{tikzpicture}
    \caption{A dilated closed path, the result of tensoring the polyomino from Figure \ref{fig:L-configuration example} with $\mathcal{R}_{2,1}$, that has a (2,1)-dilated L-configuration and a (2,1)-dilated ladder of 3 steps}
    \label{fig:dilated L-configuration example}
\end{figure}

We claim that $(m,n)$-dilated L-configurations and $(m,n)$-dilated ladders of at least 3 steps are the natural choices for the primality conditions in the class of dilated closed paths. Our aim now is to prove that the containment of either an $(m,n)$-dilated L-configuration or an $(m,n)$-dilated ladder of at least 3 steps is a sufficient condition for a dilated closed path to be prime. After establishing these results we will be able to prove that the Zig-Zag Walk Conjecture holds for the class of dilated closed paths. 

For any polyomino $\mathcal{P}$, we have by its construction that the toric ideal $J_{\mathcal{P}}$ contains the polyomino ideal $I_{\mathcal{P}}$. In fact, as we will now state below, we can say something stronger about this containment by a result from \textbf{\cite{Mascia_Rinaldo_Romeo:2020}}.

\begin{lemma}(\textbf{\cite{Mascia_Rinaldo_Romeo:2020}}, Lemma 3.1). \label{Lemma: ideal containment}
    Let $\mathcal{P}$ be a polyomino and let $(J_{\mathcal{P}})_2$ represent the generators of degree 2 in $J_{\mathcal{P}}$. Then $I_{\mathcal{P}} = (J_{\mathcal{P}})_2$.
\end{lemma}

We will now specify the structure of the toric ideal to the case of dilated closed paths with $(m,n)$-dilated L-configurations. Let $\mathcal{Q} = \mathcal{P} \otimes \mathcal{R}_{m,n}$  be a $(m,n)$-dilated closed path for some $m,n\in\mathbb{Z}^+$ and let 
\[
\mathcal{L}_{\mathcal{Q}}: A_1 \otimes \mathcal{R}_{m,n}, A_2 \otimes \mathcal{R}_{m,n}, A_3 \otimes \mathcal{R}_{m,n}, A_4 \otimes \mathcal{R}_{m,n}, A_5 \otimes \mathcal{R}_{m,n}
\]
be an $(m,n)$-dilated L-configuration in $\mathcal{Q}$. Let $\mathcal{L}_{\mathcal{P}}: A_1, A_2, A_3, A_4, A_5$ be the underlying L-configuration in $\mathcal{P}$ that $\mathcal{L}_{\mathcal{Q}}$ arises from. For the sake of clarity, we set $A_3 = A$. Let $c$ be the lower left corner of $A$. We can assume that $c \leq u$ for any $u\in \bigcup\limits_{i=1}^{5} V(A_i)$, that is to say $A$ is the lowest and leftest cell of $\mathcal{L}_{\mathcal{P}}$, otherwise we can perform the necessary reflections or rotations on $\mathcal{P}$ such that this is the case. In this way, if we let $a = \Delta_{m,n}(c)$ be the lower left corner of the cell interval $A \otimes \mathcal{R}_{m,n}$ in $\mathcal{Q}$, we will have that $a \leq v$ for any $v\in \bigcup\limits_{i=1}^{5} V(A_i \otimes \mathcal{R}_{m,n})$, that is, $A \otimes \mathcal{R}_{m,n}$ is the lowest and leftest cell interval of $\mathcal{L}_{\mathcal{Q}}$.
    
We now consider a specific toric ideal representation of $Q$ as per the definition of such a toric ideal for any polyomino as given in Section \ref{Section:2}. Let $\mathcal{H}$ denote the unique hole of $\mathcal{Q}$ and let $e$ be the lower left corner of $\mathcal{H}$ as defined in Section \ref{Section:2}. Note that given the chosen orientation of $A \otimes \mathcal{R}_{m,n}$ as the lowest and leftest cell interval of $\mathcal{L}_{\mathcal{Q}}$, then for any $b \in V(A \otimes \mathcal{R}_{m,n})$, we have that $b \leq e$. Fixing $K$ to be a field, we define the following map:
\begin{eqnarray*}
\alpha: V(\mathcal{Q}) & \rightarrow & K[\{h_i,v_j,w\}\mid i\in I, j\in J] \\
v & \mapsto & h_iv_jw^k
\end{eqnarray*}
with $v\in V_i\cap H_j$, and with $k = 1$ if $v \in V(A \otimes \mathcal{R}_{m,n})$ and $k = 0$ otherwise. Letting $S = K[x_v \mid v\in V(\mathcal{Q})]$ and letting $T_{\mathcal{Q}} = K[\alpha(v) \mid v \in V(\mathcal{Q})]$ be the toric ring on $\mathcal{Q}$, we define the following surjective ring homomorphism:
\begin{eqnarray*}
\varphi: S & \rightarrow & T_{\mathcal{Q}} \\
            x_v & \mapsto & \alpha(v)
\end{eqnarray*}
We define the toric ideal of $\mathcal{Q}$ to be $J_\mathcal{Q} = \ker \varphi$. This is analogous to its general definition in Section \ref{Section:2}. For the remainder of this section, if $\mathcal{Q}$ is a dilated closed path that contains an $(m,n)$-dilated L-configuration, then when we refer to its toric ideal $J_\mathcal{Q}$, we mean the toric ideal as specified by the above map. The following lemma follows directly from the above lemma.

\begin{lemma} \label{Lemma: L-config ideal containment}
    Let $\mathcal{P}$ be a polyomino and let $\mathcal{Q} = \mathcal{P} \otimes \mathcal{R}_{m,n}$ be a dilated closed path with a $(m,n)$-dilated L-configuration for some $m,n \in \mathbb{Z}^+$. Then $I_{\mathcal{Q}} \subseteq J_{\mathcal{Q}}$.
\end{lemma}

\begin{proof}
    The result follows from Lemma \ref{Lemma: ideal containment} as per the proof given in \textbf{\cite{Mascia_Rinaldo_Romeo:2020}}.
\end{proof}

Our next aim is to prove the reverse containment of Lemma \ref{Lemma: L-config ideal containment}. Before we do so, we make a remark on the structure of dilated closed paths. 

\begin{remark}
    By Definition \ref{Defn: closed path}, we have that for any closed path $\mathcal{P}$, the path cannot split in two directions simultaneously, that is it can never admit the Y-hexomino as subpolyomino (see Figure \ref{Fig:Y-hexomino}). Let $\mathcal{Y}$ denote the Y-hexomino and let $Q = \mathcal{P} \otimes \mathcal{R}_{m,n}$ be a dilated closed path for some $m,n \in \mathbb{Z}^+$. Then for the structure of $\mathcal{Q}$ as a dilated closed path, $\mathcal{Q}$ cannot admit $\mathcal{Y} \otimes \mathcal{R}_{m,n}$  as a subpolyomino.
\end{remark}

\begin{figure} [H] 
    \centering
    \begin{tikzpicture}
        \draw[step=1cm,white,very thin] (0,0) grid (3,3);
        \polyomino[
        empty cell =x,
        grid,
        p={a}{style={lightgray,draw=black}},
        row sep =;
        ]{
        a x a;
        a a a;
        x a x
        }

        \begin{scope}[color=black]
            \node[anchor=center] () at (1.5,0.5) {$A$};
            \node[anchor=center] () at (1.5,1.5) {$B$};
            \node[anchor=center] () at (0.5,1.5) {$C$};
            \node[anchor=center] () at (0.5,2.5) {$D$};
            \node[anchor=center] () at (2.5,1.5) {$E$};
            \node[anchor=center] () at (2.5,2.5) {$F$};
        \end{scope}
        
    \end{tikzpicture}
    \caption{The Y-hexomino}
    \label{Fig:Y-hexomino}
\end{figure}
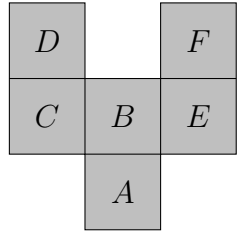

In our proofs to follow in this section, when we claim something ``for the structure of $\mathcal{Q}$ as a dilated closed path'', we are referring to the above remark in that no dilated closed path can admit $\mathcal{Y} \otimes \mathcal{R}_{m,n}$.

We note here that the theorem below is a generalization of Theorem 4.2 in \textbf{\cite{Cisto_Navarra:2023}} for closed paths and as such, our proof follows a similar argument to the proof of that theorem.

\begin{theorem} \label{Thm: 3.2.1}
    Let $\mathcal{Q} = \mathcal{P} \otimes \mathcal{R}_{m,n}$ be a dilated closed path with an $(m,n)$-dilated L-configuration. Then $I_{\mathcal{Q}} = J_{\mathcal{Q}}$.
\end{theorem}

\begin{proof}
    By Lemma \ref{Lemma: L-config ideal containment}, we have that $I_\mathcal{Q} \subseteq J_\mathcal{Q}$. We now prove that $J_\mathcal{Q} \subseteq I_\mathcal{Q}$. We proceed by equivalently proving the following two statements:
    \begin{itemize}
        \item[(1)] every binomial of degree two in $J_\mathcal{Q}$ is in $I_\mathcal{Q}$;
        \item[(2)] every irredundant binomial in $J_\mathcal{Q}$ is of degree 2.
    \end{itemize}

    We let $\mathcal{L}: A_1\otimes\mathcal{R}_{m,n}$, $A_2\otimes\mathcal{R}_{m,n}$, $A_3\otimes\mathcal{R}_{m,n}$, $A_4\otimes\mathcal{R}_{m,n}$, $A_5\otimes\mathcal{R}_{m,n}$ be an $(m.n)$-dilated L-configuration contained in $\mathcal{Q}$. For clarity, we set $A = A_3 \in \mathcal{P}$ and without loss of generality we use opportune rotations and reflections to orient $\mathcal{Q}$ in such a way to distinguish $A \otimes \mathcal{R}_{m,n}$ as the lowest and leftest cell interval in $\mathcal{L}$.

    We first prove (1). Let $f = x_q x_r-x_s x_t$ be an arbitrary nonzero binomial in $J_\mathcal{Q}$. Let $h_q,v_q$ and let $h_r,v_r$ be the variables associated to the maximal horizontal and vertical edge intervals which contain $q$ and $r$, respectively. Since $f \in J_\mathcal{Q}$, then $\varphi(x_q x_r) = w^k h_q v_q h_r v_r = \varphi(x_s x_t)$ with $k \in \{0,1,2\}$. Suppose $h_q = h_r$, then $q$ and $r$ lie on the same maximal horizontal edge interval, say $H_q$. So $h_q^2 \mid \varphi(x_q x_r) = \varphi(x_s x_t)$, and so it must be that $s$ and $t$ also lie on $H_q$. Since we must have that $v_q \mid \varphi(x_s x_t)$ and $v_r \mid \varphi(x_s x_t)$, then one of $s$ or $t$ lies on $V_q$ while the other lies on $V_r$. So either we have $q = s$ and $r = t$ or $q = t$ and $r = s$, in either case $f = x_q x_r-x_s x_t = 0$, but this contradicts that $f$ is a nonzero binomial, so therefore $h_q \neq h_r$. A symmetric argument gives that $v_q \neq v_r$. So, without loss of generality, we have that $[q,r]$ is a proper interval of $\mathbb{N}^2$, for if it is not, we can rotate or reflect $\mathcal{Q}$ such that it is. Then $s$ and $t$ must be the anti-diagonal corners of $[q,r]$. Without loss of generality, assume $s$ is the upper left corner of $[q,r]$ and $t$ is the lower right corner of $[q,r]$, then we have that $[q,s] \subseteq V_q$, $[s,r] \subseteq H_r$, $[t,r] \subseteq V_r$, and $[q,t] \subseteq H_q$. Suppose $[q,r]$ is not an inner interval of $\mathcal{Q}$, then there exists some cell $C$ contained in $[q,r]$ such that $C \notin \mathcal{Q}$. As $[q,s]$, $[s,r]$, $[t,r]$, and $[q,t]$ are edge intervals in $\mathcal{Q}$, then it must be that $C \in \mathcal{H}$, where $\mathcal{H}$ is the unique hole of $\mathcal{Q}$. Therefore, $[q,r]$ contains the hole $\mathcal{H}$. As $[q,r]$ contains $\mathcal{H}$, then at least one of the corners $q$, $r$, $s$, or $t$ is in $V(A \otimes \mathcal{R}_{m,n})$. Without loss of generality, assume $q \in V(A \otimes \mathcal{R}_{m,n})$. Note that $s,t \notin V(A \otimes \mathcal{R}_{m,n})$, because if so, $[q,r]$ could not contain $\mathcal{H}$. However, if $q \in V(A \otimes \mathcal{R}_{m,n})$, then $w \mid \varphi(x_q x_r) = \varphi(x_s x_t) = \varphi(x_s) \varphi (x_t)$, so either $w \mid \varphi(x_s)$ or $w \mid \varphi(x_t)$. If $w \mid \varphi(x_s)$, then $s \in V(A \otimes \mathcal{R}_{m,n})$, which is a contradiction. If $w \mid \varphi(x_t)$, then $t \in V(A \otimes \mathcal{R}_{m,n})$, which is a contradiction. Thus, it must be that $[q,r]$ is an inner interval of $\mathcal{Q}$, and so therefore $f \in I_\mathcal{Q}$. As $f$ was arbitrary, then we have that every binomial of degree two in $J_\mathcal{Q}$ is in $I_\mathcal{Q}$.

    We now prove (2). Towards a contradiction, suppose $f \in J_\mathcal{Q}$ such that $f$ is irredundant and deg $f \geq 3$. Let $f = f^+ - f^-$. Suppose that 
    \[
    V_f^+ \cap V(A \otimes \mathcal{R}_{m,n}) = V_f^- \cap V(A \otimes \mathcal{R}_{m,n}) = \emptyset.
    \]
    We let $\mathcal{Q}' \subset \mathcal{Q}$ be the subpolyomino of $\mathcal{Q}$ defined by 
    \[
    \mathcal{Q}' = \{C \in \mathcal{Q} \text{ }\vert \text{ } V(C)\cap V(A \otimes \mathcal{R}_{m,n}) = \emptyset \}.
    \]
    Note that $V_f^+, V_f^- \subseteq V(\mathcal{Q}')$. We show that $\mathcal{Q}'$ is a simple polyomino. Let $e \in V(\mathcal{Q})$ be the upper right corner of $A \otimes \mathcal{R}_{m,n}$, that is, $v \leq e$, for all $v \in V(A \otimes \mathcal{R}_{m,n})$. By the structure of $\mathcal{Q}$, we have that there exists some cell $E \in \mathcal{H}$ such that $e$ is its lower left corner. As well, there exists some cell $F \in A_4 \otimes \mathcal{R}_{m,n}$ such that $e$ is its upper left corner and $E \cap F = [e, e+(1,0)]$. Since $e \in V(F)$, then $F\notin \mathcal{Q}'$. Let $C,D \notin \mathcal{Q}'$ be arbitrary, if both $C,D \in (\ext \mathcal{Q}) \cup (A\otimes \mathcal{R}_{m,n})$ or both $C,D \in \mathcal{H}$ then there is an obvious walk of cells not in $\mathcal{Q}'$ that connects them. If rather, one of $C$ or $D$ is in $\mathcal{H}$ while the other is in $(\ext \mathcal{Q}) \cup (A\otimes \mathcal{R}_{m,n})$, then there exists a walk of cells not in $\mathcal{Q}'$ given by $\mathcal{W}: C,...,E,F,...,D$ that connects them. Therefore, any two cells not in $\mathcal{Q}'$ are connected and so $\mathcal{Q}'$ is a simple polyomino. Let $\varphi'$ be the restriction of the map $\varphi$ onto $K[v_a \mid a\in V(\mathcal{Q}) \setminus V(A \otimes \mathcal{R}_{m,n})]$ and let $J_{\mathcal{Q}'} = \ker \varphi'$. As $\mathcal{Q}'$ is simple, then by Theorem 2.2 in \textbf{\cite{Qureshi_Shibuta_Shikama:2017}}, we have that $I_{\mathcal{Q}'}$, the polyomino ideal attached to $\mathcal{Q}'$, is prime and therefore $I_{\mathcal{Q}'} = J_{\mathcal{Q}'}$. So then $f \in I_{\mathcal{Q}'}$, but then we can write $f$ as a linear combination of the binomial generators of $I_{\mathcal{Q}'}$ which are all of degree 2 by definition. So $f$ is redundant, but this is a contradiction. Therefore, it must be that either $V_f^+ \cap V(A \otimes \mathcal{R}_{m,n}) \neq \emptyset$ or $V_f^- \cap V(A \otimes \mathcal{R}_{m,n}) \neq \emptyset$. 
    
    Suppose $V_f^+ \cap V(A \otimes \mathcal{R}_{m,n}) \neq \emptyset$. So there exists $u_1 \in V(A \otimes \mathcal{R}_{m,n})$ such that $x_{u_1} \mid f^+$. Since $u_1 \in V(A \otimes \mathcal{R}_{m,n})$, then $w \mid \varphi(u_1)$, so then $w \mid \varphi(f^+) = \varphi(f^-)$. So then there exists $v_1 \in  V(A \otimes \mathcal{R}_{m,n})$ such that $x_{v_1} \mid f^-$. If we were rather to suppose that $V_f^- \cap V(A \otimes \mathcal{R}_{m,n}) \neq \emptyset$, we would symmetrically obtain that there must exist some $u_1, v_1 \in V(A \otimes \mathcal{R}_{m,n})$ such that $x_{u_1} \mid f^+$ and $x_{v_1} \mid f^-$. If $u_1 = v_1$, then $f = x_{u_1}(\tilde f^+ - \tilde f^-)$, with $\tilde f^+ = \frac {f^+}{x_{u_1}}$ and $\tilde f^- = \frac {f^-}{x_{u_1}}$. Note that $\tilde f^+ - \tilde f^- \in J_\mathcal{Q}$ and deg$(\tilde f^+ - \tilde f^-) <$ deg $f$. So $f$ is redundant, but this is a contradiction. Therefore, $u_1 \neq v_1$. Let $V_{u_1}$ and $H_{u_1}$ be the maximal vertical and horizontal edge intervals which contain $u_1$. Then $v_{u_1} \mid \varphi(f^+) = \varphi(f^-)$, so there exists some vertex $v_2$ that lies on $V_{u_1}$ such that $x_{v_2} \mid f^-$. We also have that $h_{u_1} \mid \varphi(f^+) = \varphi(f^-)$, so there exists some vertex $v_3$ that lies on $H_{u_1}$ such that $x_{v_3} \mid f^-$. Now, let $V_{v_1}$ and $H_{v_1}$ be the maximal vertical and horizontal edge intervals which contain $v_1$. Then $v_{v_1} \mid \varphi(f^-) = \varphi(f^+)$, so there exists some vertex $u_2$ that lies on $V_{v_1}$ such that $x_{u_2} \mid f^+$. We also have that $h_{v_1} \mid \varphi(f^-) = \varphi(f^+)$, so there exists some vertex $u_3$ that lies on $H_{v_1}$ such that $x_{u_3} \mid f^+$. Note that by the same argument as above, if $f$ is irredundant, then we must have that $u_2 \neq v_1$, $u_3 \neq v_1$, $v_2 \neq u_1$, and $v_3 \neq u_1$. The following cases could occur: 
    
    \begin{description}
    
    \item[(1)] $u_1$ and $v_1$ lie on the same vertical edge interval. For the structure of $\mathcal{Q}$ as a dilated closed path, at least one of the vertex pairs $u_1, u_3$ or $v_1, v_3$ defines an inner interval in $\mathcal{Q}$. If $u_1$ and $u_3$ define an inner interval in $\mathcal{Q}$, say $\mathcal{I}$, then $v_1$ is one of the remaining two corners of $\mathcal{I}$ and thus applying Lemma \ref{Lemma: 2.2} to $u_1$, $u_3$ and $v_1$ gives that $f$ is redundant, which is a contradiction. If rather, $v_1$ and $v_3$ define an inner interval in $\mathcal{Q}$, say $\mathcal{J}$, then $u_1$ is one of the remaining two corners of $\mathcal{I}$, and thus applying Lemma \ref{Lemma: 2.2} to $u_1$, $v_1$ and $v_3$ gives that $f$ is redundant, which is a contradiction.

    \item[(2)] $u_1$ and $v_1$ lie on the same horizontal edge interval. For the structure of $\mathcal{Q}$ as a dilated closed path, at least one of the vertex pairs $u_1, u_2$ or $v_1, v_2$ defines an inner interval in $\mathcal{Q}$. If $u_1$ and $u_2$ define an inner interval in $\mathcal{Q}$, say $\mathcal{I}$, then $v_1$ is one of the remaining two corners of $\mathcal{I}$ and thus applying Lemma \ref{Lemma: 2.2} to $u_1$, $u_2$ and $v_1$ gives that $f$ is redundant, which is a contradiction. If rather, $v_1$ and $v_2$ define an inner interval in $\mathcal{Q}$, say $\mathcal{J}$, then $u_1$ is one of the remaining two corners of $\mathcal{J}$, and thus applying Lemma \ref{Lemma: 2.2} to $u_1$, $v_1$ and $v_2$ gives that $f$ is redundant, which is a contradiction.

    \item[(3)] $u_1$ and $v_1$ lie on distinct vertical and horizontal edge intervals. In this case, we have that $u_1$ and $v_1$ define some inner interval $\mathcal{K}$ contained in $ A\otimes \mathcal{R}_{m,n}$. We break this case into the following two subcases:

    \begin{description}

    \item[(I)] $u_1$ and $v_1$ are diagonal corners of $\mathcal{K}$. Without loss of generality, let $u_1 < v_1$. We show that $v_3$ cannot lie on $V_{v_1}$. Suppose it were the case that $[v_3, v_1] \subseteq V_{v_1}$. Then, for the structure of $\mathcal{Q}$, at least one of the vertex pairs $u_1, u_2$ or $v_3, v_2$ defines an inner interval of $\mathcal{Q}$. If $u_1$ and $u_2$ define an inner interval, say $\mathcal{I}$, then as $u_1$ and $v_3$ both lie on $H_{u_1}$ and as $v_3$ and $u_2$ both lie on $V_{v_1}$, we have that $v_3$ is either an anti-diagonal or diagonal corner of $\mathcal{I}$. Thus, applying Lemma \ref{Lemma: 2.2} to $u_1$, $u_2$ and $v_3$ gives that $f$ is redundant, which is a contradiction. If rather, $v_3$ and $v_2$ define an inner interval, say $\mathcal{J}$, then as $u_1$ and $v_2$ both lie on $V_{u_1}$ and as $u_1$ and $v_3$ both lie on $H_{u_1}$, we have that $u_1$ is either a diagonal or anti-diagonal corner of $\mathcal{J}$. Thus, applying Lemma \ref{Lemma: 2.2} to $v_2$, $v_3$, and $u_1$ gives that $f$ is redundant, which is a contradiction. By similar arguments, $u_2$ cannot lie on $H_{u_1}$, $u_3$ cannot lie on $V_{u_1}$, and $v_2$ cannot lie on $H_{v_1}$. Investigating the possible positions of $u_2$ and $u_3$ relative to $v_1$ allows us to break this case into the following four subcases:

    \begin{description}

    \item[(i)] $u_2 < v_1$ and $u_3 < v_1$. In this case, for the structure of $\mathcal{Q}$, we have that $u_2,u_3 \in V(A \otimes \mathcal{R}_{m,n})$. Therefore $u_2$, $u_3$, and $v_1$ define some inner interval $\mathcal{I}$ contained in $A \otimes \mathcal{R}_{m,n}$, where $u_3$ and $u_2$ are the anti-diagonal corners of $\mathcal{I}$ and $v_1$ is a diagonal corner of $\mathcal{I}$ (see Figure \ref{Fig:Thm321_Case3Ii}). Thus, applying Lemma \ref{Lemma: 2.2} to $u_2$, $u_3$, and $v_1$ gives that $f$ is redundant, which is a contradiction.

    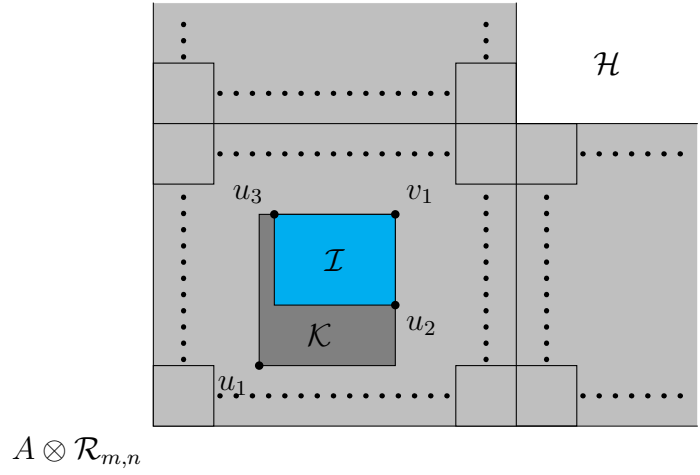
\begin{figure} [H]
    \centering
    \begin{tikzpicture} [scale = 0.8]
        \draw[step=1cm,white,very thin] (0,0) grid (9,7);
        \draw [white, thin, fill = lightgray] (0,5) rectangle (6,7);
        \draw [white, thin, fill = lightgray] (6,0) rectangle (9,5);
        \draw [black, thin, fill = lightgray] (0,0) rectangle (6,5);
        \polyomino[
        empty cell =x,
        grid,
        p={a}{style={lightgray,draw=black}},
        p={b}{style={gray,draw=black}},
        row sep =;
        ]{
        x x x x x x x x x;
        a x x x x a x x x;
        a x x x x a a x x;
        x x x x x x x x x;
        x x x x x x x x x;
        x x x x x x x x x;
        a x x x x a a x x
        }

        \draw [black, thin] (0,5) -- (0,7);
        \draw [black, thin] (6,5) -- (6,7);
        \draw [black, thin] (6,5) -- (9,5);
        \draw [black, thin] (6,0) -- (9,0);

        \draw[dots] (0.5,1.1) -- (0.5, 3.9);
        \draw[dots] (5.5,1.1) -- (5.5, 3.9);
        \draw[dots] (1.1, 0.5) -- (4.9, 0.5);
        \draw[dots] (1.1, 4.5) -- (4.9, 4.5);

        \draw[dots] (6.5,1.1) -- (6.5, 3.9);
        \draw[dots] (1.1, 5.5) -- (4.9, 5.5);
        \draw[dots] (0.5,6.1) -- (0.5, 6.9);
        \draw[dots] (5.5,6.1) -- (5.5, 6.9);
        \draw[dots] (7.1,0.5) -- (8.9, 0.5);
        \draw[dots] (7.1,4.5) -- (8.9, 4.5);

        \draw [black, thin, fill = gray] (1.75,1) rectangle (4,3.5);
        \draw [black, thin, fill = cyan] (2,2) rectangle (4,3.5);
        
        \begin{scope}[color=black]
            \node[anchor=center] () at (7.5,6) {$\mathcal{H}$};
            \node[anchor=center] () at (2.75,1.5) {$\mathcal{K}$};
            \node[anchor=center] () at (3,2.75) {$\mathcal{I}$};
        \end{scope}

        \fill(1.75,1) circle[radius=2pt] node[below left]{$u_1$};
        \fill(4,3.5) circle[radius=2pt] node[above right]{$v_1$};
        \fill(4,2) circle[radius=2pt] node[below right]{$u_2$};
        \fill(2,3.5) circle[radius=2pt] node[above left]{$u_3$};

        \fill(0,0) circle[radius=0pt] node[below left]{$A \otimes \mathcal{R}_{m,n}$};
        
    \end{tikzpicture}
    \caption{One possible configuration of inner intervals in this case}
    \label{Fig:Thm321_Case3Ii}
    \end{figure}

    \item[(ii)] $u_2 > v_1$ and $u_3 < v_1$. In this case, if $u_3$ and $u_2$ define an inner interval in $\mathcal{Q}$, say $\mathcal{I}$, then they will be the diagonal corners of this interval and $v_1$ will be an anti-diagonal corner of $\mathcal{I}$. Thus, applying Lemma \ref{Lemma: 2.2} to $u_2$, $u_3$, and $v_1$ gives that $f$ is redundant, which is a contradiction. Therefore, we suppose that $u_3$ and $u_2$ do not define an inner interval in $\mathcal{Q}$. For the structure of $\mathcal{Q}$, this means that $v_1$ and $v_2$ must define some inner interval, say $\mathcal{J}$, in $\mathcal{Q}$. Suppose that $v_2$ and $v_1$ are the diagonal corners of $\mathcal{J}$. Let $g,h \in V(\mathcal{Q})$ be the anti-diagonal corners of $\mathcal{J}$. As $\mathcal{J}$ is an inner interval of $\mathcal{Q}$, then we have that $x_{v_2}x_{v_1} - x_{g} x_{h} \in I_\mathcal{Q} \subseteq J_\mathcal{Q}$. Note that we can write $f$ in the following way: 
    \[ \quad \quad \quad \quad \quad \quad \quad
         f = f^+ - f^- = f^+ - \frac{f^-}{x_{v_1}x_{v_2}}(x_{v_2}x_{v_1} - x_g x_h) - \frac{f^-}{x_{v_1}x_{v_2}}x_g x_h
    \]
    Let $\tilde f = f^+ - \frac{f^-}{x_{v_1}x_{v_2}}x_g x_h$. Since $x_{v_2}x_{v_1} - x_g x_h \in J_\mathcal{Q}$, then $\varphi(x_{v_2} x_{v_1}) = \varphi(x_g x_h)$, so therefore $\varphi(\frac{f^-}{x_{v_1}x_{v_2}}x_g x_h) = \varphi(f^-)$ and as such, we have that $\tilde f \in J_\mathcal{Q}$. Note that $u_1$ and $u_3$ together with $g$ define some inner interval, say $\mathcal{J}'$, in $\mathcal{Q}$, where $u_1$ and $u_3$ are the diagonal (resp. anti-diagonal) corners of $\mathcal{J}'$ and $g$ is an anti-diagonal (resp. diagonal) corner of $\mathcal{J}'$ (see Figure \ref{Fig:Thm3.2.14_Case3Iiia}). Thus, applying Lemma \ref{Lemma: 2.2} to $u_1$, $u_3$ and $g$ gives that $\tilde f$ is redundant. Since $f = \tilde f - \frac{f^-}{x_{v_1}x_{v_2}}(x_{v_2}x_{v_1} - x_g x_h)$, then we also have that $f$ is redundant, which is a contradiction. 

    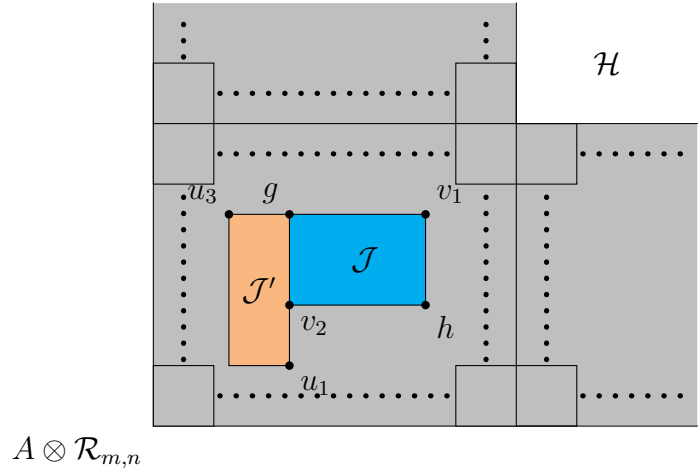
\begin{figure} [H] 
    \centering
    \begin{tikzpicture} [scale = 0.8]
        \draw[step=1cm,white,very thin] (0,0) grid (9,7);
        \draw [white, thin, fill = lightgray] (0,5) rectangle (6,7);
        \draw [white, thin, fill = lightgray] (6,0) rectangle (9,5);
        \draw [black, thin, fill = lightgray] (0,0) rectangle (6,5);
        \polyomino[
        empty cell =x,
        grid,
        p={a}{style={lightgray,draw=black}},
        p={b}{style={gray,draw=black}},
        row sep =;
        ]{
        x x x x x x x x x;
        a x x x x a x x x;
        a x x x x a a x x;
        x x x x x x x x x;
        x x x x x x x x x;
        x x x x x x x x x;
        a x x x x a a x x
        }

        \draw [black, thin] (0,5) -- (0,7);
        \draw [black, thin] (6,5) -- (6,7);
        \draw [black, thin] (6,5) -- (9,5);
        \draw [black, thin] (6,0) -- (9,0);

        \draw[dots] (0.5,1.1) -- (0.5, 3.9);
        \draw[dots] (5.5,1.1) -- (5.5, 3.9);
        \draw[dots] (1.1, 0.5) -- (4.9, 0.5);
        \draw[dots] (1.1, 4.5) -- (4.9, 4.5);

        \draw[dots] (6.5,1.1) -- (6.5, 3.9);
        \draw[dots] (1.1, 5.5) -- (4.9, 5.5);
        \draw[dots] (0.5,6.1) -- (0.5, 6.9);
        \draw[dots] (5.5,6.1) -- (5.5, 6.9);
        \draw[dots] (7.1,0.5) -- (8.9, 0.5);
        \draw[dots] (7.1,4.5) -- (8.9, 4.5);

        \draw [black, thin, fill = cyan] (2.25,2) rectangle (4.5,3.5);
        \draw [black, thin, fill = Apricot] (1.25,1) rectangle (2.25,3.5);
        
        \begin{scope}[color=black]
            \node[anchor=center] () at (7.5,6) {$\mathcal{H}$};
            \node[anchor=center] () at (3.5,2.75) {$\mathcal{J}$};
            \node[anchor=center] () at (1.75,2.25) {$\mathcal{J'}$};
        \end{scope}

        \fill(2.25,1) circle[radius=2pt] node[below right]{$u_1$};
        \fill(1.25,3.5) circle[radius=2pt] node[above left]{$u_3$};
        \fill(4.5,3.5) circle[radius=2pt] node[above right]{$v_1$};
        \fill(2.25,2) circle[radius=2pt] node[below right]{$v_2$};
        \fill(4.5,2) circle[radius=2pt] node[below right]{$h$};
        \fill(2.25,3.5) circle[radius=2pt] node[above left]{$g$};

        \fill(0,0) circle[radius=0pt] node[below left]{$A \otimes \mathcal{R}_{m,n}$};
        
    \end{tikzpicture}
    \caption{One possible configuration of inner intervals when $v_2$ and $v_1$ are diagonal corners of $\mathcal{J}$}
    \label{Fig:Thm3.2.14_Case3Iiia}
    \end{figure}
    
    Now, suppose that $v_2$ and $v_1$ are instead the anti-diagonal corners of $\mathcal{J}$. Let $g',h' \in V(\mathcal{Q})$ be the diagonal corners of $\mathcal{J}$. As $\mathcal{J}$ is an inner interval of $\mathcal{Q}$, then we have that $x_{g'}x_{h'} - x_{v_2} x_{v_1} \in I_\mathcal{Q} \subseteq J_\mathcal{Q}$. Note that we can write $f$ in the following way:
    \[ \quad \quad \quad \quad \quad \quad \quad
    f = f^+ - f^- = f^+ + \frac{f^-}{x_{v_1}x_{v_2}}(x_{g'}x_{h'} - x_{v_2} x_{v_1}) - \frac{f^-}{x_{v_1}x_{v_2}}x_{g'} x_{h'}
    \]
    Let $\hat f = f^+ - \frac{f^-}{x_{v_1}x_{v_2}}x_{g'} x_{h'}$. Since $x_{g'}x_{h'} - x_{v_2} x_{v_1} \in J_\mathcal{Q}$, then $\varphi(x_{v_2} x_{v_1}) = \varphi(x_{g'} x_{h'})$, so therefore $\varphi(\frac{f^-}{x_{v_1}x_{v_2}}x_{g'} x_{h'}) = \varphi(f^-)$ and as such, we have that $\hat f \in J_\mathcal{Q}$. Note that $u_1$ and $u_3$ together with $g'$ define some inner interval, say $\mathcal{J}''$, in $\mathcal{Q}$, where $u_1$ and $u_3$ are the diagonal (resp. anti-diagonal) corners of $\mathcal{J}''$ and $g'$ is an anti-diagonal (resp. diagonal) corner of $\mathcal{J}''$ (see Figure \ref{Fig:Thm3.2.14_Case3Iiib}). Thus, applying Lemma \ref{Lemma: 2.2} to $u_1$, $u_3$ and $g'$ gives that $\hat f$ is redundant. Since $f = \hat f + \frac{f^-}{x_{v_1}x_{v_2}}(x_{g'}x_{h'} - x_{v_2} x_{v_1})$, then we also have that $f$ is redundant, which is a contradiction.
 
 \begin{figure} [H] 
    \centering
    \begin{tikzpicture} [scale = 0.8]
        \draw[step=1cm,white,very thin] (0,0) grid (9,7);
        \draw [white, thin, fill = lightgray] (0,5) rectangle (6,7);
        \draw [white, thin, fill = lightgray] (6,0) rectangle (9,5);
        \draw [black, thin, fill = lightgray] (0,0) rectangle (6,5);
        \polyomino[
        empty cell =x,
        grid,
        p={a}{style={lightgray,draw=black}},
        p={b}{style={gray,draw=black}},
        row sep =;
        ]{
        x x x x x x x x x;
        a x x x x a x x x;
        a x x x x a a x x;
        x x x x x x x x x;
        x x x x x x x x x;
        x x x x x x x x x;
        a x x x x a a x x
        }

        \draw [black, thin] (0,5) -- (0,7);
        \draw [black, thin] (6,5) -- (6,7);
        \draw [black, thin] (6,5) -- (9,5);
        \draw [black, thin] (6,0) -- (9,0);

        \draw[dots] (0.5,1.1) -- (0.5, 3.9);
        \draw[dots] (5.5,1.1) -- (5.5, 3.9);
        \draw[dots] (1.1, 0.5) -- (4.9, 0.5);
        \draw[dots] (1.1, 4.5) -- (4.9, 4.5);

        \draw[dots] (6.5,1.1) -- (6.5, 3.9);
        \draw[dots] (1.1, 5.5) -- (4.9, 5.5);
        \draw[dots] (0.5,6.1) -- (0.5, 6.9);
        \draw[dots] (5.5,6.1) -- (5.5, 6.9);
        \draw[dots] (7.1,0.5) -- (8.9, 0.5);
        \draw[dots] (7.1,4.5) -- (8.9, 4.5);

        \draw [black, thin, fill = gray] (2.25,1) rectangle (4.5,3.5);
        \draw [black, thin, fill = cyan] (2.25,3.5) rectangle (4.5,6.5);
        \draw [black, thin, fill = Apricot] (1.25,1) rectangle (2.25,3.5);
        
        \begin{scope}[color=black]
            \node[anchor=center] () at (7.5,6) {$\mathcal{H}$};
            \node[anchor=center] () at (3.375,2.25) {$\mathcal{K}$};
            \node[anchor=center] () at (3.375,5) {$\mathcal{J}$};
            \node[anchor=center] () at (1.75,2.25) {$\mathcal{J''}$};
        \end{scope}

        \fill(2.25,1) circle[radius=2pt] node[below right]{$u_1$};
        \fill(1.25,3.5) circle[radius=2pt] node[above left]{$u_3$};
        \fill(4.5,3.5) circle[radius=2pt] node[above right]{$v_1$};
        \fill(2.25,6.5) circle[radius=2pt] node[above left]{$v_2$};
        \fill(2.25,3.5) circle[radius=2pt] node[above left]{$g'$};
        \fill(4.5,6.5) circle[radius=2pt] node[above right]{$h'$};

        \fill(0,0) circle[radius=0pt] node[below left]{$A \otimes \mathcal{R}_{m,n}$};
        
    \end{tikzpicture}
    \caption{One possible configuration of inner intervals when $v_2$ and $v_1$ are anti-diagonal corners of $\mathcal{J}$}
    \label{Fig:Thm3.2.14_Case3Iiib}
    \end{figure}
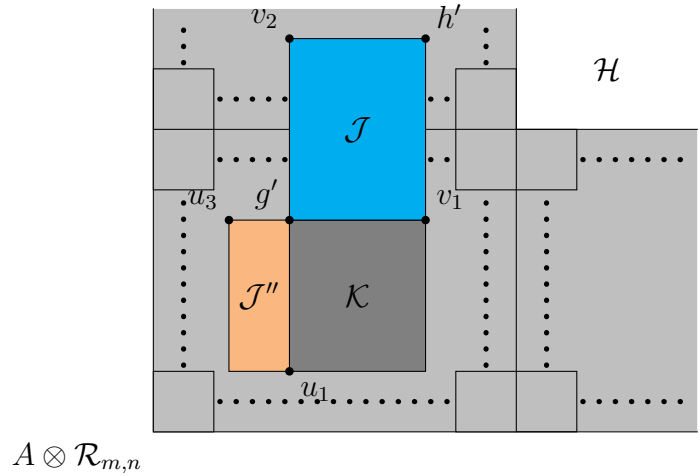
    
    \item[(iii)] $u_2 < v_1$ and $u_3 > v_1$. In this case, if $u_3$ and $u_2$ define an inner interval, say $\mathcal{I}$, in $\mathcal{Q}$, then they will be the diagonal corners of this interval and $v_1$ will be an anti-diagonal corner of $\mathcal{I}$. Thus, as in the above subcase, applying Lemma \ref{Lemma: 2.2} to $u_2$, $u_3$, and $v_1$ gives that $f$ is redundant, which is a contradiction. Therefore, we suppose that $u_3$ and $u_2$ do not define an inner interval in $\mathcal{Q}$. For the structure of $\mathcal{Q}$, this means that $v_1$ and $v_3$ must define some inner interval in $\mathcal{Q}$, say $\mathcal{J}$. Suppose that $v_3$ and $v_1$ are the diagonal corners of $\mathcal{J}$. Let $g,h \in V(\mathcal{Q})$ be the anti-diagonal corners of $\mathcal{J}$. As $\mathcal{J}$ is an inner interval of $\mathcal{Q}$, then we have that $x_{v_3}x_{v_1} - x_{g} x_{h} \in I_\mathcal{Q} \subseteq J_\mathcal{Q}$. Note that we can write $f$ in the following way:
    \[ \quad \quad \quad \quad \quad \quad \quad
    f = f^+ - f^- = f^+ - \frac{f^-}{x_{v_1}x_{v_3}}(x_{v_3}x_{v_1} - x_g x_h) - \frac{f^-}{x_{v_1}x_{v_3}}x_g x_h
    \]
    Let $\tilde f = f^+ - \frac{f^-}{x_{v_1}x_{v_3}}x_g x_h$. Since $x_{v_3}x_{v_1} - x_g x_h \in J_\mathcal{Q}$, then $\varphi(x_{v_3} x_{v_1}) = \varphi(x_g x_h)$, so therefore $\varphi(\frac{f^-}{x_{v_1}x_{v_3}}x_g x_h) = \varphi(f^-)$ and as such, we have that $\tilde f \in J_\mathcal{Q}$. Note that $u_1$ and $u_2$ together with $h$ define some inner interval, say $\mathcal{J}'$, in $\mathcal{Q}$, where $u_1$ and $u_2$ are the diagonal (resp. anti-diagonal) corners of $\mathcal{J}'$ and $h$ is an anti-diagonal (resp. diagonal) corner of $\mathcal{J}'$ (see Figure \ref{Fig:Thm3.2.14_Case3Iiiia}). Thus, applying Lemma \ref{Lemma: 2.2} to $u_1$, $u_2$ and $h$ gives that $\tilde f$ is redundant. Since $f = \tilde f - \frac{f^-}{x_{v_1}x_{v_3}}(x_{v_3}x_{v_1} - x_g x_h)$, then we also have that $f$ is redundant, which is a contradiction. 

    \begin{figure} [H] 
    \centering
    \begin{tikzpicture} [scale = 0.8]
        \draw[step=1cm,white,very thin] (0,0) grid (9,7);
        \draw [white, thin, fill = lightgray] (0,5) rectangle (6,7);
        \draw [white, thin, fill = lightgray] (6,0) rectangle (9,5);
        \draw [black, thin, fill = lightgray] (0,0) rectangle (6,5);
        \polyomino[
        empty cell =x,
        grid,
        p={a}{style={lightgray,draw=black}},
        p={b}{style={gray,draw=black}},
        row sep =;
        ]{
        x x x x x x x x x;
        a x x x x a x x x;
        a x x x x a a x x;
        x x x x x x x x x;
        x x x x x x x x x;
        x x x x x x x x x;
        a x x x x a a x x
        }

        \draw [black, thin] (0,5) -- (0,7);
        \draw [black, thin] (6,5) -- (6,7);
        \draw [black, thin] (6,5) -- (9,5);
        \draw [black, thin] (6,0) -- (9,0);

        \draw[dots] (0.5,1.1) -- (0.5, 3.9);
        \draw[dots] (5.5,1.1) -- (5.5, 3.9);
        \draw[dots] (1.1, 0.5) -- (4.9, 0.5);
        \draw[dots] (1.1, 4.5) -- (4.9, 4.5);

        \draw[dots] (6.5,1.1) -- (6.5, 3.9);
        \draw[dots] (1.1, 5.5) -- (4.9, 5.5);
        \draw[dots] (0.5,6.1) -- (0.5, 6.9);
        \draw[dots] (5.5,6.1) -- (5.5, 6.9);
        \draw[dots] (7.1,0.5) -- (8.9, 0.5);
        \draw[dots] (7.1,4.5) -- (8.9, 4.5);

        \draw [black, thin, fill = Apricot] (1.25,1) rectangle (4.5,2);
        \draw [black, thin, fill = cyan] (2.25,2) rectangle (4.5,3.5);

        \begin{scope}[color=black]
            \node[anchor=center] () at (7.5,6) {$\mathcal{H}$};
            \node[anchor=center] () at (3,1.5) {$\mathcal{J'}$};
            \node[anchor=center] () at (3.5,2.75) {$\mathcal{J}$};
        \end{scope}

        \fill(1.25,2) circle[radius=2pt] node[below left]{$u_1$};
        \fill(4.5,1) circle[radius=2pt] node[below]{$u_2$};
        \fill(4.5,3.5) circle[radius=2pt] node[above right]{$v_1$};
        \fill(2.25,2) circle[radius=2pt] node[above left]{$v_3$};
        \fill(4.5,2) circle[radius=2pt] node[below right]{$h$};
        \fill(2.25,3.5) circle[radius=2pt] node[above left]{$g$};

        \fill(0,0) circle[radius=0pt] node[below left]{$A \otimes \mathcal{R}_{m,n}$};
        
    \end{tikzpicture}
    \caption{One possible configuration of inner intervals when $v_3$ and $v_1$ are diagonal corners of $\mathcal{J}$}
    \label{Fig:Thm3.2.14_Case3Iiiia}
    \end{figure}
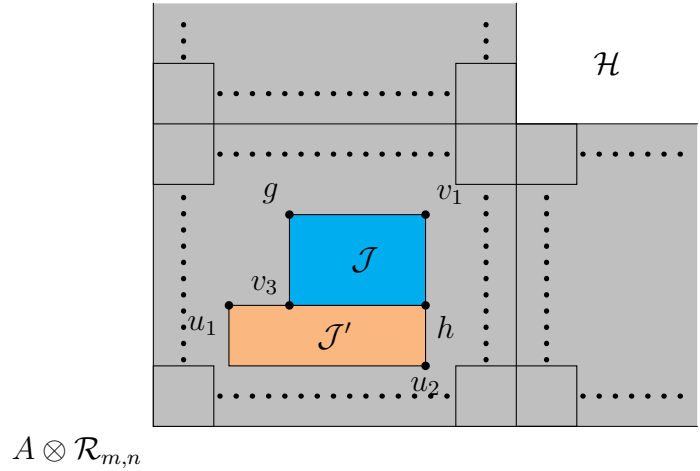
   
    Now, suppose that $v_1$ and $v_3$ are instead the anti-diagonal corners of $\mathcal{J}$. Let $g',h' \in V(\mathcal{Q})$ be the diagonal corners of $\mathcal{J}$. As $\mathcal{J}$ is an inner interval of $\mathcal{Q}$, then we have that $x_{g'}x_{h'} - x_{v_1} x_{v_3} \in I_\mathcal{Q} \subseteq J_\mathcal{Q}$. Note that we can write $f$ in the following way:
    \[ \quad \quad \quad \quad \quad \quad \quad
    f = f^+ - f^- = f^+ + \frac{f^-}{x_{v_1}x_{v_3}}(x_{g'}x_{h'} - x_{v_1} x_{v_3}) - \frac{f^-}{x_{v_1}x_{v_3}}x_{g'} x_{h'}
    \]
    Let $\hat f = f^+ - \frac{f^-}{x_{v_1}x_{v_3}}x_{g'} x_{h'}$. Since $x_{g'}x_{h'} - x_{v_1} x_{v_3} \in J_\mathcal{Q}$, then $\varphi(x_{v_1} x_{v_3}) = \varphi(x_{g'} x_{h'})$, so therefore $\varphi(\frac{f^-}{x_{v_1}x_{v_3}}x_{g'} x_{h'}) = \varphi(f^-)$ and as such, we have that $\hat f \in J_\mathcal{Q}$. Note that $u_1$ and $u_2$ together with $g'$ define some inner interval in $\mathcal{Q}$, say $\mathcal{J}''$, where $u_1$ and $u_2$ are the diagonal (resp. anti-diagonal) corners of $\mathcal{J}''$ and $g'$ is an anti-diagonal (resp. diagonal) corner of $\mathcal{J}''$ (see Figure \ref{Fig:Thm3.2.14_Case3Iiiib}). Thus, applying Lemma \ref{Lemma: 2.2} to $u_1$, $u_2$ and $g'$ gives that $\hat f$ is redundant. Since $f = \hat f + \frac{f^-}{x_{v_1}x_{v_3}}(x_{g'}x_{h'} - x_{v_1} x_{v_3})$, then we also have that $f$ is redundant, which is a contradiction.

   \begin{figure} [H]
    \centering
    \begin{tikzpicture} [scale = 0.8]
        \draw[step=1cm,white,very thin] (0,0) grid (9,7);
        \draw [white, thin, fill = lightgray] (0,5) rectangle (6,7);
        \draw [white, thin, fill = lightgray] (6,0) rectangle (9,5);
        \draw [black, thin, fill = lightgray] (0,0) rectangle (6,5);
        \polyomino[
        empty cell =x,
        grid,
        p={a}{style={lightgray,draw=black}},
        p={b}{style={gray,draw=black}},
        row sep =;
        ]{
        x x x x x x x x x;
        a x x x x a x x x;
        a x x x x a a x x;
        x x x x x x x x x;
        x x x x x x x x x;
        x x x x x x x x x;
        a x x x x a a x x
        }

        \draw [black, thin] (0,5) -- (0,7);
        \draw [black, thin] (6,5) -- (6,7);
        \draw [black, thin] (6,5) -- (9,5);
        \draw [black, thin] (6,0) -- (9,0);

        \draw[dots] (0.5,1.1) -- (0.5, 3.9);
        \draw[dots] (5.5,1.1) -- (5.5, 3.9);
        \draw[dots] (1.1, 0.5) -- (4.9, 0.5);
        \draw[dots] (1.1, 4.5) -- (4.9, 4.5);

        \draw[dots] (6.5,1.1) -- (6.5, 3.9);
        \draw[dots] (1.1, 5.5) -- (4.9, 5.5);
        \draw[dots] (0.5,6.1) -- (0.5, 6.9);
        \draw[dots] (5.5,6.1) -- (5.5, 6.9);
        \draw[dots] (7.1,0.5) -- (8.9, 0.5);
        \draw[dots] (7.1,4.5) -- (8.9, 4.5);

        \draw [black, thin, fill = Apricot] (1.25,1) rectangle (4.5,2);
        \draw [black, thin, fill = cyan] (4.5,2) rectangle (7.5,3.5);
        \draw [black, thin, fill = gray] (1.25,2) rectangle (4.5,3.5);

        \begin{scope}[color=black]
            \node[anchor=center] () at (7.5,6) {$\mathcal{H}$};
            \node[anchor=center] () at (2.875,2.75) {$\mathcal{K}$};
            \node[anchor=center] () at (6,2.75) {$\mathcal{J}$};
            \node[anchor=center] () at (2.875,1.5) {$\mathcal{J''}$};
        \end{scope}

        \fill(1.25,2) circle[radius=2pt] node[below left]{$u_1$};
        \fill(4.5,1) circle[radius=2pt] node[below]{$u_2$};
        \fill(4.5,3.5) circle[radius=2pt] node[above right]{$v_1$};
        \fill(7.5,2) circle[radius=2pt] node[below right]{$v_3$};
        \fill(4.5,2) circle[radius=2pt] node[below right]{$g'$};
        \fill(7.5,3.5) circle[radius=2pt] node[above right]{$h'$};

        \fill(0,0) circle[radius=0pt] node[below left]{$A \otimes \mathcal{R}_{m,n}$};
        
    \end{tikzpicture}
    \caption{One possible configuration of inner intervals when $v_1$ and $v_3$ are anti-diagonal corners of $\mathcal{J}$}
    \label{Fig:Thm3.2.14_Case3Iiiib}
    \end{figure}
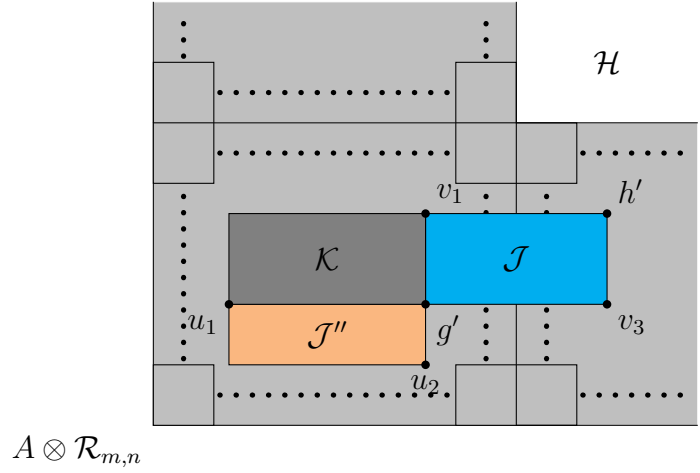
    
    \item[(iv)] $u_2 > v_1$ and $u_3 > v_1$. For the structure of $\mathcal{Q}$, we have the following two subcases:

    \begin{description}
    
    \item[(a)] $u_1$ and $u_3$ define an inner interval in $\mathcal{Q}$. Note that in this case, the inner interval is precisely $[u_1, u_3]$. Let $g,h \in V(\mathcal{Q})$ be the anti-diagonal corners of $[u_1,u_3]$. As $[u_1,u_3]$ is an inner interval of $\mathcal{Q}$, then we have that $x_{u_1}x_{u_3} - x_{g} x_{h} \in I_\mathcal{Q} \subseteq J_\mathcal{Q}$. Note that we can write $f$ in the following way:
    \[ \quad \quad \quad \quad \quad \quad \quad \quad \quad \quad
    f = f^+ - f^- = \frac{f^+}{x_{u_1} x_{u_3}}(x_{u_1} x_{u_3} - x_g x_h) + \frac{f^+}{x_{u_1} x_{u_3}}x_g x_h - f^-
    \]
    
    Let $\tilde f = \frac{f^+}{x_{u_1} x_{u_3}}x_g x_h - f^-$. Since $x_{u_1}x_{u_3} - x_g x_h \in J_\mathcal{Q}$, then $\varphi(x_{u_1} x_{u_3}) = \varphi(x_g x_h)$, so therefore $\varphi(\frac{f^+}{x_{u_1}x_{u_3}}x_g x_h) = \varphi(f^+)$ and as such, we have that $\tilde f \in J_\mathcal{Q}$. For the structure of $\mathcal{Q}$, at least one of the vertex pairs $g,u_2$ or $v_1, v_2$ defines an inner interval in $\mathcal{Q}$. Suppose $[g,u_2]$ is an inner interval. Then $v_1$ is an anti-diagonal corner of $[g,u_2]$. Thus, applying Lemma \ref{Lemma: 2.2} to $g$, $u_2$ and $v_1$ gives that $\tilde f$ is redundant. Now, suppose $v_1$ and $v_2$ define an inner interval in $\mathcal{Q}$, say $\mathcal{I}$, then $g$ is an anti-diagonal (resp. diagonal) corner of $\mathcal{I}$ given that $v_2$ and $v_1$ are its diagonal (resp. anti-diagonal) corners (see Figure \ref{Fig:Thm3.2.14_Case3Iiva}). Thus, applying Lemma \ref{Lemma: 2.2} to $v_1$, $v_2$ and $g$ gives that $\tilde f$ is redundant. In either case, we have that $\tilde f$ is redundant, and as $f = \frac{f^+}{x_{u_1} x_{u_3}}(x_{u_1} x_{u_3} - x_g x_h) + \tilde f$, then $f$ is also redundant, which is a contradiction.

\begin{figure} [H] 
    \centering
    \begin{tikzpicture} [scale = 0.8]
        \draw[step=1cm,white,very thin] (0,0) grid (9,7);
        \draw [white, thin, fill = lightgray] (0,5) rectangle (6,7);
        \draw [white, thin, fill = lightgray] (6,0) rectangle (9,5);
        \draw [black, thin, fill = lightgray] (0,0) rectangle (6,5);
        \polyomino[
        empty cell =x,
        grid,
        p={a}{style={lightgray,draw=black}},
        p={b}{style={gray,draw=black}},
        row sep =;
        ]{
        x x x x x x x x x;
        a x x x x a x x x;
        a x x x x a a x x;
        x x x x x x x x x;
        x x x x x x x x x;
        x x x x x x x x x;
        a x x x x a a x x
        }

        \draw [black, thin] (0,5) -- (0,7);
        \draw [black, thin] (6,5) -- (6,7);
        \draw [black, thin] (6,5) -- (9,5);
        \draw [black, thin] (6,0) -- (9,0);

        \draw[dots] (0.5,1.1) -- (0.5, 3.9);
        \draw[dots] (5.5,1.1) -- (5.5, 3.9);
        \draw[dots] (1.1, 0.5) -- (4.9, 0.5);
        \draw[dots] (1.1, 4.5) -- (4.9, 4.5);

        \draw[dots] (6.5,1.1) -- (6.5, 3.9);
        \draw[dots] (1.1, 5.5) -- (4.9, 5.5);
        \draw[dots] (0.5,6.1) -- (0.5, 6.9);
        \draw[dots] (5.5,6.1) -- (5.5, 6.9);
        \draw[dots] (7.1,0.5) -- (8.9, 0.5);
        \draw[dots] (7.1,4.5) -- (8.9, 4.5);

        \draw [black, thin, fill = gray] (1.25,2) rectangle (4.5,3.5);
        \draw [black, thin] (4.5,3.5)--(7.5,3.5);
        \draw [black, thin] (4.5,2)--(7.5,2);
        \draw [black, thin, fill = cyan] (1.25,3.5) rectangle (4.5,6.25);
        \draw [black, thin] (4.5,3.5)--(4.5,6.75);

        \begin{scope}[color=black]
            \node[anchor=center] () at (7.5,6) {$\mathcal{H}$};
            \node[anchor=center] () at (2.875,2.75) {$\mathcal{K}$};
            \node[anchor=center] () at (2.875,5) {$\mathcal{I}$};
        \end{scope}

        \fill(1.25,2) circle[radius=2pt] node[below left]{$u_1$};
        \fill(4.5,6.75) circle[radius=2pt] node[below right]{$u_2$};
        \fill(7.5,3.5) circle[radius=2pt] node[above right]{$u_3$};
        \fill(4.5,3.5) circle[radius=2pt] node[above right]{$v_1$};
        \fill(1.25,6.25) circle[radius=2pt] node[above left]{$v_2$};
        \fill(7.5,2) circle[radius=2pt] node[below right]{$h$};
        \fill(1.25,3.5) circle[radius=2pt] node[below left]{$g$};

        \fill(0,0) circle[radius=0pt] node[below left]{$A \otimes \mathcal{R}_{m,n}$};
        
    \end{tikzpicture}
    \caption{One possible configuration of inner intervals for this case}
    \label{Fig:Thm3.2.14_Case3Iiva}
    \end{figure}
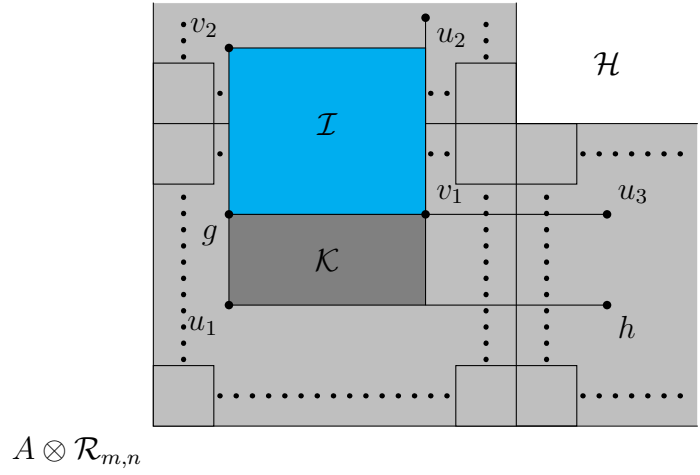

    \item[(b)] $v_1$ and $v_3$ define an inner interval in $\mathcal{Q}$. Let $\mathcal{I}$ denote this inner interval. Suppose that $v_1$ and $v_3$ are the diagonal corners of $\mathcal{I}$. Let $g,h \in V(\mathcal{Q})$ be the anti-diagonal corners of $\mathcal{I}$. As $\mathcal{I}$ is an inner interval of $\mathcal{Q}$, then we have that $x_{v_3} x_{v_1} - x_{g} x_{h} \in I_\mathcal{Q} \subseteq J_\mathcal{Q}$. Note that we can write $f$ in the following way:
    \[ \quad \quad \quad \quad \quad \quad \quad \quad \quad \quad
    f = f^+ - f^- = f^+ - \frac{f^-}{x_{v_1}x_{v_3}}(x_{v_3}x_{v_1} - x_g x_h) - \frac{f^-}{x_{v_1}x_{v_3}}x_g x_h
    \]
    Let $\tilde f = f^+ - \frac{f^-}{x_{v_1}x_{v_3}}x_g x_h$. Since $x_{v_3}x_{v_1} - x_g x_h \in J_\mathcal{Q}$, then $\varphi(x_{v_3} x_{v_1}) = \varphi(x_g x_h)$, so therefore $\varphi(\frac{f^-}{x_{v_1}x_{v_3}}x_g x_h) = \varphi(f^-)$ and as such, we have that $\tilde f \in J_\mathcal{Q}$. For the structure of $\mathcal{Q}$, at least one of the vertex pairs $g,u_2$ or $v_2,h$ defines an inner interval in $\mathcal{Q}$. Suppose $[g,u_2]$ is an inner interval. Then $v_1$ is an anti-diagonal corner of $[g,u_2]$. Thus, applying Lemma \ref{Lemma: 2.2} to $g$, $u_2$ and $v_1$ gives that $\tilde f$ is redundant. Now, suppose that rather $v_2$ and $h$ define an inner interval in $\mathcal{Q}$, say $\mathcal{J}$, then $u_1$ is an anti-diagonal (resp. diagonal) corner of $\mathcal{J}$ given that $v_2$ and $h$ are its diagonal (resp. anti-diagonal) corners (see Figure \ref{Fig:Thm3.2.14_Case3Iivb1}). Thus, applying Lemma \ref{Lemma: 2.2} to $v_2$, $h$ and $u_1$ gives that $\tilde f$ is redundant. In either case, we have that $\tilde f$ is redundant, and as $f = \tilde f - \frac{f^-}{x_{v_1}x_{v_3}}(x_{v_3}x_{v_1} - x_g x_h)$, then we also have that $f$ is redundant, which is a contradiction. 

    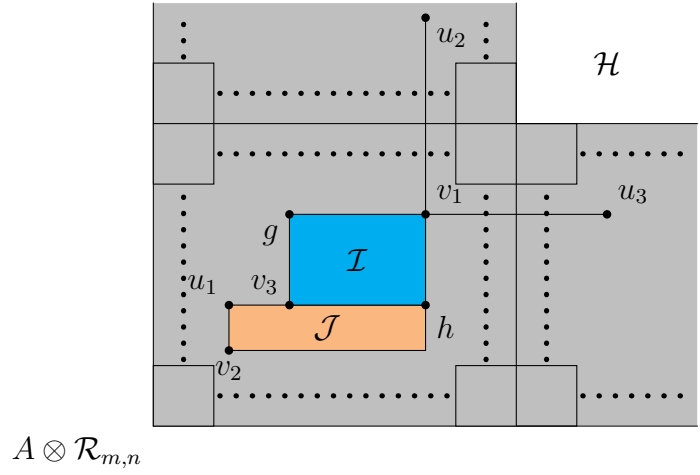
\begin{figure} [H] 
    \centering
    \begin{tikzpicture} [scale = 0.8]
        \draw[step=1cm,white,very thin] (0,0) grid (9,7);
        \draw [white, thin, fill = lightgray] (0,5) rectangle (6,7);
        \draw [white, thin, fill = lightgray] (6,0) rectangle (9,5);
        \draw [black, thin, fill = lightgray] (0,0) rectangle (6,5);
        \polyomino[
        empty cell =x,
        grid,
        p={a}{style={lightgray,draw=black}},
        p={b}{style={gray,draw=black}},
        row sep =;
        ]{
        x x x x x x x x x;
        a x x x x a x x x;
        a x x x x a a x x;
        x x x x x x x x x;
        x x x x x x x x x;
        x x x x x x x x x;
        a x x x x a a x x
        }

        \draw [black, thin] (0,5) -- (0,7);
        \draw [black, thin] (6,5) -- (6,7);
        \draw [black, thin] (6,5) -- (9,5);
        \draw [black, thin] (6,0) -- (9,0);

        \draw[dots] (0.5,1.1) -- (0.5, 3.9);
        \draw[dots] (5.5,1.1) -- (5.5, 3.9);
        \draw[dots] (1.1, 0.5) -- (4.9, 0.5);
        \draw[dots] (1.1, 4.5) -- (4.9, 4.5);

        \draw[dots] (6.5,1.1) -- (6.5, 3.9);
        \draw[dots] (1.1, 5.5) -- (4.9, 5.5);
        \draw[dots] (0.5,6.1) -- (0.5, 6.9);
        \draw[dots] (5.5,6.1) -- (5.5, 6.9);
        \draw[dots] (7.1,0.5) -- (8.9, 0.5);
        \draw[dots] (7.1,4.5) -- (8.9, 4.5);

        \draw [black, thin, fill = Apricot] (1.25,1.25) rectangle (4.5,2);
        \draw [black, thin] (4.5,3.5)--(7.5,3.5);
        \draw [black, thin, fill = cyan] (2.25,2) rectangle (4.5,3.5);
        \draw [black, thin] (4.5,3.5)--(4.5,6.75);

        \begin{scope}[color=black]
            \node[anchor=center] () at (7.5,6) {$\mathcal{H}$};
            \node[anchor=center] () at (3.375,2.75) {$\mathcal{I}$};
            \node[anchor=center] () at (2.875,1.625) {$\mathcal{J}$};
        \end{scope}

        \fill(1.25,2) circle[radius=2pt] node[above left]{$u_1$};
        \fill(4.5,6.75) circle[radius=2pt] node[below right]{$u_2$};
        \fill(7.5,3.5) circle[radius=2pt] node[above right]{$u_3$};
        \fill(4.5,3.5) circle[radius=2pt] node[above right]{$v_1$};
        \fill(1.25,1.25) circle[radius=2pt] node[below]{$v_2$};
        \fill(2.25,2) circle[radius=2pt] node[above left]{$v_3$};
        \fill(4.5,2) circle[radius=2pt] node[below right]{$h$};
        \fill(2.25,3.5) circle[radius=2pt] node[below left]{$g$};

        \fill(0,0) circle[radius=0pt] node[below left]{$A \otimes \mathcal{R}_{m,n}$};
        
    \end{tikzpicture}
    \caption{One possible configuration of inner intervals for this case}
    \label{Fig:Thm3.2.14_Case3Iivb1}
    \end{figure}
    
    Now, suppose that $v_1$ and $v_3$ are instead the anti-diagonal corners of $\mathcal{I}$. Let $g',h' \in V(\mathcal{Q})$ be the diagonal corners of $\mathcal{I}$. As $\mathcal{I}$ is an inner interval of $\mathcal{Q}$, then we have that $x_{g'}x_{h'} - x_{v_1} x_{v_3} \in I_\mathcal{Q} \subseteq J_\mathcal{Q}$. Note that we can write $f$ in the following way:
    \[ \quad \quad \quad \quad \quad \quad \quad \quad \quad \quad
    f = f^+ - f^- = f^+ + \frac{f^-}{x_{v_1}x_{v_3}}(x_{g'}x_{h'} - x_{v_1} x_{v_3}) - \frac{f^-}{x_{v_1}x_{v_3}}x_{g'} x_{h'}
    \]
    Let $\hat f = f^+ - \frac{f^-}{x_{v_1}x_{v_3}}x_{g'} x_{h'}$. Since $x_{g'}x_{h'} - x_{v_1} x_{v_3} \in J_\mathcal{Q}$, then $\varphi(x_{v_1} x_{v_3}) = \varphi(x_{g'} x_{h'})$, so therefore $\varphi(\frac{f^-}{x_{v_1}x_{v_3}}x_{g'} x_{h'}) = \varphi(f^-)$ and as such, we have that $\hat f \in J_\mathcal{Q}$. For the structure of $\mathcal{Q}$, at least one of the vertex pairs $u_1,u_2$ or $v_2,g'$ defines an inner interval in $\mathcal{Q}$. Suppose $[u_1,u_2]$ is an inner interval. Then $g'$ is an anti-diagonal corner of $[u_1,u_2]$. Thus, applying Lemma \ref{Lemma: 2.2} to $u_1$, $u_2$ and $g'$ gives that $\hat f$ is redundant. Now, suppose that rather $v_2$ and $g'$ define an inner interval in $\mathcal{Q}$, say $\mathcal{J'}$, then $u_1$ is an anti-diagonal (resp. diagonal) corner of $\mathcal{J'}$ given that $v_2$ and $g'$ are its diagonal (resp. anti-diagonal) corners (see Figure \ref{Fig:Thm3.2.14_Case3Iivb2}). Thus, applying Lemma \ref{Lemma: 2.2} to $v_2$, $g'$ and $u_1$ gives that $\hat f$ is redundant. In either case, we have that $\hat f$ is redundant, and as $f = \hat f + \frac{f^-}{x_{v_1}x_{v_3}}(x_{g'}x_{h'} - x_{v_1} x_{v_3})$, then we also have that $f$ is redundant, which is a contradiction.

     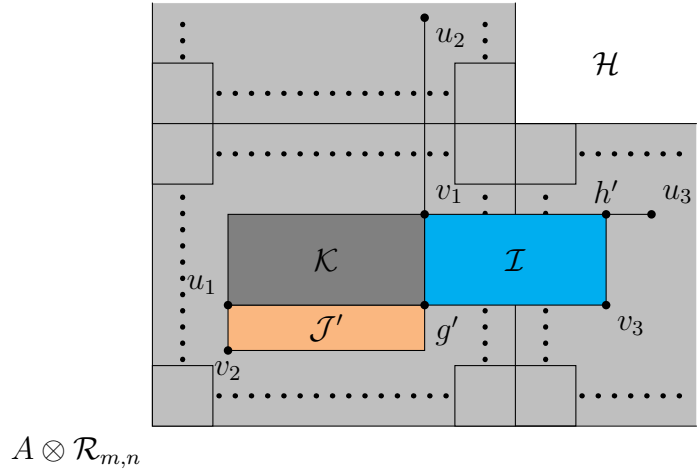
\begin{figure} [H] 
    \centering
    \begin{tikzpicture} [scale = 0.8]
        \draw[step=1cm,white,very thin] (0,0) grid (9,7);
        \draw [white, thin, fill = lightgray] (0,5) rectangle (6,7);
        \draw [white, thin, fill = lightgray] (6,0) rectangle (9,5);
        \draw [black, thin, fill = lightgray] (0,0) rectangle (6,5);
        \polyomino[
        empty cell =x,
        grid,
        p={a}{style={lightgray,draw=black}},
        p={b}{style={gray,draw=black}},
        row sep =;
        ]{
        x x x x x x x x x;
        a x x x x a x x x;
        a x x x x a a x x;
        x x x x x x x x x;
        x x x x x x x x x;
        x x x x x x x x x;
        a x x x x a a x x
        }

        \draw [black, thin] (0,5) -- (0,7);
        \draw [black, thin] (6,5) -- (6,7);
        \draw [black, thin] (6,5) -- (9,5);
        \draw [black, thin] (6,0) -- (9,0);

        \draw[dots] (0.5,1.1) -- (0.5, 3.9);
        \draw[dots] (5.5,1.1) -- (5.5, 3.9);
        \draw[dots] (1.1, 0.5) -- (4.9, 0.5);
        \draw[dots] (1.1, 4.5) -- (4.9, 4.5);

        \draw[dots] (6.5,1.1) -- (6.5, 3.9);
        \draw[dots] (1.1, 5.5) -- (4.9, 5.5);
        \draw[dots] (0.5,6.1) -- (0.5, 6.9);
        \draw[dots] (5.5,6.1) -- (5.5, 6.9);
        \draw[dots] (7.1,0.5) -- (8.9, 0.5);
        \draw[dots] (7.1,4.5) -- (8.9, 4.5);

        \draw [black, thin, fill = Apricot] (1.25,1.25) rectangle (4.5,2);
        \draw [black, thin, fill = gray] (1.25,2) rectangle (4.5,3.5);
        \draw [black, thin, fill = cyan] (4.5,2) rectangle (7.5,3.5);
        \draw [black, thin] (7.5,3.5)--(8.25,3.5);
        \draw [black, thin] (4.5,3.5)--(4.5,6.75);

        \begin{scope}[color=black]
            \node[anchor=center] () at (7.5,6) {$\mathcal{H}$};
            \node[anchor=center] () at (6,2.75) {$\mathcal{I}$};
            \node[anchor=center] () at (2.875,1.625) {$\mathcal{J'}$};
            \node[anchor=center] () at (2.875,2.75) {$\mathcal{K}$};
        \end{scope}

        \fill(1.25,2) circle[radius=2pt] node[above left]{$u_1$};
        \fill(4.5,6.75) circle[radius=2pt] node[below right]{$u_2$};
        \fill(8.25,3.5) circle[radius=2pt] node[above right]{$u_3$};
        \fill(4.5,3.5) circle[radius=2pt] node[above right]{$v_1$};
        \fill(1.25,1.25) circle[radius=2pt] node[below]{$v_2$};
        \fill(7.5,2) circle[radius=2pt] node[below right]{$v_3$};
        \fill(4.5,2) circle[radius=2pt] node[below right]{$g'$};
        \fill(7.5,3.5) circle[radius=2pt] node[above]{$h'$};

        \fill(0,0) circle[radius=0pt] node[below left]{$A \otimes \mathcal{R}_{m,n}$};
        
    \end{tikzpicture}
    \caption{Another possible configuration of inner intervals for this case}
    \label{Fig:Thm3.2.14_Case3Iivb2}
    \end{figure}
    
    \end{description}
    
    \end{description}

    \item[(II)] $u_1$ and $v_1$ are anti-diagonal corners of $\mathcal{K}$. By a symmetric argument to the one given in the previous case \textbf{(I)}, the contradiction that $f$ is redundant will always be reached.
    
    \end{description}

    \end{description}

    As all cases lead to a contradiction, then it must be that any irredundant binomial in $J_\mathcal{Q}$ is of degree 2. Thus $J_\mathcal{Q} \subseteq I_\mathcal{Q}$, and so $I_\mathcal{Q} = J_\mathcal{Q}$.
\end{proof}

\begin{corollary}
    Let $\mathcal{Q} = \mathcal{P} \otimes \mathcal{R}_{m,n}$ be a dilated closed path with an $(m,n)$-dilated L-configuration for some $m,n \in \mathbb{Z}^+$. Then $I_{\mathcal{Q}}$ is prime.
\end{corollary}

\begin{proof}
    By Theorem \ref{Thm: 3.2.1}, we have that $I_\mathcal{Q} = J_\mathcal{Q}$. As $J_\mathcal{Q}$ is a toric ideal, then it is prime by definition. Thus $I_\mathcal{Q}$ is prime.
\end{proof}

Our goal now is to prove an analogous theorem to Theorem \ref{Thm: 3.2.1} for dilated closed paths with $(m,n)$-dilated ladders of at least 3 steps. Our steps towards proving the next theorem are much the same as they were for the case of dilated closed paths with $(m,n)$-dilated L-configurations, so therefore in a similar fashion, we first give a useful specification of the toric ideal for dilated closed paths with $(m,n)$-dilated ladders of at least 3 steps. 

Let $\mathcal{Q} = \mathcal{P} \otimes \mathcal{R}_{m,n}$ be a closed path. Let $m,n\in\mathbb{Z}^+$ be arbitrary and let $\mathcal{L}_{\mathcal{Q}} = \mathcal{B}\otimes\mathcal{R}_{m,n} = \{\mathcal{B}_i \otimes \mathcal{R}_{m,n}\}_{i=1,...s}$ be an $(m,n)$-dilated ladder in $\mathcal{Q} = \mathcal{P} \otimes \mathcal{R}_{m,n}$ with $s \geq 3$. Let $\mathcal{L}_{\mathcal{P}} = \{\mathcal{B}_i\}_{i=1,...s}$ be the underlying ladder in $\mathcal{P}$ that $\mathcal{L}_{\mathcal{Q}}$ arises from. Without loss of generality we can consider $\mathcal{B}_1,...,\mathcal{B}_s$ to be horizontal blocks that are going down for if this is not the case we can perform the necessary rotations or reflections on $\mathcal{P}$ such that it is. Let the block $\mathcal{B}_{s-1}$ contain $k$ cells which we will label $A_1,...,A_k$. Let $A$ be the cell in $\mathcal{B}_s$ such that $A_k$ and $A$ share a common edge. We denote by $a_0, b_0$ the diagonal corners of $A$ and by $c_0, d_0$ the anti-diagonal corners of $A$. Note that fixing $\mathcal{L}_\mathcal{P}$ in this way fixes $\mathcal{L}_\mathcal{Q}$ such that the $(m,n)$-dilated block $\mathcal{B}_{s-1}\otimes \mathcal{R}_{m,n}$ in $\mathcal{Q}$ consists of a sequence of $k$ cell intervals $A_1 \otimes \mathcal{R}_{m,n},...,A_k \otimes \mathcal{R}_{m,n}$ such that $A_k \otimes \mathcal{R}_{m,n}$ shares a common horizontal edge interval with $A \otimes \mathcal{R}_{m,n}$. Let $a_1,...,a_{km -1}, b$ be the labels for the $mk$ vertices contained in the maximal horizontal edge interval at the bottom of $\mathcal{B}_{s-1}$, that is the maximal horizontal edge interval which contains the common horizontal edge interval between $A_k \otimes \mathcal{R}_{m,n}$ and $A \otimes \mathcal{R}_{m,n}$. Note that $a_{(k-1)m+1} = \Delta_{m,n}(c_0) $ and $b = \Delta_{m,n}(b_0)$. We let $a = \Delta_{m,n}(a_0)$ and $d = \Delta_{m,n}(d_0)$.
    
We set $L_{\mathcal{B}\otimes \mathcal{R}_{m,n}} = \{a_1,...,a_{(k-1)m}\}\cup V(A \otimes \mathcal{R}_{m,n})$. We now consider a specific toric ideal representation of $Q$ as per the definition of such a toric ideal for any polyomino as given in Section \ref{Section:2}. Let $\mathcal{H}$ denote the unique hole of $\mathcal{Q}$ and let $e$ be the lower left corner of $\mathcal{H}$ as defined in Section \ref{Section:2}. Given the orientation of $\mathcal{L}_\mathcal{Q}$ we have that for any $v \in L_{\mathcal{B}\otimes \mathcal{R}_{m,n}}$, it must be that $v \leq e$. Fixing $K$ to be a field, we define the following map:
\begin{eqnarray*}
\alpha: V(\mathcal{Q}) & \rightarrow & K[\{h_i,v_j,w\}\mid i\in I, j\in J] \\
v & \mapsto & h_iv_jw^k
\end{eqnarray*}
with $v\in V_i\cap H_j$ and with $k = 1$ if $v \in L_{\mathcal{B}\otimes \mathcal{R}_{m,n}}$ and $k = 0$ otherwise. Letting $S = K[x_v \mid v\in V(\mathcal{Q})]$ and letting $T_{\mathcal{Q}} = K[\alpha(v) \mid v \in V(\mathcal{Q})]$ be the toric ring of $\mathcal{Q}$, we define the following surjective ring homomorphism:
\begin{eqnarray*}
\varphi: S & \rightarrow & T_{\mathcal{Q}} \\
            x_v & \mapsto & \alpha(v)
\end{eqnarray*}
We define the toric ideal of $\mathcal{Q}$ to be $J_\mathcal{Q} = \ker \varphi$. For the remainder of this section, if $\mathcal{Q}$ is a dilated closed path that contains an $(m,n)$-dilated ladder of at least 3 steps, then when we refer to its toric ideal $J_\mathcal{Q}$, we mean the toric ideal as specified by the above map.

\begin{lemma}
    Let $\mathcal{Q} = \mathcal{P} \otimes \mathcal{R}_{m,n}$ be a dilated closed path with an $(m,n)$-dilated ladder of at least 3 steps for some $m,n \in \mathbb{Z}^+$. Then $I_{\mathcal{Q}} \subseteq J_{\mathcal{Q}}$.
\end{lemma}

\begin{proof}
    The result follows from Lemma \ref{Lemma: ideal containment} as per the proof given in \textbf{\cite{Mascia_Rinaldo_Romeo:2020}}.
\end{proof}

 We note here that the theorem below is a generalization of Theorem 5.2 in \textbf{\cite{Cisto_Navarra:2023}} for closed paths and as such, our proof follows a similar argument to the proof of that theorem.

\begin{theorem} \label{Thm: 3.2.2}
    Let $\mathcal{Q} = \mathcal{P} \otimes \mathcal{R}_{m,n}$ be a dilated closed path with an $(m,n)$-dilated ladder of at least 3 steps for some $m,n \in \mathbb{Z}^+$. Then $I_{\mathcal{Q}} = J_{\mathcal{Q}}$.
\end{theorem}

\begin{proof}
    By the previous proposition, we have that $I_\mathcal{Q} \subseteq J_\mathcal{Q}$. We now prove that $J_\mathcal{Q} \subseteq I_\mathcal{Q}$. As in the proof of Theorem \ref{Thm: 3.2.1}, we proceed by equivalently proving the following two statements:
    \begin{itemize}
        \item[(1)] every binomial of degree two in $J_\mathcal{Q}$ is in $I_\mathcal{Q}$;
        \item[(2)] every irredundant binomial in $J_\mathcal{Q}$ is of degree 2.
    \end{itemize}

    We let $\mathcal{B} \otimes \mathcal{R}_{m,n} = \{ \mathcal{B}_i \otimes \mathcal{R}_{m,n}\}_{i = 1,...,s}$ with $s \geq 3$ be an $(m,n)$-dilated ladder configuration contained in $\mathcal{Q}$. Without loss of generality we assume that each $\mathcal{B}_i \otimes \mathcal{R}_{m,n}$ is a horizontal dilated block and that the dilated ladder is going down. Again, we use $A_1 \otimes \mathcal{R}_{m,n},...,A_k \otimes \mathcal{R}_{m,n}$ to denote the cell intervals, from left to right, of the dilated block $\mathcal{B}_{s-1} \otimes \mathcal{R}_{m,n}$, and $A \otimes \mathcal{R}_{m,n}$ to denote the cell interval of $\mathcal{B}_{s} \otimes \mathcal{R}_{m,n}$ that shares a common horizontal edge interval with $A_k \otimes \mathcal{R}_{m,n}$. As well, we set $L_{\mathcal{B} \otimes \mathcal{R}_{m,n}} = \{a_1,...,a_{(k-1)m}\}\cup V(A \otimes \mathcal{R}_{m,n})$ where $a_{\ell m +1},...,a_{(\ell+1) m}$ are the vertices of $A_{\ell +1} \otimes \mathcal{R}_{m,n}$, for $0 \leq \ell \leq k-2$, that are on the bottom maximal horizontal edge interval of $\mathcal{B}_{s-1}$.

    We first prove (1). We can use a similar argument here as was used to prove (1) for Theorem \ref{Thm: 3.2.1}. Let $f = x_q x_r-x_s x_t$ be an arbitrary nonzero binomial in $J_Q$. Let $h_q,v_q$ and let $h_r,v_r$ be the variables associated to the maximal horizontal and vertical edge intervals which contain $q$ and $r$, respectively. Since $f \in J_\mathcal{Q}$, then $\varphi(x_q x_r) = w^k h_q v_q h_r v_r = \varphi(x_s x_t)$ with $k \in \{0,1,2\}$. Suppose $h_q = h_r$, then $q$ and $r$ lie on the same maximal horizontal edge interval, say $H_q$. So then $h_q^2 \mid \varphi(x_q x_r) = \varphi(x_s x_t)$, and so it must be that $s$ and $t$ also lie on $H_q$. Since we must have that $v_q \mid \varphi(x_s x_t)$ and $v_r \mid \varphi(x_s x_t)$, then one of $s$ or $t$ lies on $V_q$ while the other lies on $V_r$. So either we have $q = s$ and $r = t$ or $q = t$ and $r = s$, in either case $f = x_q x_r-x_s x_t = 0$, but this contradicts that $f$ is a nonzero binomial, so therefore $h_q \neq h_r$. A symmetric argument gives $v_q \neq v_r$. So, without loss of generality, we have that $[q,r]$ is a proper interval of $\mathbb{N}^2$, for if it is not, we can rotate or reflect $\mathcal{Q}$ such that it is. Then $s$ and $t$ must be the anti-diagonal corners of $[q,r]$. Without loss of generality, assume $s$ is the upper left corner of and $t$ is the lower right corner of $[q,r]$, then we have that $[q,s] \subseteq V_q$, $[s,r] \subseteq H_r$, $[t,r] \subseteq V_r$, and $[q,t] \subseteq H_q$. Suppose $[q,r]$ is not an inner interval of $\mathcal{Q}$, then there exists some cell $C$ contained in $[q,r]$ such that $C \notin \mathcal{Q}$. As $[q,s]$, $[s,r]$, $[t,r]$, and $[q,t]$ are edge intervals in $\mathcal{Q}$, then it must be that $C \in \mathcal{H}$, where $\mathcal{H}$ is the unique hole of $\mathcal{Q}$. Therefore, $[q,r]$ contains the hole $\mathcal{H}$. As $[q,r]$ contains $\mathcal{H}$, then at least one of the corners $q$, $r$, $s$, or $t$ is in $L_{\mathcal{B} \otimes \mathcal{R}_{m,n}}$. Without loss of generality, assume $q \in L_{\mathcal{B} \otimes \mathcal{R}_{m,n}}$. Note that $s,t \notin L_{\mathcal{B} \otimes \mathcal{R}_{m,n}} $, because if so, $[q,r]$ could not contain $\mathcal{H}$. However, if $q \in  L_{\mathcal{B} \otimes \mathcal{R}_{m,n}}$, then $w \mid \varphi(x_q x_r) = \varphi(x_s x_t) = \varphi(x_s) \varphi (x_t)$, so either $w \mid \varphi(x_s)$ or $w \mid \varphi(x_t)$. If $w \mid \varphi(x_s)$, then $s \in  L_{\mathcal{B} \otimes \mathcal{R}_{m,n}}$, which is a contradiction. If rather $w \mid \varphi(x_t)$, then $t \in  L_{\mathcal{B} \otimes \mathcal{R}_{m,n}}$, which is a contradiction. Thus, it must be that $[q,r]$ is an inner interval of $\mathcal{Q}$, and so therefore $f \in I_\mathcal{Q}$. As $f$ was arbitrary, then we have that every binomial of degree two in $J_\mathcal{Q}$ is in $I_\mathcal{Q}$.

    We now prove (2). Towards a contradiction, suppose $f \in J_\mathcal{Q}$ such that $f$ is irredundant and deg $f \geq 3$. Let $f = f^+ - f^-$. Suppose that 
    \[
    V_f^+ \cap L_{\mathcal{B} \otimes \mathcal{R}_{m,n}} = V_f^- \cap L_{\mathcal{B} \otimes \mathcal{R}_{m,n}} = \emptyset.
    \]
    We let $\mathcal{Q}' \subset \mathcal{Q}$ be the subpolyomino of $\mathcal{Q}$ defined by 
    \[
    \mathcal{Q}' = \{C \in \mathcal{Q} \text{ }\vert \text{ } V(C)\cap L_{\mathcal{B} \otimes \mathcal{R}_{m,n}} = \emptyset \}.
    \]
    Note that $V_f^+, V_f^- \subseteq V(\mathcal{Q}')$. We show that $\mathcal{Q}'$ is a simple polyomino. Let $b \in V(\mathcal{Q})$ be the upper right corner of $A \otimes \mathcal{R}_{m,n}$. By the structure of $\mathcal{Q}$, we have that there exists some cell $E \in \mathcal{H}$ such that $b$ is its lower left corner. As well, there exists some cell $F \in A_k \otimes \mathcal{R}_{m,n}$ such that $b$ is its lower right corner and $E \cap F = [b, b+(0,1)]$. Since $b \in V(F)$, then $F\notin \mathcal{Q}'$. Let $C,D \notin \mathcal{Q}'$ be arbitrary. If both $C,D \in (\ext \mathcal{Q}) \cup (\mathcal{Q} \setminus \mathcal{Q'})$ or both $C,D \in \mathcal{H}$ then there is an obvious walk of cells not in $\mathcal{Q}'$ that connects them. If rather, one of $C$ or $D$ is in $\mathcal{H}$ while the other is in $(\ext \mathcal{Q}) \cup (\mathcal{Q} \setminus \mathcal{Q'})$, then there exists a walk of cells not in $\mathcal{Q}'$ given by $\mathcal{W}: C,...,E,F,...,D$ that connects them. Therefore, any two cells not in $\mathcal{Q}'$ are connected and so $\mathcal{Q}'$ is a simple polyomino. Let $\varphi'$ be the restriction of the map $\varphi$ onto $K[v_a \mid a\in V(\mathcal{Q}) \setminus L_{\mathcal{B} \otimes \mathcal{R}_{m,n}}]$ and let $J_{\mathcal{Q}'} = \ker \varphi'$. As $\mathcal{Q}'$ is simple, then by Theorem 2.2 in \textbf{\cite{Qureshi_Shibuta_Shikama:2017}}, we have that $I_{\mathcal{Q}'}$, the polyomino ideal attached to $\mathcal{Q}'$, is prime and therefore $I_{\mathcal{Q}'} = J_{\mathcal{Q}'}$. So then $f \in I_{\mathcal{Q}'}$, but then we can write $f$ as a linear combination of the binomial generators of $I_{\mathcal{Q}'}$, which are all of degree 2 by definition. So $f$ is redundant, but this is a contradiction. Therefore it must be that either $V_f^+ \cap L_{\mathcal{B} \otimes \mathcal{R}_{m,n}} \neq \emptyset$ or $V_f^- \cap L_{\mathcal{B} \otimes \mathcal{R}_{m,n}} \neq \emptyset$. 
    
    Suppose $V_f^+ \cap L_{\mathcal{B} \otimes \mathcal{R}_{m,n}} \neq \emptyset$. So there exists $u_1 \in L_{\mathcal{B} \otimes \mathcal{R}_{m,n}}$ such that $x_{u_1} \mid f^+$. Since $u_1 \in L_{\mathcal{B} \otimes \mathcal{R}_{m,n}}$, then $w \mid \varphi(u_1)$, so then $w \mid \varphi(f^+) = \varphi(f^-)$. Therefore, there exists some $v_1 \in  L_{\mathcal{B} \otimes \mathcal{R}_{m,n}}$ such that $x_{v_1} \mid f^-$. If we were rather to suppose that $V_f^- \cap V(A \otimes \mathcal{R}_{m,n}) \neq \emptyset$, we would symmetrically obtain that there must exist some $u_1, v_1 \in L_{\mathcal{B} \otimes \mathcal{R}_{m,n}}$ such that $x_{u_1} \mid f^+$ and $x_{v_1} \mid f^-$. If $u_1 = v_1$, then $f = x_{u_1}(\tilde f^+ - \tilde f^-)$, with $\tilde f^+ = \frac {f^+}{x_{u_1}}$ and $\tilde f^- = \frac {f^-}{x_{u_1}}$. Note that $\tilde f^+ - \tilde f^- \in J_\mathcal{Q}$ and deg$(\tilde f^+ - \tilde f^-) <$ deg $f$. So $f$ is redundant, but this is a contradiction. Therefore, $u_1 \neq v_1$. Let $V_{u_1}$ and $H_{u_1}$ be the maximal vertical and horizontal edge intervals which contain $u_1$. Then $v_{u_1} \mid \varphi(f^+) = \varphi(f^-)$, so there exists some vertex $v_2$ that lies on $V_{u_1}$ such that $x_{v_2} \mid f^-$. We also have that $h_{u_1} \mid \varphi(f^+) = \varphi(f^-)$, so there exists some vertex $v_3$ that lies on $H_{u_1}$ such that $x_{v_3} \mid f^-$. Now, let $V_{v_1}$ and $H_{v_1}$ be the maximal vertical and horizontal edge intervals which contain $v_1$. Then $v_{v_1} \mid \varphi(f^-) = \varphi(f^+)$, so there exists some vertex $u_2$ that lies on $V_{v_1}$ such that $x_{u_2} \mid f^+$. We also have that $h_{v_1} \mid \varphi(f^-) = \varphi(f^+)$, so there exists some vertex $u_3$ that lies on $H_{v_1}$ such that $x_{u_3} \mid f^+$. Note that by the same argument as above, if $f$ is irredundant, then we must have that $u_2 \neq v_1$, $u_3 \neq v_1$, $v_2 \neq u_1$, and $v_3 \neq u_1$. The following cases could occur:

    \begin{description}
        \item[(1)] $u_1$ and $v_1$ lie on the same vertical edge interval. Then it must be the case that $u_1,v_1 \in V(A \otimes \mathcal{R}_{m,n})$. For the structure of $\mathcal{Q}$, either $u_1$ and $u_3$ define an inner interval or $v_1$ and $v_3$ do. If $u_1$ and $u_3$ define an inner interval, say $\mathcal{I}$, then applying Lemma \ref{Lemma: 2.2} to $u_1$, $u_3$, and $v_1$ gives that $f$ is redundant, which is a contradiction (see Figure \ref{Fig:Thm3.2.17_Case1}). If, rather, $v_1$ and $v_3$ define an inner interval, say $\mathcal{J}$, then applying Lemma \ref{Lemma: 2.2} to $v_1$, $v_3$, and $u_1$ gives that $f$ is redundant, which is a contradiction.

    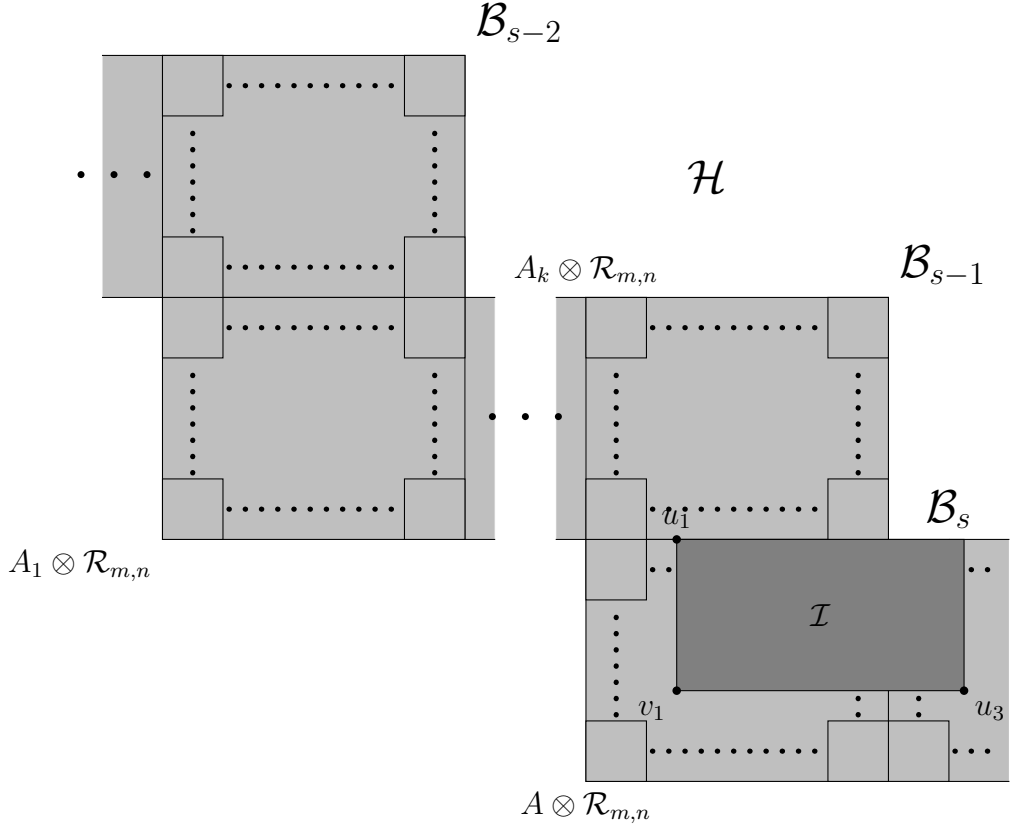
\begin{figure} [H] 
    \centering
    \begin{tikzpicture} [scale = 0.8]
        \draw[step=1cm,white,very thin] (0,0) grid (15,12);

        \draw [white, thin, fill = lightgray] (0,8) rectangle (1,12);
        \draw [white, thin, fill = lightgray] (6,4) rectangle (6.5,8);
        \draw [white, thin, fill = lightgray] (7.5,4) rectangle (8,8);
        \draw [white, thin, fill = lightgray] (13,0) rectangle (15,4);
        \draw [black, thin, fill = lightgray] (1,8) rectangle (6,12);
        \draw [black, thin, fill = lightgray] (1,4) rectangle (6,8);
        \draw [black, thin, fill = lightgray] (8,4) rectangle (13,8);
        \draw [black, thin, fill = lightgray] (8,0) rectangle (13,4);
        \polyomino[
        empty cell =x,
        grid,
        p={a}{style={lightgray,draw=black}},
        p={b}{style={gray,draw=black}},
        row sep =;
        ]{
        x a x x x a x x x x x x x x x;
        x x x x x x x x x x x x x x x;
        x x x x x x x x x x x x x x x;
        x a x x x a x x x x x x x x x;
        x a x x x a x x a x x x a x x;
        x x x x x x x x x x x x x x x;
        x x x x x x x x x x x x x x x;
        x a x x x a x x a x x x a x x;
        x x x x x x x x a x x x a a x;
        x x x x x x x x x x x x x x x;
        x x x x x x x x x x x x x x x;
        x x x x x x x x a x x x a a x
        }

        \draw [black, thin] (0,12)--(1,12);
        \draw [black, thin] (0,8)--(1,8);
        \draw [black, thin] (6,8)--(6.5,8);
        \draw [black, thin] (7.5,8)--(8,8);
        \draw [black, thin] (6,4)--(6.5,4);
        \draw [black, thin] (7.5,4)--(8,4);
        \draw [black, thin] (14,4)--(15,4);
        \draw [black, thin] (14,0)--(15,0);

        \draw[dots] (2.1, 11.5) -- (4.9, 11.5);
        \draw[dots] (2.1, 8.5) -- (4.9, 8.5);
        \draw[dots] (2.1,7.5) -- (4.9, 7.5);
        \draw[dots] (2.1, 4.5) -- (4.9, 4.5);
        \draw[dots] (9.1, 7.5) -- (11.9, 7.5);
        \draw[dots] (9.1,4.5) -- (11.9, 4.5);
        \draw[dots] (9.1,3.5) -- (11.9, 3.5);
        \draw[dots] (9.1,0.5) -- (11.9, 0.5);
        \draw[dots] (14.1,3.5) -- (14.9, 3.5);
        \draw[dots] (14.1,0.5) -- (14.9, 0.5);

        \draw[dots] (1.5,9.1) -- (1.5, 10.9);
        \draw[dots] (5.5,9.1) -- (5.5, 10.9);
        \draw[dots] (1.5,5.1) -- (1.5, 6.9);
        \draw[dots] (5.5,5.1) -- (5.5, 6.9);
        \draw[dots] (8.5,5.1) -- (8.5, 6.9);
        \draw[dots] (12.5,5.1) -- (12.5, 6.9);
        \draw[dots] (8.5,1.1) -- (8.5, 2.9);
        \draw[dots] (12.5,1.1) -- (12.5, 2.9);
        \draw[dots] (13.5,1.1) -- (13.5, 2.9);

        \draw [black, thin, fill = gray] (9.5,1.5) rectangle (14.25,4);

          \begin{scope}[color=black]
            \node[anchor=center] () at (0.2,10) {\huge $\cdot \cdot \cdot$};
            \node[anchor=center] () at (7,6) {\huge $\cdot \cdot \cdot$};
            
            \node[anchor=center] () at (10,10) {\Large $\mathcal{H}$};
            \node[above right] () at (6,12) {\Large $\mathcal{B}_{s-2}$};
            \node[above right] () at (13,8) {\Large $\mathcal{B}_{s-1}$};
            \node[above] () at (14,4) {\Large $\mathcal{B}_s$};

            \node[anchor=center] () at (11.875,2.75) { $\mathcal{I}$};
        \end{scope}

        \fill(9.5,4) circle[radius=2pt] node[above]{$u_1$};
        \fill(9.5,1.5) circle[radius=2pt] node[below left]{$v_1$};
        \fill(14.25,1.5) circle[radius=2pt] node[below right]{$u_3$};

        \fill(8,0) circle[radius=0pt] node[below]{$A \otimes \mathcal{R}_{m,n}$};
        \fill(8,8) circle[radius=0pt] node[above]{$A_k \otimes \mathcal{R}_{m,n}$};
        \fill(1,4) circle[radius=0pt] node[below left]{$A_1 \otimes \mathcal{R}_{m,n}$};
        
    \end{tikzpicture}
    \caption{One possible configuration of an inner interval on the last three steps of the dilated ladder $\mathcal{B} \otimes \mathcal{R}_{m,n}$ that gives a contradiction for this case}
    \label{Fig:Thm3.2.17_Case1}
    \end{figure}

        \item[(2)] $u_1$ and $v_1$ lie on the same horizontal edge interval. In this case, it must be that either $u_1, v_1 \in V(A \otimes \mathcal{R}_{m,n})$ or $u_1,v_1 \in \{a_1,...,a_{(k-1)m}\}$. If $u_1, v_1 \in V(A \otimes \mathcal{R}_{m,n})$, then for the structure of $\mathcal{Q}$, either $u_1$ and $u_2$ define an inner interval or $v_1$ and $v_2$ do. If $u_1$ and $u_2$ define an inner interval, then applying Lemma \ref{Lemma: 2.2} to $u_1$, $u_2$, and $v_1$ gives that $f$ is redundant, which is a contradiction. If, rather, $v_1$ and $v_2$ define an inner interval, then applying Lemma \ref{Lemma: 2.2} to $v_1$, $v_2$, and $u_1$ gives that $f$ is redundant, which is a contradiction. Suppose now that both $u_1,v_1 \in \{a_1,...,a_{(k-1)m}\}$, then for the structure of $\mathcal{Q}$, either $u_1$ and $u_2$ define an inner interval or $v_1$ and $v_2$ do. If $u_1$ and $u_2$ define an inner interval, say $\mathcal{I}$, then applying Lemma \ref{Lemma: 2.2} to $u_1$, $u_2$, and $v_1$ gives that $f$ is redundant, which is a contradiction. If, rather, $v_1$ and $v_2$ define an inner interval, say $\mathcal{J}$, then applying Lemma \ref{Lemma: 2.2} to $v_1$, $v_2$, and $u_1$ gives that $f$ is redundant, which is a contradiction (see Figure \ref{Fig:Thm3.2.17_Case2}).
        
    \begin{figure} [H]
    \centering
    \begin{tikzpicture} [scale = 0.8]
        \draw[step=1cm,white,very thin] (0,0) grid (15,12);

        \draw [white, thin, fill = lightgray] (0,8) rectangle (1,12);
        \draw [white, thin, fill = lightgray] (6,4) rectangle (6.5,8);
        \draw [white, thin, fill = lightgray] (7.5,4) rectangle (8,8);
        \draw [white, thin, fill = lightgray] (13,0) rectangle (15,4);
        \draw [black, thin, fill = lightgray] (1,8) rectangle (6,12);
        \draw [black, thin, fill = lightgray] (1,4) rectangle (6,8);
        \draw [black, thin, fill = lightgray] (8,4) rectangle (13,8);
        \draw [black, thin, fill = lightgray] (8,0) rectangle (13,4);
        \polyomino[
        empty cell =x,
        grid,
        p={a}{style={lightgray,draw=black}},
        p={b}{style={gray,draw=black}},
        row sep =;
        ]{
        x a x x x a x x x x x x x x x;
        x x x x x x x x x x x x x x x;
        x x x x x x x x x x x x x x x;
        x a x x x a x x x x x x x x x;
        x a x x x a x x a x x x a x x;
        x x x x x x x x x x x x x x x;
        x x x x x x x x x x x x x x x;
        x a x x x a x x a x x x a x x;
        x x x x x x x x a x x x a a x;
        x x x x x x x x x x x x x x x;
        x x x x x x x x x x x x x x x;
        x x x x x x x x a x x x a a x
        }

        \draw [black, thin] (0,12)--(1,12);
        \draw [black, thin] (0,8)--(1,8);
        \draw [black, thin] (6,8)--(6.5,8);
        \draw [black, thin] (7.5,8)--(8,8);
        \draw [black, thin] (6,4)--(6.5,4);
        \draw [black, thin] (7.5,4)--(8,4);
        \draw [black, thin] (14,4)--(15,4);
        \draw [black, thin] (14,0)--(15,0);

        \draw[dots] (2.1, 11.5) -- (4.9, 11.5);
        \draw[dots] (2.1, 8.5) -- (4.9, 8.5);
        \draw[dots] (2.1,7.5) -- (4.9, 7.5);
        \draw[dots] (2.1, 4.5) -- (4.9, 4.5);
        \draw[dots] (9.1, 7.5) -- (11.9, 7.5);
        \draw[dots] (9.1,4.5) -- (11.9, 4.5);
        \draw[dots] (9.1,3.5) -- (11.9, 3.5);
        \draw[dots] (9.1,0.5) -- (11.9, 0.5);
        \draw[dots] (14.1,3.5) -- (14.9, 3.5);
        \draw[dots] (14.1,0.5) -- (14.9, 0.5);

        \draw[dots] (1.5,9.1) -- (1.5, 10.9);
        \draw[dots] (5.5,9.1) -- (5.5, 10.9);
        \draw[dots] (1.5,5.1) -- (1.5, 6.9);
        \draw[dots] (5.5,5.1) -- (5.5, 6.9);
        \draw[dots] (8.5,5.1) -- (8.5, 6.9);
        \draw[dots] (12.5,5.1) -- (12.5, 6.9);
        \draw[dots] (8.5,1.1) -- (8.5, 2.9);
        \draw[dots] (12.5,1.1) -- (12.5, 2.9);
        \draw[dots] (13.5,1.1) -- (13.5, 2.9);

          \begin{scope}[color=black]
            \node[anchor=center] () at (0.2,10) {\huge $\cdot \cdot \cdot$};
            \node[anchor=center] () at (7,7) {\huge $\cdot \cdot \cdot$};
        \end{scope}

        \draw [black, thin, fill = gray] (3,4) rectangle (10,6.5);

        \begin{scope}[color=black]
            \node[anchor=center] () at (10,10) {\Large $\mathcal{H}$};
            \node[above right] () at (6,12) {\Large $\mathcal{B}_{s-2}$};
            \node[above right] () at (13,8) {\Large $\mathcal{B}_{s-1}$};
            \node[above] () at (14,4) {\Large $\mathcal{B}_s$};

            \node[anchor=center] () at (6.5,5.25) { $\mathcal{J}$};
        \end{scope}

        \fill(3,4) circle[radius=2pt] node[below left]{$u_1$};
        \fill(10,4) circle[radius=2pt] node[below right]{$v_1$};
        \fill(3,6.5) circle[radius=2pt] node[above left]{$v_2$};

        \fill(8,0) circle[radius=0pt] node[below]{$A \otimes \mathcal{R}_{m,n}$};
        \fill(8,8) circle[radius=0pt] node[above]{$A_k \otimes \mathcal{R}_{m,n}$};
        \fill(1,4) circle[radius=0pt] node[below left]{$A_1 \otimes \mathcal{R}_{m,n}$};

        \end{tikzpicture}
    \caption{One possible configuration of an inner interval on the last three steps of the dilated ladder $\mathcal{B} \otimes \mathcal{R}_{m,n}$ that gives a contradiction for this case}
    \label{Fig:Thm3.2.17_Case2}
    \end{figure}
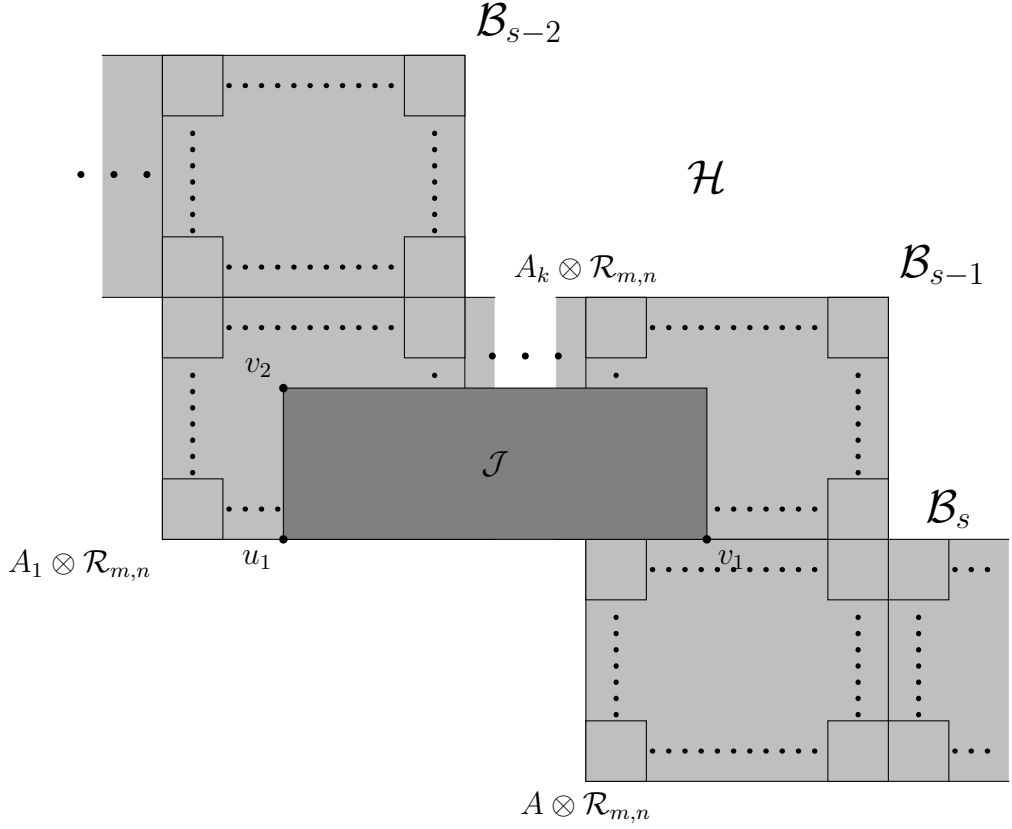

        \item[(3)] $u_1$ and $v_1$ lie on distinct vertical and horizontal edge intervals. This case breaks into the following three subcases:

        \begin{description}
            \item[(I)] $u_1, v_1 \in V(A \otimes \mathcal{R}_{m,n})$. We can use similar arguments in case (3) of the proof of Theorem \ref{Thm: 3.2.1} to show that $f$ must be redundant, which is a contradiction.

            \item[(II)] $u_1 \in \{a_1,...,a_{(k-1)m}\}$ and $v_1 \in V(A\otimes \mathcal{R}_{m,n}) \setminus V(A_k \otimes \mathcal{R}_{m,n})$. Investigating the possible positions of $u_2$ gives the two following subcases:

            \begin{description}
                \item[(i)] $u_2 \in V(A \otimes \mathcal{R}_{m,n})$. Then for the structure of $\mathcal{Q}$, $u_2$ and $u_3$ define an inner interval, say $\mathcal{I}$ (see Figure \ref{Fig:Thm3.2.17_Case3IIi}). Applying Lemma \ref{Lemma: 2.2} to $u_2$,$u_3$, and $v_1$ gives that $f$ is redundant which is a contradiction.

    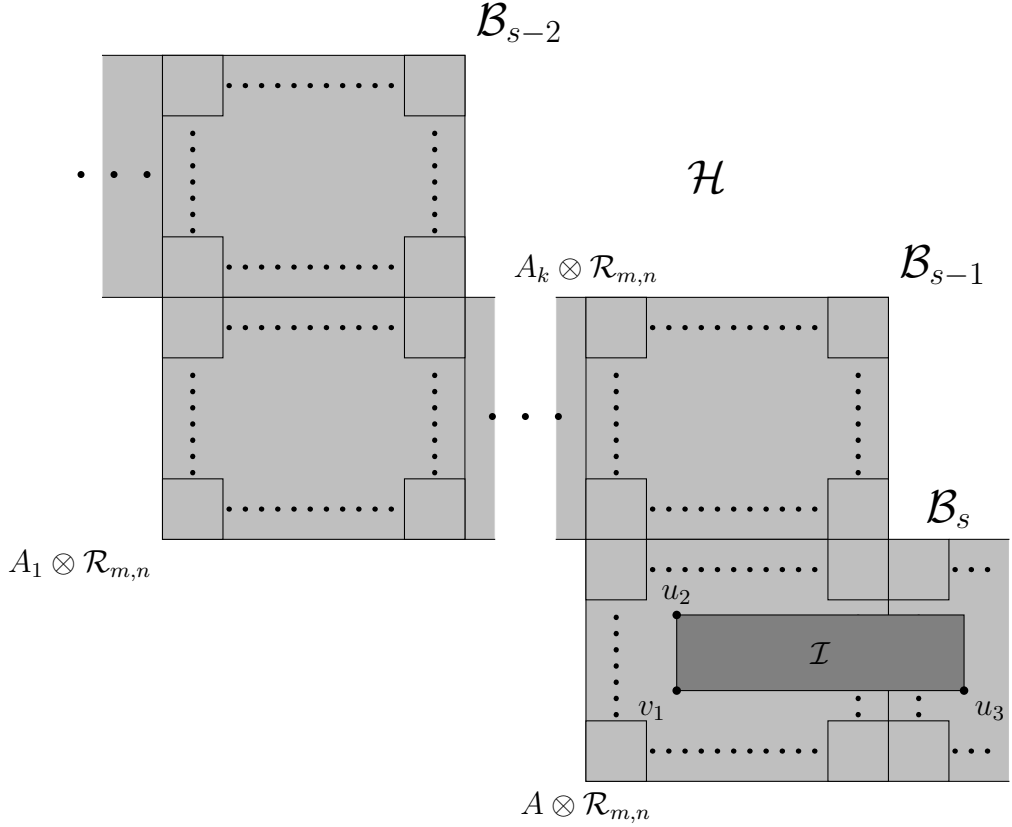
\begin{figure} [H]
    \centering
    \begin{tikzpicture} [scale = 0.8]
        \draw[step=1cm,white,very thin] (0,0) grid (15,12);

        \draw [white, thin, fill = lightgray] (0,8) rectangle (1,12);
        \draw [white, thin, fill = lightgray] (6,4) rectangle (6.5,8);
        \draw [white, thin, fill = lightgray] (7.5,4) rectangle (8,8);
        \draw [white, thin, fill = lightgray] (13,0) rectangle (15,4);
        \draw [black, thin, fill = lightgray] (1,8) rectangle (6,12);
        \draw [black, thin, fill = lightgray] (1,4) rectangle (6,8);
        \draw [black, thin, fill = lightgray] (8,4) rectangle (13,8);
        \draw [black, thin, fill = lightgray] (8,0) rectangle (13,4);
        \polyomino[
        empty cell =x,
        grid,
        p={a}{style={lightgray,draw=black}},
        p={b}{style={gray,draw=black}},
        row sep =;
        ]{
        x a x x x a x x x x x x x x x;
        x x x x x x x x x x x x x x x;
        x x x x x x x x x x x x x x x;
        x a x x x a x x x x x x x x x;
        x a x x x a x x a x x x a x x;
        x x x x x x x x x x x x x x x;
        x x x x x x x x x x x x x x x;
        x a x x x a x x a x x x a x x;
        x x x x x x x x a x x x a a x;
        x x x x x x x x x x x x x x x;
        x x x x x x x x x x x x x x x;
        x x x x x x x x a x x x a a x
        }

        \draw [black, thin] (0,12)--(1,12);
        \draw [black, thin] (0,8)--(1,8);
        \draw [black, thin] (6,8)--(6.5,8);
        \draw [black, thin] (7.5,8)--(8,8);
        \draw [black, thin] (6,4)--(6.5,4);
        \draw [black, thin] (7.5,4)--(8,4);
        \draw [black, thin] (14,4)--(15,4);
        \draw [black, thin] (14,0)--(15,0);

        \draw[dots] (2.1, 11.5) -- (4.9, 11.5);
        \draw[dots] (2.1, 8.5) -- (4.9, 8.5);
        \draw[dots] (2.1,7.5) -- (4.9, 7.5);
        \draw[dots] (2.1, 4.5) -- (4.9, 4.5);
        \draw[dots] (9.1, 7.5) -- (11.9, 7.5);
        \draw[dots] (9.1,4.5) -- (11.9, 4.5);
        \draw[dots] (9.1,3.5) -- (11.9, 3.5);
        \draw[dots] (9.1,0.5) -- (11.9, 0.5);
        \draw[dots] (14.1,3.5) -- (14.9, 3.5);
        \draw[dots] (14.1,0.5) -- (14.9, 0.5);

        \draw[dots] (1.5,9.1) -- (1.5, 10.9);
        \draw[dots] (5.5,9.1) -- (5.5, 10.9);
        \draw[dots] (1.5,5.1) -- (1.5, 6.9);
        \draw[dots] (5.5,5.1) -- (5.5, 6.9);
        \draw[dots] (8.5,5.1) -- (8.5, 6.9);
        \draw[dots] (12.5,5.1) -- (12.5, 6.9);
        \draw[dots] (8.5,1.1) -- (8.5, 2.9);
        \draw[dots] (12.5,1.1) -- (12.5, 2.9);
        \draw[dots] (13.5,1.1) -- (13.5, 2.9);

        \draw [black, thin, fill = gray] (9.5,1.5) rectangle (14.25,2.75);

          \begin{scope}[color=black]
            \node[anchor=center] () at (0.2,10) {\huge $\cdot \cdot \cdot$};
            \node[anchor=center] () at (7,6) {\huge $\cdot \cdot \cdot$};
            
            \node[anchor=center] () at (10,10) {\Large $\mathcal{H}$};
            \node[above right] () at (6,12) {\Large $\mathcal{B}_{s-2}$};
            \node[above right] () at (13,8) {\Large $\mathcal{B}_{s-1}$};
            \node[above] () at (14,4) {\Large $\mathcal{B}_s$};

            \node[anchor=center] () at (11.875,2.125) { $\mathcal{I}$};
        \end{scope}

        \fill(9.5,2.75) circle[radius=2pt] node[above]{$u_2$};
        \fill(9.5,1.5) circle[radius=2pt] node[below left]{$v_1$};
        \fill(14.25,1.5) circle[radius=2pt] node[below right]{$u_3$};

        \fill(8,0) circle[radius=0pt] node[below]{$A \otimes \mathcal{R}_{m,n}$};
        \fill(8,8) circle[radius=0pt] node[above]{$A_k \otimes \mathcal{R}_{m,n}$};
        \fill(1,4) circle[radius=0pt] node[below left]{$A_1 \otimes \mathcal{R}_{m,n}$};
        
    \end{tikzpicture}
    \caption{One possible configuration of an inner interval on the last three steps of the dilated ladder $\mathcal{B} \otimes \mathcal{R}_{m,n}$ that gives a contradiction for this case}
    \label{Fig:Thm3.2.17_Case3IIi}
    \end{figure}

                \item[(ii)] $u_2 \notin V(A \otimes \mathcal{R}_{m,n})$. Then for the structure of $\mathcal{Q}$, $u_1$ and $u_2$ must define an inner interval, say $\mathcal{I}$ where $u_1$ and $u_2$ are its diagonal corners, so $\mathcal{I} = [u_1, u_2]$. Let $g,h \in V(\mathcal{Q})$ be the anti-diagonal corners of $\mathcal{I}$. As $\mathcal{I}$ is an inner interval of $\mathcal{Q}$, then we have that $x_{u_1}x_{u_2} - x_{g} x_{h} \in I_\mathcal{Q} \subseteq J_\mathcal{Q}$. Note that we can write $f$ in the following way:
                \[ \quad \quad \quad \quad \quad \quad \quad
                f = f^+ - f^- = \frac{f^+}{x_{u_1} x_{u_2}}(x_{u_1} x_{u_2} - x_g x_h) + \frac{f^+}{x_{u_1} x_{u_2}}x_g x_h - f^-
                \]
    
                 Let $\tilde f = \frac{f^+}{x_{u_1} x_{u_2}}x_g x_h - f^-$. Since $x_{u_1}x_{u_2} - x_g x_h \in J_\mathcal{Q}$, then $\varphi(x_{u_1} x_{u_2}) = \varphi(x_g x_h)$, so therefore $\varphi(\frac{f^+}{x_{u_1}x_{u_2}}x_g x_h) = \varphi(f^+)$ and as such, we have that $\tilde f \in J_\mathcal{Q}$. Note that for the structure of $\mathcal{Q}$, we have that $h, u_3$, and $v_1$ will define an inner interval, say $\mathcal{J}$, where $h$ and $u_3$ are the anti-diagonal corners of $\mathcal{J}$ and $v_1$ is the lower left corner of $\mathcal{J}$ (see Figure \ref{Fig:Thm3.2.17_Case3IIii}). Applying Lemma \ref{Lemma: 2.2} to $h, u_3$, and $v_1$, gives that $\tilde f$ is redundant. As we have that $f = \frac{f^+}{x_{u_1} x_{u_2}}(x_{u_1} x_{u_2} - x_g x_h) + \tilde f$, then this gives that $f$ is redundant, which is a contradiction.

    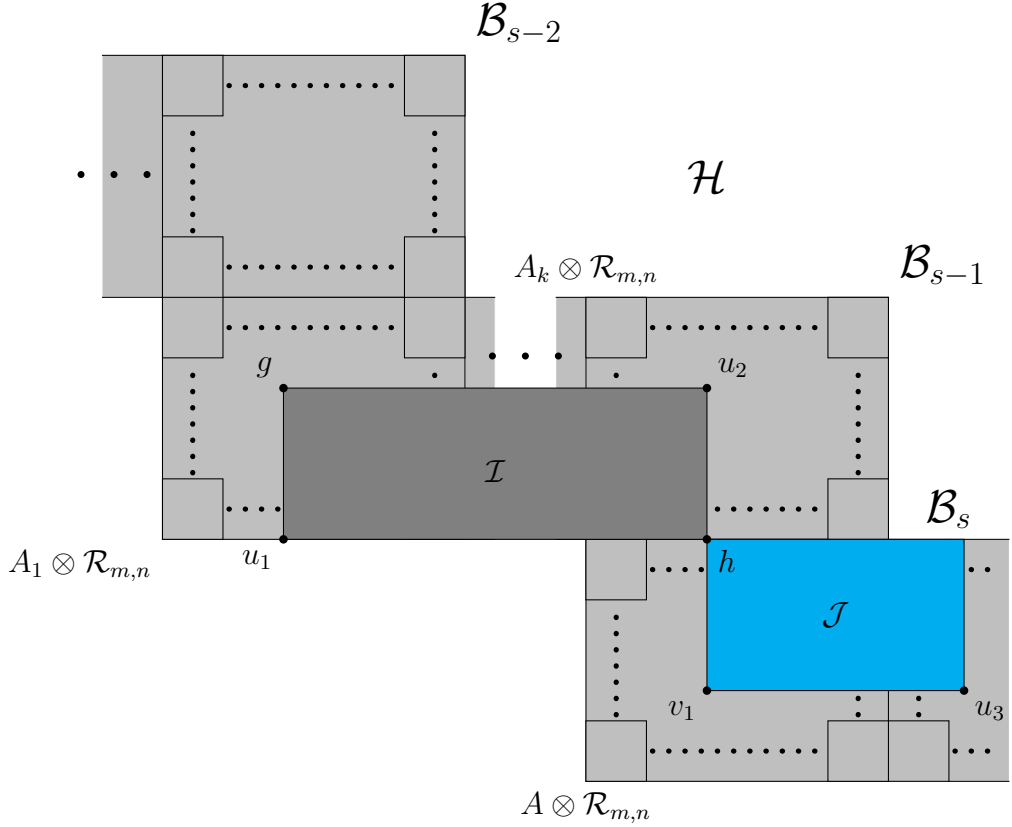
\begin{figure} [H]
    \centering
    \begin{tikzpicture} [scale = 0.8]
        \draw[step=1cm,white,very thin] (0,0) grid (15,12);

        \draw [white, thin, fill = lightgray] (0,8) rectangle (1,12);
        \draw [white, thin, fill = lightgray] (6,4) rectangle (6.5,8);
        \draw [white, thin, fill = lightgray] (7.5,4) rectangle (8,8);
        \draw [white, thin, fill = lightgray] (13,0) rectangle (15,4);
        \draw [black, thin, fill = lightgray] (1,8) rectangle (6,12);
        \draw [black, thin, fill = lightgray] (1,4) rectangle (6,8);
        \draw [black, thin, fill = lightgray] (8,4) rectangle (13,8);
        \draw [black, thin, fill = lightgray] (8,0) rectangle (13,4);
        \polyomino[
        empty cell =x,
        grid,
        p={a}{style={lightgray,draw=black}},
        p={b}{style={gray,draw=black}},
        row sep =;
        ]{
        x a x x x a x x x x x x x x x;
        x x x x x x x x x x x x x x x;
        x x x x x x x x x x x x x x x;
        x a x x x a x x x x x x x x x;
        x a x x x a x x a x x x a x x;
        x x x x x x x x x x x x x x x;
        x x x x x x x x x x x x x x x;
        x a x x x a x x a x x x a x x;
        x x x x x x x x a x x x a a x;
        x x x x x x x x x x x x x x x;
        x x x x x x x x x x x x x x x;
        x x x x x x x x a x x x a a x
        }

        \draw [black, thin] (0,12)--(1,12);
        \draw [black, thin] (0,8)--(1,8);
        \draw [black, thin] (6,8)--(6.5,8);
        \draw [black, thin] (7.5,8)--(8,8);
        \draw [black, thin] (6,4)--(6.5,4);
        \draw [black, thin] (7.5,4)--(8,4);
        \draw [black, thin] (14,4)--(15,4);
        \draw [black, thin] (14,0)--(15,0);

        \draw[dots] (2.1, 11.5) -- (4.9, 11.5);
        \draw[dots] (2.1, 8.5) -- (4.9, 8.5);
        \draw[dots] (2.1,7.5) -- (4.9, 7.5);
        \draw[dots] (2.1, 4.5) -- (4.9, 4.5);
        \draw[dots] (9.1, 7.5) -- (11.9, 7.5);
        \draw[dots] (9.1,4.5) -- (11.9, 4.5);
        \draw[dots] (9.1,3.5) -- (11.9, 3.5);
        \draw[dots] (9.1,0.5) -- (11.9, 0.5);
        \draw[dots] (14.1,3.5) -- (14.9, 3.5);
        \draw[dots] (14.1,0.5) -- (14.9, 0.5);

        \draw[dots] (1.5,9.1) -- (1.5, 10.9);
        \draw[dots] (5.5,9.1) -- (5.5, 10.9);
        \draw[dots] (1.5,5.1) -- (1.5, 6.9);
        \draw[dots] (5.5,5.1) -- (5.5, 6.9);
        \draw[dots] (8.5,5.1) -- (8.5, 6.9);
        \draw[dots] (12.5,5.1) -- (12.5, 6.9);
        \draw[dots] (8.5,1.1) -- (8.5, 2.9);
        \draw[dots] (12.5,1.1) -- (12.5, 2.9);
        \draw[dots] (13.5,1.1) -- (13.5, 2.9);

        \draw [black, thin, fill = gray] (3,4) rectangle (10,6.5);
        \draw [black, thin, fill = cyan] (10,1.5) rectangle (14.25,4);

          \begin{scope}[color=black]
            \node[anchor=center] () at (0.2,10) {\huge $\cdot \cdot \cdot$};
            \node[anchor=center] () at (7,7) {\huge $\cdot \cdot \cdot$};
            
            \node[anchor=center] () at (10,10) {\Large $\mathcal{H}$};
            \node[above right] () at (6,12) {\Large $\mathcal{B}_{s-2}$};
            \node[above right] () at (13,8) {\Large $\mathcal{B}_{s-1}$};
            \node[above] () at (14,4) {\Large $\mathcal{B}_s$};

            \node[anchor=center] () at (6.5,5.125) { $\mathcal{I}$};
            \node[anchor=center] () at (12.125,2.75) { $\mathcal{J}$};
        \end{scope}
        
        \fill(3,4) circle[radius=2pt] node[below left]{$u_1$};
        \fill(10,6.5) circle[radius=2pt] node[above right]{$u_2$};
        \fill(10,1.5) circle[radius=2pt] node[below left]{$v_1$};
        \fill(14.25,1.5) circle[radius=2pt] node[below right]{$u_3$};
        \fill(3,6.5) circle[radius=2pt] node[above left]{$g$};
        \fill(10,4) circle[radius=2pt] node[below right]{$h$};

        \fill(8,0) circle[radius=0pt] node[below]{$A \otimes \mathcal{R}_{m,n}$};
        \fill(8,8) circle[radius=0pt] node[above]{$A_k \otimes \mathcal{R}_{m,n}$};
        \fill(1,4) circle[radius=0pt] node[below left]{$A_1 \otimes \mathcal{R}_{m,n}$};
        
    \end{tikzpicture}
    \caption{One possible configuration of inner intervals on the last three steps of the dilated ladder $\mathcal{B} \otimes \mathcal{R}_{m,n}$ that gives a contradiction for this case}
    \label{Fig:Thm3.2.17_Case3IIii}
    \end{figure}

            \item[(III)] $u_1 \in V(A\otimes \mathcal{R}_{m,n}) \setminus V(A_k \otimes \mathcal{R}_{m,n})$ and $v_1 \in\{a_1,...,a_{(k-1)m}\}$. A symmetric argument to subcase (II) shows that $f$ is redundant, which is a contradiction.
        \end{description}
    \end{description}
    \end{description}

    As all cases lead to a contradiction, then it must be that any irredundant binomial in $J_\mathcal{Q}$ is of degree 2. Thus $J_\mathcal{Q} \subseteq I_\mathcal{Q}$, and so $I_\mathcal{Q} = J_\mathcal{Q}$.
\end{proof}

\begin{corollary}
    Let $\mathcal{Q} = \mathcal{P} \otimes \mathcal{R}_{m,n}$ be a dilated closed path with an $(m,n)$-dilated ladder of at least 3 steps for some $m,n\in \mathbb{Z}^+$. Then $I_{\mathcal{Q}}$ is prime.
\end{corollary}

\begin{proof}
    By Theorem \ref{Thm: 3.2.2}, we have that $I_\mathcal{Q} = J_\mathcal{Q}$. As $J_\mathcal{Q}$ is a toric ideal, then it is prime by definition. Thus $I_\mathcal{Q}$ is prime.
\end{proof}

The final piece of the puzzle we must place in order to prove our main result is of the form of a proposition that appears in \textbf{\cite{Cisto_Navarra:2023}}. For a proof of the following proposition, we refer the reader to Cisto and Navarra's proof of Proposition 6.1 in \textbf{\cite{Cisto_Navarra:2023}}.

\begin{proposition} (\textbf{\cite{Cisto_Navarra:2023}}, Proposition 6.1). \label{Proposition 6.1}
    Let $\mathcal{P}$ be a closed path and suppose that $\mathcal{P}$ has no zig-zag walks. Then $\mathcal{P}$ has an L-configuration or a ladder with at least 3 steps.
\end{proposition}

We are now ready to state our main result.

\begin{theorem} \label{Thm: 3.2.3}
    Let $\mathcal{P}$ be a closed path and let $\mathcal{Q} = \mathcal{P} \otimes \mathcal{R}_{m,n}$ be a dilated closed path for some $m,n\in \mathbb{Z}^+$. If $\mathcal{Q}$ contains no zig-zag walks, then $I_{\mathcal{Q}}$ is prime.
\end{theorem}

\begin{proof}
    As $\mathcal{Q} = \mathcal{P} \otimes \mathcal{R}_{m,n}$ and $\mathcal{Q}$ contains no zig-zag walks, then by Proposition \ref{Prop: Construct a zig-zag}, we have that $\mathcal{P}$ contains no zig-zag walks. By Proposition \ref{Proposition 6.1}, this implies that $\mathcal{P}$ has an L-configuration or a ladder of at least 3 steps. If $\mathcal{P}$ has an L-configuration, then $\mathcal{Q}$ contains an $(m,n)$-dilated L-configuration by definition. Therefore, by Theorem \ref{Thm: 3.2.1}, we have that $I_\mathcal{Q}$ is prime. If rather $\mathcal{P}$ has a ladder of at least 3 steps, then $\mathcal{Q}$ contains an $(m,n)$-dilated ladder of at least 3 steps by definition. Therefore, by Theorem \ref{Thm: 3.2.2}, we have that $I_\mathcal{Q}$ is prime. In either case, it must be that $I_\mathcal{Q}$ is prime.
\end{proof}

We can combine the above result with Corollary 3.6 in \textbf{\cite{Mascia_Rinaldo_Romeo:2020}} to give the following theorem.

\begin{theorem} \label{Thm: ZZW for dilated closed paths}
    Let $\mathcal{P}$ be a closed path and let $\mathcal{Q} = \mathcal{P} \otimes \mathcal{R}_{m,n}$ be a dilated closed path for some $m,n \in \mathbb{Z}^+$. Then $I_{\mathcal{Q}}$ is prime if and only if $\mathcal{Q}$ contains no zig-zag walks.
\end{theorem}

\begin{proof}
    The necessary condition is given by Corollary 3.6 in \textbf{\cite{Mascia_Rinaldo_Romeo:2020}}. The sufficient condition is given by Theorem \ref{Thm: 3.2.3}.
\end{proof}

The above theorem gives that the Zig-Zag Walk Conjecture holds for the class of dilated closed path polyominoes. Therefore, we have fully characterized primality for dilated closed paths. We can in fact assert an even stronger connection between the primality of a closed path and its dilation as given by the following proposition.

\begin{proposition}
    Let $\mathcal{P}$ be a closed path and let $\mathcal{Q} = \mathcal{P} \otimes \mathcal{R}_{m,n}$ for some $m,n \in \mathbb{Z}^+$. If $\mathcal{P}$ contains no zig-zag walks, then $\mathcal{Q}$ contains no zig-zag walks.
\end{proposition}

\begin{proof}
    If $\mathcal{P}$ is a closed path which contains no zig-zag walks, then by Proposition \ref{Proposition 6.1}, $\mathcal{P}$ contains an L-configuration or a ladder of at least 3 steps. Then $\mathcal{Q}$ must contain either a $(m,n)$-dilated L-configuration or a $(m,n)$-dilated ladder of at least 3 steps by their definitions. So by either Theorem \ref{Thm: 3.2.1} or Theorem \ref{Thm: 3.2.2}, $I_\mathcal{Q}$ is prime. Thus by Corollary 3.6 in \textbf{\cite{Mascia_Rinaldo_Romeo:2020}}, we have that $\mathcal{Q}$ contains no zig-zag walks.
\end{proof}

We make note that the above proposition is the converse of Proposition \ref{Prop: Construct a zig-zag} specified to the class of closed paths. Importantly, we obtain from this proposition the following corollary.

\begin{corollary} \label{Cor: prime implies dilated prime closed path}
      Let $\mathcal{P}$ be a closed path and let $\mathcal{Q} = \mathcal{P} \otimes \mathcal{R}_{m,n}$ be a dilated closed path for some $m,n \in \mathbb{Z}^+$. If $I_\mathcal{P}$ is prime, then $I_\mathcal{Q}$ is prime.
\end{corollary}

\begin{proof}
    If $\mathcal{P}$ is prime, then by Corollary 3.6 in \textbf{\cite{Mascia_Rinaldo_Romeo:2020}} we have that $\mathcal{P}$ contains no zig-zag walks. So by Proposition \ref{Proposition 6.1}, $\mathcal{P}$ contains an L-configuration or a ladder of at least 3 steps. Then $\mathcal{Q}$ must contain either a $(m,n)$-dilated L-configuration or a $(m,n)$-dilated ladder of at least 3 steps by their definitions. So by either Theorem \ref{Thm: 3.2.1} or Theorem \ref{Thm: 3.2.2}, $I_\mathcal{Q}$ is prime.
\end{proof}

Combining our above results allows us to state the following theorem.

\begin{theorem} \label{Thm: closed path prime iff dilation prime}
    Let $\mathcal{P}$ be a closed path and let $\mathcal{Q} = \mathcal{P} \otimes \mathcal{R}_{m,n}$ be a dilated closed path for some $m,n \in \mathbb{Z}^+$. Then $I_\mathcal{P}$ is prime if and only if $I_\mathcal{Q}$ is prime.
\end{theorem}

\begin{proof}
    Assume $I_\mathcal{P}$ is prime. Then $I_\mathcal{Q}$ is prime by Corollary \ref{Cor: prime implies dilated prime closed path}. Now assume $I_\mathcal{Q}$ is prime. Then by Corollary 3.6 in \textbf{\cite{Mascia_Rinaldo_Romeo:2020}}, $\mathcal{Q}$ contains no zig-zag walks. So by Proposition \ref{Prop: Construct a zig-zag}, the polyomino $\mathcal{P}$ contains no zig-zag walks. Since $\mathcal{P}$ is a closed path, then the Zig-Zag Walk Conjecture holds for the class and by Theorem \ref{Thm: ZZWC for closed paths} (Theorem 6.2 in \textbf{\cite{Cisto_Navarra:2023}}), we have that $I_\mathcal{P}$ is prime. Thus, the result follows.
\end{proof}

So in fact, we have an equivalence between the primality of a closed path and the primality of its dilation. One interesting question to explore next would be to see if such an equivalence exists between other classes of polyominoes and their dilations.

\section*{Acknowledgements}
I want to thank Dr. Colin Ingalls for his insight and helpful suggestions on this work.

\bibliographystyle{plain}
\nocite{*}
\bibliography{bibliography}

\end{document}

%% file: preamble.tex
\documentclass[12pt, a4paper]{amsart}
 \pdfoutput=1
\usepackage{fullpage}
\usepackage{amscd}
\usepackage{amsmath}
\usepackage{amsthm}
\usepackage{amssymb}
\usepackage[hidelinks]{hyperref}

\usepackage[dvipsnames]{xcolor}
\usepackage{braket}
\usepackage{tikz}
\usepackage{polyomino}
\usepackage{dirtytalk}
\usepackage{float}

\makeatletter
    \tikzset{
    dot diameter/.store in=\dot@diameter,
    dot diameter=2pt,
    dot spacing/.store in=\dot@spacing,
    dot spacing=6pt,
    dots/.style={
        line width=\dot@diameter,
        line cap=round,
        dash pattern=on 0pt off \dot@spacing
    }
    }
\makeatother

\theoremstyle{definition}
\newtheorem{theorem}{Theorem}[section]
\newtheorem{definition}[theorem]{Definition}
\newtheorem{proposition}[theorem]{Proposition}
\newtheorem{lemma}[theorem]{Lemma}
\newtheorem{conjecture}[theorem]{Conjecture}
\newtheorem{corollary}[theorem]{Corollary}
\newtheorem{remark}[theorem]{Remark}

\DeclareMathOperator{\ext}{ext}
\DeclareMathOperator{\intpolyo}{int}

\title{Primality of dilated closed path polyominoes}
\author{Niall Larney}
\date{September 2026}